\documentclass[11pt]{article}
\usepackage[colorlinks,linkcolor=blue,citecolor=orange]{hyperref}

\usepackage[letterpaper]{geometry}
\usepackage{amsmath,amsthm,amsfonts,amssymb,booktabs}
\usepackage{enumerate,color,xcolor}
\usepackage{graphicx}
\usepackage{subfigure}
\usepackage{url}

\usepackage[colorlinks,linkcolor=blue,citecolor=orange]{hyperref}

\usepackage{wrapfig}

\numberwithin{equation}{section}
\theoremstyle{plain}

\newtheorem{theorem}{Theorem}

\newtheorem{corollary}[theorem]{Corollary}

\newtheorem{lemma}[theorem]{Lemma}

\newtheorem{proposition}[theorem]{Proposition}

\newtheorem{remark}[theorem]{Remark}
\newtheorem{assumption}[theorem]{Assumption}

\usepackage{comment}

\newcommand{\R}{\mathbb{R}}

\newcommand{\beq}{\begin{eqnarray}}
\newcommand{\eeq}{\end{eqnarray}}
\newcommand{\beqs}{\begin{eqnarray*}}
\newcommand{\eeqs}{\end{eqnarray*}}

\newcommand{\bigO}{\mathcal{O}}

\newcommand{\E}{\mathbb{E}}

\makeatletter
\def\munderbar#1{\underline{\sbox\tw@{$#1$}\dp\tw@\z@\box\tw@}}
\makeatother

\newcommand{\Tr}{\mathcal{T}_{\text{rec}}}
\newcommand{\Te}{\mathcal{T}_{\text{esc}}}

\title{Breaking Reversibility Accelerates Langevin Dynamics \\for Global Non-Convex Optimization}
\author{{Xuefeng
  Gao}\footnote{The authors are in alphabetical order.}\,\,\,\footnote{Department of Systems
    Engineering and Engineering Management, The Chinese University of Hong Kong, Shatin, N.T. Hong Kong;
    xfgao@se.cuhk.edu.hk.},
    {Mert G\"{u}rb\"{u}zbalaban}$^{\ast}$\,\footnote{Department of Management Science
and Information Systems and the DIMACS Institute, Rutgers University, Piscataway, NJ-08854, United States of America;
    mg1366@rutgers.edu.},
  Lingjiong Zhu$^{\ast}$\,\footnote{Department of Mathematics, Florida State University, 1017 Academic Way, Tallahassee, FL-32306, United States of America; zhu@math.fsu.edu.
  } }
\date{}

\begin{document}

\maketitle
% \begin{center}
% 	\today
% \end{center}
\begin{abstract} 
Langevin dynamics (LD) has been proven to be a powerful technique for optimizing a non-convex objective as an efficient algorithm to find local minima while eventually visiting a global minimum on longer time-scales. LD is based on the first-order Langevin diffusion which is reversible in time. We study two variants that are based on non-reversible Langevin diffusions: the underdamped Langevin dynamics (ULD) and the Langevin dynamics with a non-symmetric drift (NLD). 
Adopting the techniques of Tzen et al. (2018) for LD to non-reversible diffusions, we show that for a given local minimum that is within an arbitrary distance from the initialization, 
with high probability, either the ULD trajectory ends up somewhere outside a small neighborhood of this local minimum within a recurrence time which depends on the smallest eigenvalue of the Hessian at the local minimum or they enter this neighborhood by the recurrence time and stay there for a potentially exponentially long escape time. The ULD algorithm improves upon 
the recurrence time obtained for LD in Tzen et al. (2018) with respect to the dependency
on the smallest eigenvalue of the Hessian at the local minimum. 
Similar results and improvements are obtained for the NLD algorithm.
We also show that non-reversible variants can exit the basin of attraction of a local minimum faster in discrete time when the objective has two local minima separated by a saddle point and quantify the amount of improvement. Our analysis suggests that non-reversible Langevin algorithms are more efficient to locate a local minimum as well as exploring the state space.
\end{abstract}

\author{}

\section{Introduction}
%\vspace{-0.1in}
Consider the stochastic optimization problem:
\begin{equation*}
\min_{x\in\mathbb{R}^{d}}\overline{F}(x):=\mathbb{E}_{Z\sim P}[f(x,Z)]= \int_{\mathcal{Z}}  f(x,z)P(dz),
\end{equation*}
where $f:\mathbb{R}^{d}\times\mathcal{Z}\rightarrow\mathbb{R}$ 
is a real-valued, smooth, possibly non-convex objective function with two inputs, 
the decision vector $x\in\R^d$ and a random vector $Z$ with probability distribution $P$ defined on a set $\mathcal{Z}$. A standard approach 
for solving stochastic optimization problems is to approximate the expectation as an average over 
independent observations $z =(z_1,z_2,\dots,z_n) \in \mathcal{Z}^n$ and to solve:
\begin{equation}
\min_{x\in\mathbb{R}^{d}}F(x):= \frac{1}{n} \sum_{i=1}^n f(x,z_i).\label{pbm-finite-sum}
\end{equation}
Such problems with finite-sum structure arise in many applications, e.g. data analysis and machine learning (\cite{goodfellow2016deep}). In this work, our primary focus will be non-convex objectives.

%
%For example, in the context of stochastic learning problems, if $z = (a,b)$ is an input-output pair of data sampled from 
%an unknown underlying joint distribution, the function $f$ corresponds to the loss function $f(x,z)$ of using the decision variable $x$ and 
%input $a$ to predict $b$. For non-linear regression problems and for optimization problems arising 
%in training of neural networks where the decision variable $x$ corresponds to the choice of the parameters of the neural network we want to 
%fit to the dataset $z$ (see e.g. \cite{goodfellow2016deep}), the objective is typically non-convex with respect to its first argument. In this work, our primary focus will be non-convex objectives.

%$f$   In this work, our primary focus will be non-convex objectives.
%
%
%This formulation encompasses a number of regression and classification problems where $f$ 
%can be both convex and non-convex with respect to its first argument. For example, for linear and logistic regressions, 
%$f$ is convex whereas $f$ is typically non-convex for non-linear regression problems and for optimization problems arising 
%in training of neural networks where the decision variable $x$ corresponds to the choice of the parameters of the neural network we want to 
%fit to the dataset $z$ (see e.g. \cite{goodfellow2016deep}).  due to their empirical performance, their scalability and cheaper iteration and storage cost compared to second-order methods in high dimensions 

First-order methods such as gradient descent and stochastic gradient descent
and their variants with momentum have been popular for solving such optimization problems (see e.g. \cite{Bertsekas2015a,Bubeck2014}). These first-order methods admit some theoretical guarantees to locate a local minimizer, however their convergence depends strongly on the initialization and they do not have guarantees to visit a global minimum. % algorithm iterates do not typically visit global minima. 
%These methods are locally convergent for non-convex optimization problems in the sense that they come up with guarantees to converge to a local minimizer which depends on the initialization, however they are unable to locate a global minimum in general for non-convex objectives. 
The Langevin Dynamics (LD) is a variant of gradient descent where a properly scaled Gaussian noise is added to the gradients:
\begin{equation*}
X_{k+1}=X_{k}-\eta\nabla F(X_{k})+\sqrt{2\eta\beta^{-1}}\xi_{k},
\end{equation*}
where $\eta>0$ is the stepsize, $\xi_{k}$ is a $d$-dimensional {isotropic} Gaussian noise with distribution $\mathcal{N}(0,I)$ where for every $k$, the noise $\xi_{k}$ is independent of the (filtration) past up to time $k$ and $\beta>0$ is called the \emph{inverse temperature} parameter. With proper choice of parameters and under mild assumptions, LD algorithm converges to a stationary distribution that concentrates around a global minimum (see e.g. \cite{Borkar,gelfand1991recursive}) from an arbitrary initial point. Therefore, LD algorithm has a milder dependency on the initialization, visiting a global minimum eventually. %\footnote{In the worst case, the number of iterations required to converge to a global minimum can be exponential in the dimension for an objective with multiple minima, in particular finding a global minimum of a non-convex objective is hard in general. However, when the objective has further structure such as a growth condition, the dependency can also be polynomial in dimension.} %Compared to gradient descent or stochastic gradient descent, this algorithm has milder dependence to the initialization in the sense that it eventually concentrates around a global minimum for
The analysis of the convergence behavior of LD is often based on viewing LD as a discretization of the associated stochastic differential equation (SDE), known as the \emph{overdamped Langevin} diffusion or the \emph{first-order Langevin} diffusion, 
\begin{equation}\label{eqn:overdamped}
dX(t)=-\nabla F(X(t))dt+\sqrt{2\beta^{-1}}dB_{t},
\end{equation}
where $B_{t}$ is a $d$-dimensional standard Brownian motion (see e.g. \cite{gelfand1991recursive}). Under some mild assumptions on $F$, this SDE admits the following unique stationary distribution: 
\begin{equation}\label{overdamped:stationary}
\pi(dx) = \Gamma^{-1}e^{-\beta F(x)}dx,
\end{equation}
where $\Gamma>0$ is a normalizing constant.
%where $\Gamma>0$ is a normalizing constant.
%where $\Gamma>0$ is a normalizing constant. Without the noise term, overdamped Langevin SDE reduces to the gradient descent dynamics 
%   \beq\label{gradient-flow-ODE} x'(t) = -\nabla F(x(t)),
%   \eeq
%which is an ordinary differential equation (ODE)  that arises naturally in the study of LD (see e.g. \cite{gelfand1991recursive}, \cite[Sec. 4]{fang1997annealing}). This ODE is the continuum limit of the gradient descent algorithm as the stepsize goes to zero (see e.g.  \cite{scieur2017integration}).
%The generator $\mathcal{L}$ of a  Markov process $X(t)$ is formally defined as
% $$Lf(x) = (\partial_t)_{| t = 0} \mathbb{E} (f(X_t) | X_0 = x).$$
Note that overdamped Langevin diffusion is \emph{reversible} in the sense that if $X(0)$ is distributed according to the stationary measure $\pi$, then $(X_t)_{0\leq t \leq T}$ and $(X_{T-t})_{0\leq t\leq T}$ have the same law. 
It is known that the reversible Langevin algorithm converges to a local minimum in time polynomial 
with respect to parameters $\beta$ and $d$, 
the intuition being that the expectation of the iterates follows the gradient descent dynamics 
which converges to a local minimum (see e.g. \cite{zhang-sgld,fang1997annealing}). 
It is also known that once Langevin algorithms arrive to a neighborhood of a local optimum, 
they can spend exponentially many iterations in dimension to escape from the basin of attraction of this local minimum. This behavior is known as ``metastability" and has been studied well (see e.g. \cite{bovier2005metastability,Bovier2004,berglund2013kramers}).

Recently, \cite{pmlr-v75-tzen18a} provided a finer characterization of 
this metastability phenomenon. They showed that for a given local optimum $x_{\ast}$, with high probability and arbitrary initialization, either LD iterates arrive at a point outside an $\varepsilon$-neighborhood of this local minimum within a recurrence time $\Tr = \mathcal{O}\left(\frac{1}{m}\log(\frac{1}{\varepsilon})\right)$, where $m$ is smallest eigenvalue of the Hessian $\nabla^2 F(x_{\ast})$ at the local minimum or they enter this $\varepsilon$-neighborhood by the recurrence time and stay there until a potentially exponentially long escape time $\Te$. The escape time $\Te$ measures how quickly the LD algorithm can get away from a given neighborhood around a local minimum, therefore it can be viewed as a measure of the effectiveness of LD for the search of a global minimizer, whereas the recurrence time $\Tr$ can be viewed as the order of the time-scale for which search for local minima in the basin of attraction of that minimum happens.

One popular non-reversible variant of overdamped Langevin that can improve its performance in practice for a variety of applications (Section 4 in \cite{Lelievre-optdrift}) 
is based on the \emph{underdamped Langevin} diffusion, also known as the \emph{second-order Langevin} diffusion (\cite{kramers1940brownian}), 
\begin{align}
&dV(t)=-\gamma V(t)dt- \nabla F(X(t))dt+\sqrt{2\gamma \beta^{-1}}dB_{t},
\label{eq:VL}
\\
&dX(t)=V(t)dt,\label{eq:XL}
\end{align}
where $X(t), V(t) \in \mathbb{R}^d$, and $\gamma>0$ is known as the \emph{friction coefficient}.  % and $B_{t}$ is a standard $d$-dimensional Brownian motion
It is known that under mild assumption on $F$, the Markov process $(X, V)$ is ergodic and have a unique stationary distribution 
%This diffusion goes back to Kramers \cite{kramers1940brownian} and was derived in the physics literature to model particles moving in a potential subject to random noise.
\begin{equation}
\pi(dx, dv)= \Gamma_{U}^{-1} e^{-\beta(\frac{1}{2} \Vert v\Vert^2 + F(x))} dx dv,\label{eq-stat-dist}
\end{equation}
where $\Gamma_{U}>0$ is a normalizing constant. Hence, the marginal distribution in $X$ of the Gibbs distribution $\pi(dx, dv)$ is the same as the invariant distribution \eqref{overdamped:stationary} of the overdamped Langevin dynamics \eqref{eqn:overdamped}.  {We refer the readers to \cite{Bakry,Cheng,cheng-nonconvex,dalalyan2018kinetic,Eberle,Ma2019,mattingly2002ergodicity,Wu2001,Villani2009} for more on underdamped Langevin diffusions.} 
The Euler discretization of the underdamped Langevin diffusion
is known as the inertial Langevin dynamics or the Hamiltonian Langevin Monte Carlo algorithm {\cite{duane1987hybrid,neal2010mcmc}.}
\cite{Cheng} introduced a more accurate discretization of
underdamped Langevin diffusion,
where for any $k\in\mathbb{N}$ and any $k\eta<t\leq(k+1)\eta$,
\begin{equation*}
d\tilde{V}(t)=-\gamma\tilde{V}(t)dt-\nabla F(\tilde{X}(k\eta))dt+\sqrt{2\gamma\beta^{-1}}dW_{t},
\qquad
d\tilde{X}(t)=\tilde{V}(t)dt.
\end{equation*}
Note that when $t$ is between $k\eta$ and $(k+1)\eta$, the above diffusion process
is an Ornstein-Uhlenbeck process, which is a Gaussian process with explicit mean and covariance.
\cite{Cheng} showed that $(V_{k},X_{k})$ 
have the same distribution as $(\tilde{V}(k\eta),\tilde{X}(k\eta))$,
where the discrete iterates $(V_{k},X_{k})$,
called the \emph{underdamped Langevin dynamics} (ULD)\footnote{This algorithm is also known as the kinetic Langevin Monte Carlo algorithm.}
are generated as follows:
\begin{align}
&V_{k+1}=\psi_{0}(\eta)V_{k}-\psi_{1}(\eta)\nabla F(X_{k})+\sqrt{2\gamma\beta^{-1}}\xi_{k+1},
\label{eq:V-iterate}
\\
&X_{k+1}=X_{k}+\psi_{1}(\eta)V_{k}-\psi_{2}(\eta)\nabla F(X_{k})+\sqrt{2\gamma\beta^{-1}}\xi'_{k+1},
\label{eq:X-iterate}
\end{align}
where $(\xi_{k+1},\xi'_{k+1})$ is a $2d$-dimensional centered Gaussian vector
so that $(\xi_{j},\xi'_{j})$ are i.i.d. and independent of the initial condition,
and for any fixed $j$, the random vectors $((\xi_{j})_{1},(\xi'_{j})_{1})$, 
$((\xi_{j})_{2},(\xi'_{j})_{2})$, $\ldots$, $((\xi_{j})_{d},(\xi'_{j})_{d})$
are i.i.d. with the covariance matrix:
$$C(\eta)=\int_{0}^{\eta}[\psi_{0}(t),\psi_{1}(t)]^{T}[\psi_{0}(t),\psi_{1}(t)]dt,$$
where $\psi_{0}(t)=e^{-\gamma t}$ and $\psi_{k+1}(t)=\int_{0}^{t}\psi_{k}(s)ds$ for every $k\geq 0$.
Recent work \cite{GGZ} showed that ULD admits better non-asymptotic performance guarantees compared to known guarantees for LD in the context of non-convex optimization when the objective satisfies a dissipativity condition. Recent work also showed that ULD or alternative discretizations of the underdamped diffusion can sample from the Gibbs distribution more efficiently than LD when $F$ is globally strongly convex (see e.g. \cite{Cheng,dalalyan2018kinetic,Mangoubi-Smith17}) or strongly convex outside a compact set (see e.g. \cite{cheng-nonconvex}). 
%Indeed, one can compute that
%\begin{align*}
%&\psi_{0}(\eta)=e^{-\gamma\eta},
%\quad
%\psi_{1}(\eta)=\frac{1-e^{-\gamma\eta}}{\gamma},
%\\
%&\psi_{2}(\eta)=\frac{\eta}{\gamma}-\frac{1-e^{-\gamma\eta}}{\gamma^{2}},
%\end{align*}
%and $C(\eta)=(C_{ij}(\eta))_{1\leq i,j\leq 2}$ where
%\begin{align*}
%&C_{11}(\eta)=
%\frac{1}{2\gamma}(1-e^{-2\gamma\eta}), 
%\\
%&C_{12}(\eta)=C_{21}(\eta)=\frac{1}{2\gamma^{2}}(1-e^{-\gamma\eta})^{2},
%\\
%&C_{22}(\eta)=\frac{1}{2\gamma^{3}}\left(2\gamma\eta-3+4e^{-\gamma\eta}-e^{-2\gamma\eta}\right).
%\end{align*}
%\begin{align} 
%&V_{k+1}=V_{k}-\eta[\gamma V_{k}+\nabla F(X_{k})]+\sqrt{2\gamma \beta^{-1} \eta}\xi_{k},
% \label{eq:V-iterate}\\
%&X_{k+1}=X_{k}+\eta V_{k}, \label{eq:X-iterate}
%\end{align}
%where $(\xi_k)_{k\geq 0}$ is a sequence of i.i.d standard Gaussian random vectors in $\mathbb{R}^d$. 

The second non-reversible variant of overdamped Langevin involves adding a drift term: %to make the dynamics non-reversible:
\begin{equation}\label{eqn:nonreversible1}
dX(t)=-A_{J}(\nabla F(X(t)))dt+\sqrt{2\beta^{-1}}dB_{t}, 
\qquad
A_{J}:=I+J,
\end{equation} 
where $J\neq 0$ is a $d\times d$ anti-symmetric matrix, i.e. $J^{T}=-J$ and $I$ is the $d\times d$ identity matrix,
and $B_{t}$ is a standard $d$-dimensional Brownian motion.
It can be shown that such a drift preserves the stationary distribution \eqref{overdamped:stationary} 
(Gibbs distribution) of 
the overdamped Langevin dynamics, and it can lead to a faster convergence to the stationary distribution than the reversible Langevin diffusion (the case $J=0$), see e.g. \cite{HHS93, HHS05,Lelievre-optdrift,pavliotis2014stochastic,guillin2016} for details. 
{Algorithms based on \eqref{eqn:nonreversible1} have been applied to sampling, see e.g. \cite{Futami,rey2016improving,reyGraphs,DLP2016,DPZ17},
and non-convex optimization, see e.g. \cite{HWGGZ20}.}
The {Euler} discretization of \eqref{eqn:nonreversible1} leads to 
\begin{equation}\label{eqn:nonreversible:discrete}
X_{k+1}=X_{k}-\eta A_{J}(\nabla F(X_{k}))+\sqrt{2\eta\beta^{-1}}\xi_{k},
\end{equation}
which we refer to as the \emph{non-reversible Langevin dynamics} (NLD).

\textbf{Contributions.} 
%{\color{red}Second}, we give a refined analysis of non-reversible Langevin dynamics around a local minimum. Our analysis is based on
%linearizing the gradient of the objective around a local minimum and builds on our results for quadratic functions. 
% 
%Our results show that iterates for the non-reversible dynamics can both escape a local minima 
%and exit from the basin of the attraction of a local minimum faster. 
%Our results quantify the improvement that can be obtained in performance 
%by breaking the reversibility in the Langevin dynamics and fill a gap 
%between the theory and practice of non-reversible Langevin algorithms
In this paper, we investigate the metastability behavior of non-reversible Langevin algorithms for non-convex objectives.  We extend the results of \cite{pmlr-v75-tzen18a}
to non-reversible Langevin dynamics and show that for a given local minimum 
that is within an arbitrary distance $r$ from the initialization, with high probability, either ULD trajectory ends up somewhere outside an $\varepsilon$-neighborhood of this local minimum within a recurrence time $\mathcal{T}_{\text{rec}}^{U} = \bigO\left(\frac{|\log(m)|}{\sqrt{m}}\log(r/\varepsilon)\right)$ or they enter this neighborhood by the recurrence time and stay there for a  potentially exponentially long escape time. The analogous result shown in \cite{pmlr-v75-tzen18a} for reversible LD requires a recurrence time of $\mathcal{T}_{\text{rec}} = \bigO\left(\frac{1}{m}\log(r/\varepsilon)\right)$. This shows that underdamped dynamics requires a smaller recurrence time by a square root factor in $m$ (ignoring a $\log(m)$ factor). The difference is significant as the smallest eigenvalue $m$ of the Hessian matrix at a local optimum can be very small in a number of applications, including deep learning (see e.g. \cite{Chaudhari,sagun2016eigenvalues}). Since the recurrence time can be viewed as a measure of the efficiency of the search of a local minimum \cite{pmlr-v75-tzen18a}, our results suggest that ULD operates on a faster time-scale to locate a local minimum. Similar results are obtained for NLD. In order to obtain the results, we first give a refined characterization of the dynamics around a local minimum by linearizing the gradients. The analysis here is more complicated compared to the LD case
in \cite{pmlr-v75-tzen18a} due to non-reversibility,
and requires us to develop new estimates, e.g. Lemma~\ref{lem:keylemma}, 
where the eigenvalue and the norm estimates require a significant amount of work
because the forward iterations correspond to non-symmetric matrices $H_\gamma$ {(defined in \eqref{def-Hgamma0})} and achieving the
acceleration behavior requires careful estimates.
The analysis here also requires us to establish novel uniform $L^2$ bounds for NLD in both continuous and discrete times. 

In addition, we consider the mean exit time from the basin of attraction of a local minimum for non-reversible algorithms. We focus on the double-well example which has been the recent focus of the literature \cite{BR2016, Landim2017} as it is the simplest non-convex function that gives intuition about the more general case and for which mean exit time has been studied in continuous time.
Our analysis shows that non-reversible dynamics can exit the basin of attraction of a local minimum faster under some conditions and characterizes the improvement for both ULD and NLD compared to LD when the parameters of these algorithms are chosen appropriately. These results support the numerical evidence that non-reversible algorithms can explore the state space more efficiently  \cite{carin-2015-langevin-integrators,emilyfox-sghmc,guillin2016} and bridges a gap between the theory and practice of Langevin algorithms.

\textbf{Other related work.} Langevin dynamics has been studied under simulated annealing algorithms in the optimization, physics and statistics literature and its asymptotic convergence guarantees are well known (see e.g. \cite{gidas1985nonstationary, hajek1985tutorial,gelfand1991recursive,kirkpatrick1983optimization,bertsimas1993simulated,Belloni,Borkar}). However, finite-time performance guarantees for LD have not been studied until more recently {(see e.g. \cite{Dalalyan,DM2017}).} 
 {Non-asymptotic performance guarantees for stochastic gradient versions have also been studied.}
 See also e.g. \cite{Raginsky, zhang-sgld,CDT2020} for related results. \cite{xu2018global} shows that it suffices to have $\mathcal{O}(nd/(\lambda \varepsilon))$ gradient evaluations or $\mathcal{O}(d^7/(\lambda^5 \varepsilon^5))$ stochastic gradient evaluations to compute an almost minimizer where $\lambda$ is a spectral gap parameter that is exponentially small in the dimension $d$ and $\varepsilon$ is the target accuracy. These results improve upon the existing results from the seminal work of \cite{Raginsky}. \cite{erdogdu18} also considered Euler discretization of general dissipative diffusions in the non-convex setting, proved a $1/\varepsilon^2$ convergence rate, showing that different diffusions are suitable for minimizing different convex/non-convex objective functions $f$. Their expected suboptimality bound also generalizes the results in \cite{Raginsky}. See also \cite{nguyen19} for non-asymptotic guarantees for non-convex optimization using L\'evy-driven Langevin dynamics {proposed in \cite{FLMC}}.

\textbf{Notations.}
Throughout the paper, for any $x,y\in\mathbb{R}$, we use
the notation $x\wedge y$ to denote $\min\{x,y\}$ 
and $x\vee y$ to denote $\max\{x,y\}$.
For any $n\times n$ matrix $A$, we use
$\lambda_{i}(A)$, $1\leq i\leq n$,
to denote the $n$ eigenvalues of $A$.
We also assume that $H$ is the Hessian matrix $\nabla^{2}F$
evaluated at the local minimum $x_{\ast}$, and is positive definite. The norm $\|\cdot\|$ denotes the 2-norm of a vector and the (spectral norm) 2-norm of a matrix.
%In our analysis, we will use the following 2-Wasserstein distance. For any two Borel probability measures $\nu_1$ and $\nu_2$ with finite second moments, the 2-Wasserstein distance is defined as 
%$\mathcal{W}_2(\nu_1, \nu_2) := \inf_{Y_1 \sim \nu_1, Y_2 \sim \nu_2} \left( \E \| Y_1 - Y_2 \|^2 \right)^{1/2}$, 
%where the infimum is taken over all the random couples $(Y_1, Y_2)$ taking values in $\R^d \times \R^d$ with marginals $\nu_1$ and $\nu_2$. %We refer the reader to \cite{villani2008optimal} for more on Wasserstein distances.

%%%%%%%%%%%%%%%%%%%%%%%%%%%%%%%%%%%%%%%%%%%%%%%%%%%%%%%%
%\vspace{-0.1in}

%%%%%%%%%%%%%%%%%%%%%%%%%

%%%%%%%%%%%%%%%%%%%%%%%%%%%%%%%%%%%%%%%%%%%

%%%%%%%%%%%%%%%%%%%%%%%%%
\vspace{-0.1in}
\section{Main results}\label{sec:rec:escape}
\vspace{-0.05in}

%In Section \ref{sec:quadratic}, we observed that both underdamped Langevin and reversible Langevin concentrates around 
%the local (global) minimum faster in continuous-time for the quadratic objectives. 
%When the stepsize is small, we also expect similar results to hold in discrete time. 

In this section, we will study the recurrence time %and escape time 
$\mathcal{T}_{\text{rec}}^U$ 
%and $\mathcal{T}_{\text{esc}}^U$ 
of underdamped Langevin dynamics (ULD), 
and the corresponding time-scale $\mathcal{T}_{\text{rec}}^J$ %and $\mathcal{T}_{\text{esc}}^J$ 
for non-reversible Langevin dynamics (NLD). 
We will show that recurrence time of underdamped and non-reversible Langevin algorithms will improve upon that of reversible Langevin 
algorithms in terms of its dependency to the smallest eigenvalue of the Hessian at a local minimum. For non-convex optimizations, our results suggest that for ULD and NLD, searching for a local minimum happens on a faster time scale compared to the reversible LD. 
 %Our analysis is 
%based on linearizing the gradient of the objective around a local minimum and is based on the results 
%derived in the previous section on quadratic objectives.
%We will also show in Section \ref{sec:global} that for the double-well potential, 
%the mean exit times from \emph{the basin of attraction of a local minimum}\footnote{Formally, we define the basin of attraction of a local minimum $x_*$ as the set of all initial points $x_0\in\R^d$ such that the \emph{gradient flow ODE} $x'(t) = -\nabla F(x(t))$ with initial point $x_0$ converges to $x_*$ as time $t\to\infty$.} for non-reversible Langevin dynamics will improve upon that of reversible Langevin dynamics 
%in terms of dependency to the curvature at the saddle point.      
For the rest of the paper, we impose the following assumptions.
\begin{assumption}\label{assumptions}
We impose the following assumptions.  
\begin{itemize}
\item [$(i)$]
The functions $f(\cdot,z)$ are twice continuously differentiable, non-negative valued, and
$|f(0,z)|\leq A$,
$\left\Vert\nabla f(0,z)\right\Vert\leq B$,
and $\left\Vert\nabla^{2}f(0,z)\right\Vert\leq C$,
uniformly in $z\in\mathcal{Z}$ for some $A,B,C>0$.
\item [$(ii)$]
$f(\cdot,z)$ have Lipschitz-continuous gradients and Hessians,
uniformly in $z\in\mathcal{Z}$, there exist constants $L,M>0$
so that for all $x,y\in\mathbb{R}^{d}$,
\begin{equation}
\left\Vert\nabla f(x,z)-\nabla f(y,z)\right\Vert
\leq M\Vert x-y\Vert,\label{eq:hessian-L}
\qquad
\left\Vert\nabla^{2}f(x,z)-\nabla^{2}f(y,z)\right\Vert\leq L\Vert x-y\Vert. 
\end{equation}
\item [$(iii)$]
The empirical risk $F(\cdot)$ is $(m,b)$-dissipative:
%\footnote{In terms of notations, the dissipativity constant $m$ here is taken to be the same
%as the minimum eigenvalue of the Hessian matrix $H$ in this section, as well
%as for the quadratic case in Section \ref{sec:quadratic}. 
%With abuse of notations, we use $m$ both for non-convex and convex cases.}:
$\langle x,\nabla F(x)\rangle
\geq m\Vert x\Vert^{2}-b$.
\item [$(iv)$] The initialization satisfies $\Vert X_{0}\Vert\leq R:=\sqrt{b/m}$.
\end{itemize}
\end{assumption}
%\vspace{-0.15in}

%The first assumption and the second assumption on gradient Lipschitzness is standard (\cite{mattingly2002ergodicity, pmlr-v75-tzen18a,Eberle,Cheng}). The second assumption on the Lipschitzness of the Hessian is also frequently made in the literature (\cite{pmlr-v75-tzen18a,chatterji2018theory,dalalyan2018kinetic}), this allows a more accurate local approximation of the Langevin dynamics as an Ornstein-Uhlenbeck process. The third assumption on dissipativity is also standard in the literature to ensure convergence of Langevin
%diffusions to the stationary measure. It is not hard to see from dissipativity condition in Assumption \ref{assumptions} (iii)
%above that for any local minimum $x_{\ast}$ of the Hessian matrix $H=\nabla^{2}F$,
%which is positive definite and the minimum eigenvalue of $H$ is $m$, then
%we have $\Vert x_{\ast}\Vert\leq R=\sqrt{b/m}$. This indicates that we can choose an initialization $X_0$ to satisfy $\Vert X_{0}\Vert\leq R=\sqrt{b/m}$ as in part (iv) of Assumption \ref{assumptions}.

For the Hessian $H$ at a fixed local minimum of $F$ defined in \eqref{pbm-finite-sum}, it is a $d\times d$ symmetric positive definite matrix with eigenvalues $\{\lambda_i\}_{i=1}^d$ in increasing order, i.e. 
{\begin{equation}\label{defn:m:M}
m=\lambda_1 \leq \lambda_2 \leq \dots \leq \lambda_d=M.
\end{equation}}

%%%%%%%%%%%%%%%%%%%%%%%%%%%%%%%%%%%%%%%%%%%
\vspace{-0.15in}
\subsection{Underdamped Langevin dynamics}\label{sec:underdamped}
%\vspace{-0.1in}
%\textbf{Underdamped Langevin dynamics.}
Recall the underdamped Langevin \eqref{eq:V-iterate}-\eqref{eq:X-iterate}.
Define
\begin{eqnarray} \label{def-Hgamma0}
H_{\gamma}:=
\left[
\begin{array}{cc}
\gamma I & H
\\
-I & 0
\end{array}
\right],
\end{eqnarray}
{where we recall that $H$ is the Hessian matrix $\nabla^{2}F$ evaluated at the local minimum $x_{\ast}$.}
In the first lemma, we provide an estimate on $\| e^{-tH_\gamma}\|$. This is the key result that allows the underdamped dynamics to achieve faster rate compared with overdamped dynamics. %, in particular achieving the fastest mixing rate $e^{-\sqrt{m}t}$ requires tuning the parameter $\gamma$. 
\begin{lemma}[Estimate on $\| e^{-tH_\gamma}\|$]\label{lem:keylemma}
(i) If $\gamma \in (0,2\sqrt{m}$), then
$\left\| e^{-t H_\gamma}\right\| \leq C_{\hat{\varepsilon}} e^{-\sqrt{m}(1-\hat{\varepsilon})t}$,
where $C_{\hat{\varepsilon}}:=\frac{1+M}{\sqrt{m(1-(1-\hat{\varepsilon})^{2})}}$,
and $\hat{\varepsilon}:=1-\frac{\gamma}{2\sqrt{m}} \in (0,1)$, 
{where $m,M$ are defined in \eqref{defn:m:M}}. 
(ii) If $\gamma=2\sqrt{m}$, then we have
$\left\| e^{-t H_\gamma}\right\| \leq \sqrt{C_{H}+2+ (m+1)^2 t^2 }\cdot e^{-\sqrt{m}t}$,
where $C_H := \max_{i: \lambda_i > m}\frac{(1+\lambda_i)^{2}}{\lambda_i-m}$.  
%\textcolor{red}{(iii) add the convergence rate with $\gamma> 2 \sqrt{m}$}
\end{lemma}

We investigate the behavior around local minima for 
the underdamped Langevin dynamics \eqref{eq:V-iterate}-\eqref{eq:X-iterate}
by studying recurrence and escape times
with the choice of the friction coefficient $\gamma=2\sqrt{m}$
which is optimal for $\Vert e^{-tH_{\gamma}}\Vert$. % as discussed in Section \ref{sec:quadratic}.
 %Our results in this section is therefore stated with $\gamma=2\sqrt{m}$ which gives the best guarantees for the underdamped dynamics. However, 
%in practice, the constant $m$ is typically not known exactly, but can only be estimated. Hence, a natural question that arises is what happens if $m$ is underestimated or overestimated; or equivalently what happens if  $\gamma<2\sqrt{m}$ or $\gamma>2\sqrt{m}$. For the sake of completeness, 
%We also give an analogous theorem for the case  $\gamma<2\sqrt{m}$ in Section \ref{sec:small} of the Appendix; the remaining case $\gamma>2\sqrt{m}$ can also be treated similarly. 
We first recall the Lambert W function $W(x)$ {which} is defined via
the solution of the algebraic equation $W(x)e^{W(x)}=x$.
When $-e^{-1}\leq x<0$, $W(x)$ has two branches, the upper branch $W_{0}(x)$ and
the lower branch $W_{-1}(x)$, see e.g. \cite{Corless}.

%$C_H := \max_{i: \lambda_i > m}\frac{(1+\lambda_i)^{2}}{\lambda_i-m}$ is a constant that depends only on the eigenvalues of $H$.
%
%We will summarize the technical constants that will be used
%in Table \ref{table_constants_1} in the Appendix.
%In the following result, the notation $\mathcal{O(\cdot)}$
%hides dependency to the constants $C_{H}$, $M$, $L$. 
%We give explicit expressions for all the constants in the proofs.
%The constant $m>0$ can be arbitrarily small.

%%%%%%%%%%%%%%%%%%%%%%%%%%%%%%%%%%%%%%%%

\begin{theorem}\label{MainThm1}
Fix $\gamma=2\sqrt{m}$, $\delta\in(0,1)$ and $r>0$. For a given $\varepsilon$ satisfying 
$
0<\varepsilon<\overline{\varepsilon}^{U}=
\min\left\{\mathcal{O}(r),\mathcal{O}(m)\right\},
$ we define the recurrence time 
%\footnote{The recurrence time $\mathcal{T}_{\text{rec}}^{U}$ defined in \eqref{rec:time:U} 
%using the Lambert W function is indeed
%the unique solution greater than $\frac{1}{\sqrt{m}}$ of the algebraic equation
%\footnote{We used the fact that the solution
%to the algebraic equation $ax=p^{x}$, $p>0$, $p\neq 1$, $a\neq 0$, can be
%expressed as $x=-\frac{1}{\log p}W(-\frac{1}{a}\log p)$, 
%where $W$ is the Lambert W function.}:
%$\mathcal{T}_{\text{rec}}^{U}e^{-\sqrt{m}\mathcal{T}_{\text{rec}}^{U}}
%=\frac{\varepsilon^{2}}{8r^{2}\sqrt{C_{H}+2+(m+1)^{2}}}$.
%Since $\varepsilon<\overline{\varepsilon}^{U}$
%and by the definition of $\overline{\varepsilon}^{U}$
%in Table \ref{table_constants_1}, we have
%$\varepsilon<\overline{\varepsilon}^{U}\leq
%2\sqrt{2}(C_{H}+2+(m+1)^{2})^{1/4}\frac{e^{-1/2}r}{m^{1/4}}$
%so that $\frac{\varepsilon^{2}}{8r^{2}\sqrt{C_{H}+2+(m+1)^{2}}}<\frac{1}{\sqrt{m}}e^{-1}$ 
%and thus $\mathcal{T}_{\text{rec}}^{U}\geq\frac{1}{\sqrt{m}}$ is well-defined
%and moreover for every $x\geq\mathcal{T}_{\text{rec}}^{U}$, 
%the function $x\mapsto xe^{-\sqrt{m}x}$ is decreasing in $x$.}
\begin{equation}\label{rec:time:U}
\mathcal{T}_{\text{rec}}^{U}
=-\frac{1}{\sqrt{m}}
W_{-1}\left(\frac{-\varepsilon^{2}\sqrt{m}}{8r^{2}\sqrt{C_{H}+2+(m+1)^{2}}}\right)
=\mathcal{O}\left(\frac{|\log(m)|}{\sqrt{m}}\log\left(\frac{r}{\varepsilon}\right)\right),
\nonumber
\end{equation}
and the escape time $\mathcal{T}_{\text{esc}}^{U}:=\mathcal{T}_{\text{rec}}^{U}+\mathcal{T}$,
for any arbitrary $\mathcal{T}>0$. Consider an arbitrary initial point $x$ for the underdamped Langevin dynamics and a local minimum $x_{\ast}$ at a distance at most $r$. 
Assume that the stepsize $\eta$ satisfies
\begin{align*}
\eta\leq\overline{\eta}^{U}&=
\min\Bigg\{\mathcal{O}(\varepsilon),
\mathcal{O}\left(\frac{m^{2}\beta\delta\varepsilon^{2}}
{(md+\beta)\mathcal{T}_{\text{rec}}^{U}}\right),
\mathcal{O}\left(\frac{m^{5/4}\delta}{(d+\beta)^{1/2}(\mathcal{T}_{\text{esc}}^{U})^{1/2}}\right),
\mathcal{O}(m^{5/2})\Bigg\}
\,,
\end{align*}
and $\beta$ satisfies
\begin{align*}
\beta\geq
\underline{\beta}^{U}
&=\max\Bigg\{\Omega\left(\frac{d+\log((\mathcal{T}+1)/\delta)}{m\varepsilon^{2}}\right),
\Omega\left(\frac{d\eta m^{1/2}\log(\delta^{-1}\mathcal{T}_{\text{rec}}^{U}/\eta)}{\varepsilon^{2}}\right)\Bigg\},
\end{align*}  
for any realization of training data $z$, with probability at least $1-\delta$
with respect to the Gaussian noise, at least one of the following events will occur:
(1) $\Vert X_{k}-x_{\ast}\Vert\geq\frac{1}{2}\left(\varepsilon+re^{-\sqrt{m}k\eta}\right)$
for some $k\leq\eta^{-1}\mathcal{T}_{\text{rec}}^{U}$;
(2) $\Vert X_{k}-x_{\ast}\Vert\leq\varepsilon+re^{-\sqrt{m}k\eta}$
for every $\eta^{-1}\mathcal{T}_{\text{rec}}^{U}\leq k\leq\eta^{-1}\mathcal{T}_{\text{esc}}^{U}$.
\end{theorem}

The expressions of technical constants in the statement of Theorem~\ref{MainThm1}, including $\overline{\varepsilon}^{U}, \overline{\eta}^{U}$ and $\underline{\beta}^{U}$,
can be found in the proof of Theorem~\ref{MainThm1} in the Supplementary File.
%are summarized in Table \ref{table_constants_1} and Table \ref{table_constants_1_2} in the Appendix. 
%Theorem \ref{MainThm1} is about the empirical risk but the ideas can be generalized to population risk problems. See Section \ref{sec:population} of the Appendix.

%\begin{remark}
%The expressions of technical constants in the statement of Theorem~\ref{MainThm1}, including $\overline{\varepsilon}^{U}, \overline{\eta}^{U}$ and $\underline{\beta}^{U}$, are summarized in Table \ref{table_constants_1} in the Appendix. 
%\footnote{Notice that in Theorem \ref{MainThm1}, the definition of $\eta$ and $\beta$ are coupled
%since $\overline{\eta}^{U}$ depends on $\beta$ and $\underline{\beta}^{U}$ depends on $\eta$.
%A closer look reveals that when $\eta$ is sufficiently small, 
%the first term in the definition of $\underline{\beta}^{U}$ dominates the second term
%and $\underline{\beta}^{U}$ is independent of $\eta$. So to satisfy
%the constraints in Theorem \ref{MainThm1}, it suffices to first choose
%$\beta$ to be larger than the first term in $\underline{\beta}^{U}$ and then choose
%$\eta$ to be sufficiently small.}
%\end{remark}

%\begin{remark} 
%In \cite{pmlr-v75-tzen18a}, they used $\Vert X_{k}-x_{\ast}\Vert_{H}$.
%In Theorem \ref{MainThm1}, we used $\Vert X_{k}-x_{\ast}\Vert$ instead. 
%Note that the relation $\Vert\cdot\Vert_{H}\geq\sqrt{m}\Vert\cdot\Vert$ holds.
%\end{remark}
%\begin{remark} 
%Theorem \ref{MainThm1} is about the empirical risk but the ideas can be generalized to population risk problems. See Section \ref{sec:population} of the Appendix.
%\end{remark}
In many applications, the eigenvalues of the Hessian at local extrema often concentrate around zero and the magnitude of the smallest eigenvalue $m$ of the Hessian can be very small (see e.g. \cite{Chaudhari,sagun2016eigenvalues}).
In \cite{pmlr-v75-tzen18a}, the overdamped Langevin algorithm is analyzed
and the recurrence time 
$\mathcal{T}_{\text{rec}}=\mathcal{O}\left(\frac{1}{m}\log(\frac{r}{\varepsilon})\right)$,
while our recurrence time 
$\mathcal{T}_{\text{rec}}^{U}=\mathcal{O}\left(\frac{|\log(m)|}{\sqrt{m}}\log(\frac{r}{\varepsilon})\right)$
for the underdamped Langevin algorithm, 
which has a square root factor improvement. Since the recurrence time can be viewed as a measure of the efficiency of the search of a local minimum, our result suggests that ULD require smaller recurrence time, so they operate on a faster time scale to locate a local minimum.

%Moreover, in Theorem \ref{MainThm1}, we used $\varepsilon+re^{-\sqrt{m}k\eta}$
%which is an improvement from $\varepsilon+re^{-mk\eta}$ used in \cite{pmlr-v75-tzen18a}.

%%%%%%%%%%%%%%%%%%%%%%%%%%%%%%%
%\vspace{-0.15in}
\subsection{Non-reversible Langevin dynamics}\label{sec:anti}
%\vspace{-0.1in}
%\textbf{Non-reversible Langevin dynamics.}
We investigate the behavior around local minima for 
the non-reversible Langevin dynamics \eqref{eqn:nonreversible:discrete}
by studying recurrence and escape times. 
One can expect that the convergence behavior of the non-reversible Langevin diffusion is controlled by the decay of $\left\| e^{-tA_{J}H}\right\|$ in time $t$, which is related to the real part of the eigenvalues
    $\lambda^J_{i}:=\text{Re}\left(\lambda_{i}(A_{J}H)\right)$
indexed with increasing order and their multiplicity. It has been {shown} that
for any anti-symmetric matrix $J$, we have $ m = \lambda_1 \leq \lambda_1^J \leq \lambda_d^J \leq \lambda_d = M, $
and $m = \lambda_1 = \lambda_1^J$ if a very special condition holds. See Theorem 3.3. in \cite{HHS93} for details. This suggests generically non-reversible Langevin leads to a faster exponential decay compared to reversible Langevin, i.e $\lambda_1^J > \lambda_1$.
In addition, we have the {following} estimates: there exists a positive constant $C_J$ that depends on $J$ such that 
\begin{equation} \label{eq:norm_J}
\left\| e^{-tA_J H}\right\| \leq C_J (1 + t^{n_1-1}) e^{-t \lambda_1^J},
\end{equation}
where $n_1$ is the maximal size of a Jordan block of $A_J H$ corresponding to the eigenvalue $\lambda_1^J$. It follows that 
for any $\tilde{\varepsilon}>0$, there exist some constant $C_{J}(\tilde{\varepsilon})$ that depends on $\tilde{\varepsilon}$ and $J$ such that for every $t\geq 0$,
\begin{equation}\label{CJ:mJ1}
\Vert e^{-tA_{J}H}\Vert\leq C_{J}(\tilde{\varepsilon})e^{-tm_{J}(\tilde{\varepsilon})},
\qquad m_{J}(\tilde{\varepsilon}):=\lambda_{1}^{J}-\tilde{\varepsilon}.
\end{equation}
%We recall from \eqref{CJ:mJ} in Lemma \ref{lem:key:anti} that for any $\tilde{\varepsilon}>0$,
%there exists some constant $C_{J}(\tilde{\varepsilon})$ that may depend on $J$ and $\tilde{\varepsilon}$
%such that for every $t\geq 0$:
%$\Vert e^{-tA_{J}H}\Vert\leq C_{J}(\tilde{\varepsilon})e^{-tm_{J}(\tilde{\varepsilon})}$,
%where $m_{J}(\tilde{\varepsilon})=(\lambda_{1}^{J}-\tilde{\varepsilon})$.
%Note that by considering $t=0$, it is clear $C_{J}(\tilde{\varepsilon})\geq 1$.
%We will summarize the technical constants that will be used in Table \ref{table_constants_2}
%in the Appendix.

%%%%%%%%%%%%%%%%%%%%%%%%%%%%%%%%%%%%%%%
Now we state the main result on the metastability of non-reversible Langevin dynamics \eqref{eqn:nonreversible:discrete}.
\begin{theorem}\label{MainThm1:Q} 
Fix $\delta\in(0,1)$ and $r>0$. For a given $\varepsilon$ satisfying
%\begin{equation*}
$0<\varepsilon
<\overline{\varepsilon}^{J}=\min\left\{\mathcal{O}\left(\frac{m_{J}(\tilde{\varepsilon})}{C_{J}(\tilde{\varepsilon})}\right),\mathcal{O}(rC_{J}(\tilde{\varepsilon}))\right\}$,
%\end{equation*}
we define the recurrence time
\begin{align}\label{rec:time:J}
\mathcal{T}_{\text{rec}}^{J}
&:=\frac{2}{m_{J}(\tilde{\varepsilon})}\log\left(\frac{8r}{C_{J}(\tilde{\varepsilon})\varepsilon}\right)
=\mathcal{O}\left(\frac{1}{m_{J}(\tilde{\varepsilon})}\log\left(\frac{r}{C_{J}(\tilde{\varepsilon})\varepsilon}\right)\right),
\nonumber
\end{align}
and the escape time
$\mathcal{T}_{\text{esc}}^{J}:=\mathcal{T}_{\text{rec}}^{J}+\mathcal{T}$ 
for any arbitrary $\mathcal{T}>0$.
For any initial point $x$ 
and a local minimum $x_{\ast}$ at a distance at most $r$.
Assume the stepsize
% $\eta$ satisfies
\begin{align*}
\eta\leq
\overline{\eta}^{J}
&=\min\Bigg\{\mathcal{O}\left(\varepsilon\right),
\mathcal{O}\left(\frac{\delta\varepsilon^{2}m^{3}}{(m+\beta^{-1}d)\mathcal{T}_{\text{rec}}^{J}}\right),
\mathcal{O}\left(\frac{\delta^{2}m^{3}}{(d+m\beta+dm^{3})\mathcal{T}_{\text{esc}}^{J}}\right)
\Bigg\}\,,
\end{align*}
and $\beta$ satisfies
%\footnote{Notice that in Theorem \ref{MainThm1:Q}, the definition of $\eta$ and $\beta$ are coupled
%since $\overline{\eta}^{J}$ depends on $\beta$ and $\underline{\beta}^{J}$ depends on $\eta$.
%A closer look reveals that when $\eta$ is sufficiently small, 
%the first term in the definition of $\underline{\beta}^{J}$ dominates the second term
%and $\underline{\beta}^{J}$ is independent of $\eta$. So to satisfy
%the constraints in Theorem \ref{MainThm1:Q}, it suffices to first choose
%$\beta$ to be larger than the first term in $\underline{\beta}^{J}$ and then choose
%$\eta$ to be sufficiently small.}
\begin{align*}
\beta\geq
\underline{\beta}^{J}
&=
\max
\Bigg\{\Omega\left(\frac{C_{J}(\tilde{\varepsilon})^{2}}{m_{J}(\tilde{\varepsilon})\varepsilon^{2}}\left(d
+\log\left(\frac{\mathcal{T}+1}{\delta}\right)\right)\right),
\Omega\left(\frac{d\eta\log(\delta^{-1}\mathcal{T}_{\text{rec}}^{J}/\eta)}
{\varepsilon^{2}}\right)
\Bigg\},
\end{align*}
for any realization of training data $z$, with probability at least $1-\delta$
with respect to the Gaussian noise, at least one of the following events will occur:
(1) $\Vert X_{k}-x_{\ast}\Vert\geq\frac{1}{2}\left(\varepsilon+re^{-m_{J}(\tilde{\varepsilon}) k\eta}\right)$
for some $k\leq\eta^{-1}\mathcal{T}_{\text{rec}}^{J}$;
(2) $\Vert X_{k}-x_{\ast}\Vert\leq\varepsilon+re^{-m_{J}(\tilde{\varepsilon}) k\eta}$
for every $\eta^{-1}\mathcal{T}_{\text{rec}}^{J}\leq k\leq\eta^{-1}\mathcal{T}_{\text{esc}}^{J}$.
%where $C_{J}(\tilde{\varepsilon})$ and $m_{J}(\tilde{\varepsilon})$ are given by \eqref{CJ:mJ}.
\end{theorem}

The expressions of technical constants in the statement of Theorem~\ref{MainThm1:Q}, 
including $\overline{\varepsilon}^{J}, \overline{\eta}^{J}$ and $\underline{\beta}^{J}$, 
can be found in the proof of Theorem~\ref{MainThm1:Q} in the Supplementary File.
%are summarized in Table \ref{table_constants_2} in the Appendix. 
%Theorem \ref{MainThm1:Q} is about the empirical risk but the ideas can be generalized to population risk problems. See
%Section \ref{sec:population} of the Appendix.

%\begin{remark} 
%In \cite{pmlr-v75-tzen18a}, they used $\Vert X_{k}-x_{\ast}\Vert_{H}$.
%In Theorem \ref{MainThm1:Q}, we used $\Vert X_{k}-x_{\ast}\Vert$ instead. 
%Note that the relation $\Vert\cdot\Vert_{H}\geq\sqrt{m}\Vert\cdot\Vert$ holds.
%\end{remark}
%\begin{remark} Theorem \ref{MainThm1:Q} is about the empirical risk but the ideas can be generalized to population risk problems. See
%Section \ref{sec:population} of the Appendix.
%\end{remark}

In \cite{pmlr-v75-tzen18a}, the overdamped Langevin algorithm is used
and the recurrence time 
$\mathcal{T}_{\text{rec}}=\mathcal{O}\left(\frac{1}{m}\log(\frac{r}{\varepsilon})\right)$,
while our recurrence time 
$\mathcal{T}_{\text{rec}}^{J}=\mathcal{O}\left(\frac{1}{m_{J}(\tilde{\varepsilon})}
\log(\frac{r}{C_{J}(\tilde{\varepsilon})\varepsilon})\right)$
for the non-reversible Langevin algorithm, {and  
$\mathcal{T}_{\text{rec}}^{J}=\mathcal{O}\left(\frac{1}{m_{J}(\tilde{\varepsilon})}
\log(\frac{r}{\varepsilon})\right)$ when $C_{J}(\tilde{\varepsilon})=\mathcal{O}(1)$, which has the improvement over the overdamped
Langevin algorithm since $m_{J}(\tilde{\varepsilon})>m$ in general. }
%Moreover, in Theorem \ref{MainThm1:Q}, we used $\varepsilon+re^{-m_{J}(\tilde{\varepsilon})k\eta}$
%which is an improvement from $\varepsilon+re^{-mk\eta}$ used in \cite{pmlr-v75-tzen18a}.\mg{the last sentence, why is it an improvement ? Or, how does it help?
%The same question for ULD}

One could also ask what is the choice of the matrix $J$ in NLD. A natural idea is to 
 maximize the exponent $\lambda_1^J$ that appears in Equation~\eqref{eq:norm_J}, i.e., let $J_{opt} := \arg \max_{J = -J^T}  \lambda_1^J\,. $
A formula for $J_{opt}$ and an algorithm to compute it is known (see Fig. 1 in \cite{Lelievre-optdrift}), however this is not practical to compute for optimization purposes as it requires the knowledge of the eigenvectors and eigenvalues of the matrix $H$ which is generally unknown in practice. Nevertheless, $J_{opt}$ gives information about the extent of acceleration that can be obtained. It is known that $\lambda_1^{J_{opt}} = \frac{\text{Tr}(H)}{d}$,
as well as a characterization of the constants $C_{J_{opt}}$ and $n_1$ arising in Equation~\eqref{eq:norm_J} when $J=J_{opt}$ (see Equation (46) in \cite{Lelievre-optdrift}). We see that $md \leq \text{Tr}(H) \leq M(d-1) + m$ as the smallest and the largest eigenvalue of $H$ is $m$ and $M$. Therefore, we have
\begin{equation*}
1 \leq \frac{\lambda_1^{J_{opt}}}{\lambda_1} \leq \frac{M(d-1) + m}{md}.
\end{equation*}
The acceleration is not possible (the ratio above is $1$) if and only if all the eigenvalues of $H$ are the same and are equal to $m$; i.e. when $M=m$ and $\text{Tr}(H)=md$. Otherwise, $J_{opt}$ can accelerate by a factor of $\frac{M(d-1) + m}{md}$ which is on the order of the condition number $\kappa:=M/m$ up to a constant $\frac{d-1}{d}$ which is close to one for $d$ large.
 In practice, one can also use some easily constructed choices of $J$ as suggested in the literature (see e.g. \cite{HHS93}), and run the NLD algorithm using $\alpha J$ (which is still anti-symmetric), where $\alpha$ is a constant that can be tuned and it represents the magnitude of non-reversible purturbations. For example, one can choose
$J$ to be a circulant matrix, e.g.,
$$-J \nabla F(x) = \left( \partial_{x_{d}} F(x) - \partial_{x_{2}} F(x),  \partial_{x_{1}} F(x) -  \partial_{x_{3}} F(x), \ldots,
\partial_{x_{d-1}} F(x) - \partial_{x_{1}} F(x) \right)\,,$$
and this product is easy to implement by shifting the gradient vector in the memory by one unit to the left and one unit to the right and then taking the difference.

%%%%%%%%%%%%%%%%%%%%%%%%%%%%%%%%%%%%%%%%%%
%%%%%%%%%%%%%%%%%%%%%%%%%%%%%%%%%%%%%%%%%%
\section{Exit time for non-reversible Langevin dynamics} \label{sec:global}

For convergence to a small neighborhood of the global minimum, Langevin trajectory needs to not only escape from the neighborhood of a local optimum but also exit the basin of attraction of the current minimum and transit to the basin of attraction of other local minima including the global minima. In particular, the convergence rate to a global minimum is controlled by the mean \emph{exit time}  from the basin of attraction of a local minima in a potential landscape $F(\cdot)$ in \eqref{eqn:overdamped}. In this section we will show that non-reversible Langevin dynamics can lead to faster (smaller) exit times.

Throughout this section, we consider a double-well potential 
$F:\mathbb{R}^d \rightarrow \mathbb{R}$, which has two local minima $a_1, a_2$ with $F(a_2) < F(a_1)$. 
The two minima are separated by a saddle point $\sigma$. See Figure~\ref{fig:double-well} in
the Supplementary File.
In addition to Assumption~\ref{assumptions} (i)-(iii), 
we make generic assumptions that $F \in C^3$, the Hessian of $F$ at each of the local minima is positive definite, 
and that the Hessian of $F$ at the saddle point $\sigma$ has exactly one strictly negative eigenvalue 
(denoted as $-\mu^{\ast}(\sigma)<0$) and other eigenvalues are all positive.
%For the overdamped Langevin diffusion, it is known that
%the expected time of the process starting from $a_1$ and hitting a small neighborhood of $a_2$ is given by  \gao{this paragraph may be removed since it is discussed in Intro already}
%\begin{equation} \label{eq:exit-time-over}
%\mathbb{E}\left[\theta_{a_1 \rightarrow a_2}^{\beta}\right] 
%= [1+ o_{\beta}(1)] \cdot\frac{2 \pi}{\lambda_{J=0}^{\ast}}\cdot e^{\beta [F(\sigma)- F (a_1)]}\cdot\sqrt{ \frac{|\det\mbox{Hess }F (\sigma)|}{\det\mbox{Hess }F(a_1)}}.
%\end{equation}
%Here, $o_{\beta}(1) \rightarrow 0$ as $\beta \rightarrow \infty,$ $\det \mbox{Hess } F(x)$ stands for the determinant of the Hessian of $F$ at $x$, and
%$-\lambda_{J=0}^{\ast}$ is the unique negative eigenvalue of the matrix $\mathbb{L}^{\sigma}$, where $\mathbb{L}^{\sigma} := \mbox{Hess }F (\sigma)$, the Hessian of $F$ at the saddle point $\sigma$. %The rigorous proof of this Eyring-Kramers formula for reversible diffusions
%was first obtained by \cite{Bovier2004} 
%by a potential analysis approach, and then by \cite{Helffer} 
%through Witten Laplacian analysis. We refer
%to \cite{berglund2013kramers} for a survey on mathematical approaches to the Eyring-Kramers formula. Recently, \cite{BR2016} extended the Eyring-Kramers formula to the non-reversible setting and a rigorous proof of this result was obtained in \cite{Landim2017}.
For the overdamped Langevin diffusion, it is known (\cite{Bovier2004, berglund2013kramers}) that
the expected time of the process starting from $a_1$ and hitting a small neighborhood of $a_2$ is:
\begin{equation} \label{eq:exit-time-over}
 \mathbb{E}\left[\theta_{a_1 \rightarrow a_2}^{\beta}\right]
= [1+ o_{\beta}(1)] \cdot\frac{2 \pi}{\mu^{\ast}(\sigma)}\cdot e^{\beta [F(\sigma)- F (a_1)]}\cdot\sqrt{ \frac{|\det\mbox{Hess }F (\sigma)|}{\det\mbox{Hess }F(a_1)}}.
\end{equation}
Here, $o_{\beta}(1) \rightarrow 0$ as $\beta \rightarrow \infty,$ $\det \mbox{Hess } F(x)$ stands for the determinant of the Hessian of $F$ at $x$, and
$-\mu^{\ast}(\sigma)$ is the unique negative eigenvalue of the Hessian of $F$ at the saddle point $\sigma$. This formula is known as the Eyring-Kramers formula for reversible diffusions.
Its rigorous proof was first obtained by \cite{Bovier2004} 
by a potential analysis approach, and then by \cite{Helffer} 
through Witten Laplacian analysis. We refer
to \cite{berglund2013kramers} for a survey on mathematical approaches to the Eyring-Kramers formula.
We note that in many practical applications, for instance in the training of neural networks, the eigenvalues of the Hessian at local extrema concentrate around zero and the
magnitude of the eigenvalues $m$ and $\mu^{\ast}(\sigma)$ can often be very small (see e.g. \cite{sagun2016eigenvalues,Chaudhari}).

% In this case, the recurrence time $\Tr$ and the expected hitting time $\mathbb{E}\left[\theta_{a_1 \rightarrow a_2}^{\beta}\right]$, being quantities that scale inversely with $m$ and $\mu^{*}(\sigma)$ respectively, can be quite large. This results in slow convergence of the overdamped process to its equilibrium (\cite{bovier2005metastability}). It is also known that the overdamped Langevin diffusion is a reversible\footnote{We note that overdamped Langevin SDE is \emph{reversible} in the sense that if $X_0$ is distributed according to the stationary measure $\pi$, then $(X_t)_{0\leq t \leq T}$ and $(X_{T-t})_{0\leq t\leq T}$ have the same law. Mathematically, this is equivalent to the infinitesimal generator of $X_{t}$ process being (self-adjoint) symmetric in $L^2(\pi)$ \cite{Lelievre-optdrift,rogers2000diffusions}. Roughly speaking, reversibility of a Markov process says that statistical properties of the process are preserved if the process is run backwards in time.} Markov process, and reversible processes can be much slower than their non-reversible variants that admit the same equilibrium distribution in terms of their convergence rate to the equilibrium (see e.g. \cite{shah2009gossip,Lelievre-optdrift,diaconis-sampler,HHS93,HHS05,Eberle,gade2007optimizing}). 

%%%%%%%%%%%%%%%%%%%%%%%%%%%%%%%%
\subsection{Underdamped Langevin dynamics}

%Recall the underdamped Langevin diffusion defined in \eqref{eq:VL}--\eqref{eq:XL} and the underdamped Langevin dynamics (ULD) defined in \eqref{eq:V-iterate}--\eqref{eq:X-iterate}.
%As ULD tracks the underdamped diffusion closely when the stepsize is small, we first discuss the exit times for the underdamped diffusion to get some intuition for the accerelation compared with the overdamped Langevin diffusion. 

%\subsubsection{Acceleration of exit times in continuous time dynamics}

%For the underdamped diffusion, the diffusion matrix is not invertible. It is argued in Remark 5.2 in 
%\cite{BR2016} that
%the analogue of \eqref{eq:exit-time-over}, i.e., the Eyring-Kramers formula, can still be expected to hold in this underdamped case. 
%A rigorous mathematical proof for general dimension is not known in the literature to the best of our knowledge.  

%Suppose $F$ has two minima located at $a_1, a_2$ and separated by a saddle-point $\sigma$. 
%\textbf{Underdamped Langevin dynamics.}
Denote $\Theta_{a_1 \rightarrow a_2}^{\beta}$ as the first time of that 
the underdamped diffusion \eqref{eq:VL}--\eqref{eq:XL}
starting from $a_1$ and hitting a small neighborhood of $a_2$. 
\begin{comment}
When dimension $d=1$, it was derived originally in \cite{kramers1940brownian} that %(see also \cite{Stein93}) that
\begin{equation} \label{eq:exit-time-under}
\mathbb{E}\left[\Theta_{a_1 \rightarrow a_2}^{\beta}\right] 
 = [1+ o_{\beta}(1)] \cdot \frac{2 \pi}{\mu_{\ast}}\cdot e^{\beta [F(\sigma)- F(a_1)]}\cdot\sqrt{ \frac{|F''(\sigma)|}{F''(a_1)}}\,,
\end{equation}
where
\begin{equation*}
\mu_{\ast} = \frac{1}{2} \cdot \left(\sqrt{\gamma^2 - 4F''(\sigma)} - \gamma \right)\,.
\end{equation*}
In view of \eqref{eq:exit-time-over}, we can deduce that if $$\gamma - F''(\sigma) <1,$$ we have when $d=1,$
\begin{equation*}
\lim_{\beta \rightarrow \infty} \frac{\mathbb{E}\left[\Theta_{a_1 \rightarrow a_2}^{\beta}\right]}{\mathbb{E}\left[\theta_{a_1 \rightarrow a_2}^{\beta}\right]} = \frac{- F''(\sigma)}{\mu_{\ast}} < 1.
\end{equation*}
That is, for large $\beta$,
the mean exit time of one dimensional underdamped diffusion is strictly smaller than that of the overdamped diffusion when $\mu_{\ast} > - F''(\sigma)$.
\end{comment}
\cite{BR2016} (Remark 5.2) suggests that the expected exit time is:
%\footnote{For the underdamped diffusion, the diffusion matrix is not invertible. It is argued in Remark 5.2 in 
%\cite{BR2016} that
%the analogue of \eqref{eq:exit-time-over}, i.e., the Eyring-Kramers formula, can still be expected to hold in this underdamped case. 
%A rigorous mathematical proof for general dimension is not known in the literature to the best of our knowledge. } 
\begin{equation*} 
\mathbb{E}\left[\Theta_{a_1 \rightarrow a_2}^{\beta}\right] 
= [1+ o_{\beta}(1)] \cdot\frac{2 \pi}{\mu_{\ast}}\cdot e^{\beta [F(\sigma)- F (a_1)]}\cdot\sqrt{ \frac{|\det\mbox{Hess }F (\sigma)|}{\det\mbox{Hess }F(a_1)}},
\end{equation*}
where $\mu_*$ is the unique positive eigenvalue of the matrix 
$
\hat H_{\gamma} (\sigma)=
\left[
\begin{array}{cc}
- \gamma I & - \mathbb{L}^{\sigma}
\\
I & 0
\end{array}
\right],
$
where $\mathbb{L}^{\sigma}$ is the Hessian matrix of $F$ at the saddle point $\sigma$. %The matrix $\hat H_{\gamma} (\sigma)$ is obtained as the Jacobian matrix
%\begin{equation}
%D \left( \left[
%\begin{array}{cc}
%- \gamma I & - I
%\\
%I & 0
%\end{array} 
%\right] 
%\left[
%\begin{array}{c}
%V
%\\
%\nabla F(x)
%\end{array} 
%\right] \right),
%\end{equation}
%at the saddle point. See Section 2.2.2, Remarks 5.1 and 5.2 in 
%\cite{BR2016} for details. 
%So to compare with the expected exit time for the overdamped diffusion in \eqref{eq:exit-time-over}, we need to compare $\mu_*$ and $\mu^{\ast}(\sigma)$.
 One can readily show that $\mu_*$ is given by the positive eigenvalue of the $2 \times 2$ matrix
$
\left[\begin{array}{cc}
-\gamma & \mu^{\ast}(\sigma)
\\
1  & 0 
\end{array} 
\right],
$
which implies that
$\mu_* = \frac{1}{2} \cdot \left(\sqrt{\gamma^2 + 4 \mu^{\ast}(\sigma)} - \gamma \right).$
So if $\gamma + \mu^{\ast}(\sigma) <1$, then we have $\mu_{\ast} > \mu^{\ast}(\sigma)$ and therefore
\begin{equation}\label{eq-exit-acceleration}
\lim_{\beta \rightarrow \infty} \mathbb{E}\left[\Theta_{a_1 \rightarrow a_2}^{\beta}\right]\Big/\mathbb{E}\left[\theta_{a_1 \rightarrow a_2}^{\beta}\right] = \mu^{\ast}(\sigma)/\mu_{\ast} < 1.
\end{equation}
That is, the mean exit time for the underdamped diffusion is smaller compared with that of the overdamped diffusion. Roughly speaking, the condition $\gamma + \mu^{\ast}(\sigma) <1$ says that if the curvature of the saddle point in the negative descent direction is not too steep (i.e. if $\mu^{\ast}(\sigma) < 1$), we can choose $\gamma$ small enough to accelerate the exit time of the reversible Langevin dynamics. Intuitively speaking, it can be argued that the underdamped process can climb hills and explore the state space faster as it is less likely to go back to the recent states visited due to the momentum term 
(see e.g. \cite{BR2016}). For the discrete time dynamics, it is intuitive to expect that the exit time of the underdamped discrete dynamics is close to that of the continuous time diffusion when the step size is small \cite{Gobet2017}, and hence a similar result as \eqref{eq-exit-acceleration} will hold for the discrete dynamics when $\gamma + \mu^{\ast}(\sigma) <1$. 

To apply the results in  \cite{Gobet2017}, we consider a sequence of bounded domains $D_n$ indexed by $n$ so that the following conditions hold: first, the region $D_n$ contains $a_1, a_2, \sigma$ for large $n$; second, as $n$ increases, $D_n$ increases to the set $D_{\infty}= O^c(a_2):= \mathbb{R}^{d}\backslash O(a_{2})$, 
where $O(a_{2})$ denotes a small neighborhood of $a_{2}$; third, the underdamped SDE (diffusion) is non-degenerate along the normal direction to the boundary of $D_n$ with probability one. 
Fix the parameters $\beta$ and $\gamma$ in the underdamped Langevin dynamics. 
Denote $ \hat \Theta_{a_1 \rightarrow a_2}^{\beta, n}$ be the exit time of $X_k$ (from the ULD dynamics) starting from $a_1$ and exiting domain $D_n$. Fix $\epsilon>0$. One can choose a sufficiently large $n$ and choose a constant $\tilde \eta(\epsilon, n, \gamma, \beta)$ so that 
for stepsize $\eta \le \tilde \eta(\epsilon, n, \gamma, \beta)$, we have
\begin{equation} \label{eq:bound}
\left|  \mathbb{E}\left[ \hat \Theta_{a_1 \rightarrow a_2}^{\beta, n} \right] - \mathbb{E}\left[\Theta_{a_1 \rightarrow a_2}^{\beta}\right] \right| < 2 \epsilon.
\end{equation}
To see this, we use Theorems 3.9 and 3.11 in \cite{Gobet2017}. 
%It suffices to verify their Assumptions~3.5 and 3.8. Assumption~3.5 in \cite{Gobet2017} is automatically satisfied as the region $D_n$ is bounded, and the drift and the diffusion coefficients of underdamped Langevin diffusions are continuous. Assumption~3.8 in \cite{Gobet2017} requires a mild non-characteristic boundary condition for the underdamped dynamics exiting $D_n$, but with our choice of $D_n$, this assumption can be satisfied. 
Write $\hat \theta_{a_1 \rightarrow a_2}^{\beta, n}$ as the exit time of $X_k$ (from the overdamped discrete dynamics) starting from $a_1$ and exiting the domain $D_n$. Then one can also expect that when $n$ is large and the step size is small, the mean of $\hat \theta_{a_1 \rightarrow a_2}^{\beta, n}$ will be close to $\mathbb{E}\left[\theta_{a_1 \rightarrow a_2}^{\beta}\right]$, the continuous exit time of the overdamped diffusion given in \eqref{eq:exit-time-over}. See Proposition~\ref{prop:anti-disc-cont} in the next section. 
It then follows from \eqref{eq-exit-acceleration} and \eqref{eq:bound} that for large enough $\beta, n$, and sufficiently small stepsize $\eta$, we obtain the acceleration in discrete time:
\begin{equation*}
\mathbb{E}\left[\hat \Theta_{a_1 \rightarrow a_2}^{\beta, n}\right]
\Big/\mathbb{E}\left[\hat \theta_{a_1 \rightarrow a_2}^{\beta, n}\right]  
= \mathcal{O}\left(\sqrt{\mu^{\ast}(\sigma)}\right) < 1.
\end{equation*}

\subsection{Non-reversible Langevin dynamics}

%Recall the non-reversible Langevin dynamics defined in \eqref{eqn:nonreversible:discrete} and the corresponding continuous time dynamics defined in \eqref{eqn:nonreversible1}. 

%\begin{assumption} from \cite{Landim2017} (\gao{XG: can be consolidated with previous ones and will do later)}
%\begin{itemize}
%\item [(P1)] $F \in C^3$ and $\lim_{n \rightarrow \infty} \inf_{\Vert x\Vert \ge n} F(x) = \infty$.
%\item [(P2)] $F$ has finitely many critical points. Only two of them, denoted by $a_1$ and $a_2$, are local minima. The Hessian of $F$ at each of these minima is positive definite.
%\item [(P3)] The Hessian of $F$ at the saddle point $\sigma$ has exactly one strictly negative eigenvalue and other eigenvalues are all positive.
%\item [(P4)] The function $F$ satisfies
%\begin{equation}
%\lim_{\Vert x\Vert \rightarrow\infty} \frac{x}{\Vert x\Vert} \nabla F (x) = \infty,
%\end{equation}
%and
%\begin{equation}
%\int_{\mathbb{R}^{d}} \exp(- \beta F (x)) dx < \infty,
%\end{equation}
%for all $\beta>0.$
%\end{itemize}
%\end{assumption}

%\subsubsection{Acceleration of exit times for continuous time dynamics}

%We first discuss the continuous-time dynamics \eqref{eqn:nonreversible1}. 
%\textbf{Non-reversible Langevin dynamics.}
\cite{Landim2017} (Theorem 5.2) showed that the expected time of the non-reversible diffusion $X(t)$ in \eqref{eqn:nonreversible1} starting from $a_1$ and hitting a small neighborhood of $a_2$ is given by 
\begin{equation} \label{eq:exit-time}
\mathbb{E}\left[\tau_{a_1 \rightarrow a_2}^{\beta}\right] 
= [1+ o_{\beta}(1)] \cdot\frac{2 \pi}{\mu_J^{\ast}}\cdot e^{\beta [F(\sigma)- F (a_1)]}\cdot\sqrt{ \frac{|\det\mbox{Hess }F (\sigma)|}{\det\mbox{Hess }F(a_1)}}.
\end{equation}
Here, 
%$o_{\beta}(1) \rightarrow 0$ as $\beta \rightarrow \infty,$ $\det \mbox{Hess } F(x)$ stands for the determinant of the Hessian of $F$ at $x$, and
$-\mu_J^{\ast}$ is the unique negative eigenvalue of the matrix $A_J \cdot \mathbb{L}^{\sigma}$, where $\mathbb{L}^{\sigma} := \mbox{Hess }F (\sigma)$, the Hessian of $F$ at the saddle point $\sigma$. 
%The existence and uniqueness of such a negative eigenvalue $-\mu_J^{\ast}$ was proved in Lemma~11.1 in \cite{Landim2018}. 
%To facilitate the presentation, 
We denote $u$ for the corresponding eigenvector of $A_{J} \mathbb{L}^{\sigma}$ for the eigenvalue $- \mu_J^{\ast}<0$.  
%, i.e., we have
%\begin{equation}\label{def-eigvec-u}
%A_{J} \mathbb{L}^{\sigma}  u = - \mu_J^{\ast} u. 
%\end{equation} 
%In addition, since $\mathbb{L}^{\sigma}$ is a real symmetric matrix, 
and $\mathbb{L}^{\sigma}=S^{T}DS$,
%\begin{equation}\label{eqn:S}
%\mathbb{L}^{\sigma}=S^{T}DS,    
%\end{equation}
for a real orthogonal matrix $S$, where $D= \text{diag} (\mu_1, \mu_2, \ldots, \mu_d)$ 
with $\mu_1< 0 < \mu_2 < \ldots < \mu_d$ being the eigenvalues of $\mathbb{L}^{\sigma}$. 
%It is clear that we have $-\mu_1 = \mu^{\ast}(\sigma)$.
%We have the following result.

%In the case dimension $d = 2$, it was checked in Section 5.1.4 of \cite{BR2016} that $\mu_J^{\ast} \ge \mu^{\ast}(\sigma)$. From \eqref{eq:exit-time}, this suggests that
%$\mathbb{E}\left[\tau_{a_1 \rightarrow a_2}^{\beta}\right]$  is smaller for non-reversible Langevin diffusions compared with reversible Langevin diffusions with $J=0$. That is, the non-reversible diffusion exits a local minimum faster. We now present a result showing that this holds for general dimension $d \ge 2$ as well. 

\begin{proposition}\label{prop:eigenvalue-compare}
We have $\mu_J^{\ast} \ge \mu^{\ast}(\sigma)$. As a consequence, 
\begin{equation}\label{eq:exit-compare-cont}
\lim_{\beta \rightarrow \infty} \mathbb{E}\left[\tau_{a_1 \rightarrow a_2}^{\beta}\right]\Big/\mathbb{E}\left[\tau_{a_1 \rightarrow a_2}^{\beta}\right]_{J=0} = \mu^{\ast}(\sigma)/\mu_J^{\ast} \le 1.
\end{equation}
The equality is attained if and only if $u$ is a singular vector of $J$ satisfying $Ju = 0$.
\end{proposition}

%if and only if $(Su)_i = 0$ for $i=2,\dots,d$ where $u$ and $S$ are defined by \eqref{def-eigvec-u}--\eqref{eqn:S}, which occurs

Proposition \ref{prop:eigenvalue-compare} shows that if $J$ is not singular, the non-reversible dynamics is generically faster than the reversible dynamics in the sense of smaller mean exit times. This holds for the discrete dynamics as well, since the exit time for the discrete dynamics is close to that of the continuous dynamics for sufficiently small stepsizes, see e.g. \cite{Gobet2005}.

%A natural question to ask is how much acceleration can be achieved for the exit time with 

%%%%%%%%%%%%%%%%%%%%%%%%%%%%%%%%%%%%%%%%%%%%%%%%%%%%%%%%%%%%%%%%%

%\subsubsection{Acceleration of exit times for discrete time dynamics}

Next let us discuss the discrete dynamics \eqref{eqn:nonreversible:discrete}. Unlike the underdamped diffusion which has a non-invertible diffusion matrix, the non-reversible Langevin diffusion in \eqref{eqn:nonreversible1} is uniformly elliptic. So to show the discrete exit time is close to the continuous exit time for non-reversible Langevin dynamics, we can apply the results in \cite{Gobet2005} and proceed as follows. Let $B_n$ be the ball centered at zero with radius $n$ in $\mathbb{R}^d$. 
For $n$ sufficiently large, we always have $a_1, a_2, \sigma \in B_n$. Let $\bar D_n= B_n \setminus O(a_2) $, 
where $O(a_2)$ denotes a small neighborhood of $a_2$. 
It follows that the set $\bar D_n$ is bounded for each $n$, and it increases to the set $O^c(a_2)$ as $n$ is sent to infinity.
Write $\hat \tau_{a_1 \rightarrow a_2}^{\beta, n}$ for the first time that the discrete-time dynamics 
starting from $a_1$ and exit the region $\bar D_n$. Then we can obtain from \cite{Gobet2005} the following result, which implies the exit times of non-reversible Langevin dynamics is smaller compared with that of reversible Langevin dynamics. %See the Appendix for a proof. 

\begin{proposition}\label{prop:anti-disc-cont}
Fix the antisymmetric matrix $J$, the temperature parameter $\beta$, and $\epsilon>0$. One can choose a sufficiently large $n$ and a constant $\bar \eta(\epsilon, n, \beta)$
so that for stepsize $\eta \le \bar \eta(\epsilon, n, \beta)$, we have
\begin{equation*}
\left|  \mathbb{E}\left[ \hat \tau_{a_1 \rightarrow a_2}^{\beta, n} \right] - \mathbb{E}\left[\tau_{a_1 \rightarrow a_2}^{\beta}\right] \right| < 2 \epsilon.
\end{equation*}
It then follows from Proposition~\ref{prop:eigenvalue-compare} that for large $\beta$ we have 
\begin{equation}\label{eq:exit-compare-disc}
\mathbb{E}\left[ \hat \tau_{a_1 \rightarrow a_2}^{\beta, n} \right]\Big/\mathbb{E}\left[\hat \tau_{a_1 \rightarrow a_2}^{\beta, n}\right]_{J=0}< 1,
\end{equation}
provided that $(Su)_i \ne 0$ for some $i \in \{2, \ldots, d\}$ which occurs if and only if $u$ is a singular vector of $J$ satisfying $Ju = 0$.
\end{proposition}

\section{Numerical Illustrations}
\paragraph{{Choice of algorithm parameters.}} 

\begin{wrapfigure}{r}{0.33\columnwidth}
    \centering
     \vspace{-7pt}
    \includegraphics[width=.32\columnwidth]{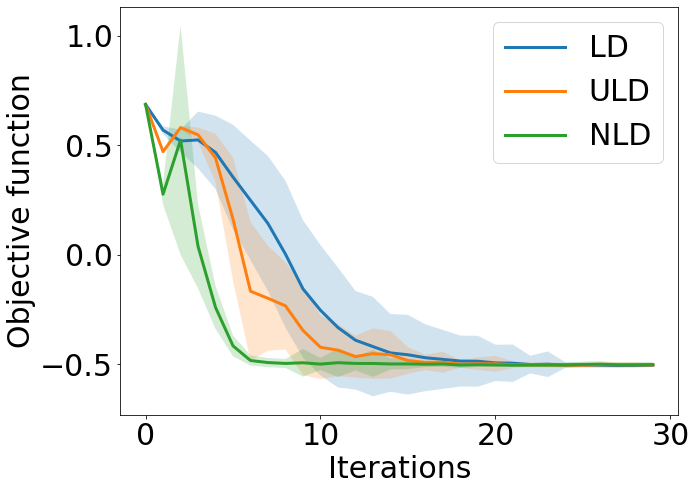}
    %\vspace{5pt}
     \includegraphics[width=.32\columnwidth]{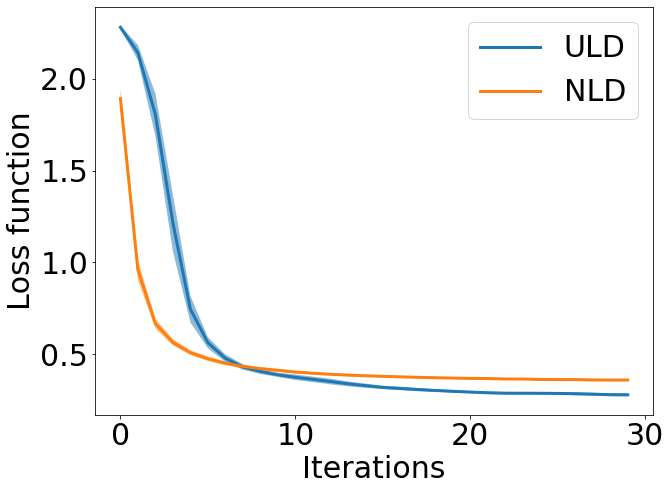}
   \vspace{-5pt}
    \caption{Choice of algorithm parameters and comparing ULD and NLD.}
   \vspace{-4pt}
    \label{fig2}
\end{wrapfigure}

In practice, for the NLD algorithm, the matrix $J$ can be chosen as a random anti-symmetric matrix. For quadratic objectives, there is a formula for optimal $J$ matrix; see e.g. \cite{Lelievre-optdrift}. 
For the ULD algorithm, we can take the parameter $\gamma = 2\sqrt{m}$ as predicted by our theory (Lemma~\ref{lem:keylemma}) for quadratics. On the top panel of Figure \ref{fig2}, we compare ULD and NLD to LD for the double well example with random initialization over 100 runs where $J$ is chosen randomly and $\gamma = 2\sqrt{m}$. In this simple example, we observe NLD and ULD have smaller mean exit times (from a barrier) compared to LD.

\paragraph{{Comparing ULD and NLD algorithms.}}

In general, it is not easy to compare 
the theoretical performance of ULD and NLD algorithms.
However, in some regimes, our theory predicts one
is better than the other.
For example, when the smallest eigenvalue $m$ is close
to the largest eigenvalue $M$, NLD will not improve much
upon LD, but ULD will improve upon LD if $m$ is small and ULD will be faster than NLD.
In this case, ULD has better performance than NLD. 
On the bottom panel of Figure \ref{fig2}, we provide an example for training fully-connected neural networks on MNIST where ULD was faster when both methods were tuned.

%%%%%%%%%%%%%%%%%%%%%%%%%%%%%%%%%%%%%%%%%%%
\vspace{-0.05in}
\section{Conclusion}
\vspace{-0.05in}
{Langevin Monte Carlo are powerful tools for sampling from a target distribution 
as well as for optimizing a non-convex objective. 
The classic Langevin dynamics (LD) is based on the first-order Langevin diffusion which is reversible in time. 
We studied the two variants that are based on non-reversible Langevin diffusions: 
the underdamped Langevin dynamics (ULD) and the Langevin dynamics with a non-symmetric
drift (NLD). We showed that both ULD and NLD can improve upon the recurrence
time for LD in \cite{pmlr-v75-tzen18a}
and discussed the amount of improvement. We also showed that non-reversible variants
can exit the basin of attraction of a local minimimum faster when
the objective has two local minima separated by a saddle point and discussed the amount of improvement. 
By breaking the reversibility in the Langevin dynamics, our results
quantify the improvement in performance and 
fill a gap between the theory and practice of non-reversible Langevin algorithms.}

\section*{Acknowledgements and Disclosure of Funding}

{The authors thank four anonymous referees for helpful suggestions.}
The authors are very grateful to Emmanuel Gobet, Claudio Landim, Insuk Seo and Gabriel Stoltz for helpful comments and discussions. 
{The authors also thank Yuanhan Hu for the help with numerical experiments.}
Xuefeng Gao acknowledges support from Hong Kong RGC Grants 24207015 and 14201117. 
Mert G\"{u}rb\"{u}zbalaban's research is supported in part by the grants NSF DMS-1723085 and NSF CCF-1814888. 
Lingjiong Zhu is grateful to the support from the grant NSF DMS-1613164.
%%%%%%%%%%%%%%%%%%%%%%%%%%%%%%%%%%%%%%%%%%%%%%%%%%%%%%%%%%%%%%%%%%%%%%%%%%%%%%%%

\bibliographystyle{alpha}
\bibliography{langevin}


\appendix

\section{Figures}\label{sec:fig}

\begin{figure}[htb]
\begin{center}
\includegraphics[scale=0.4]{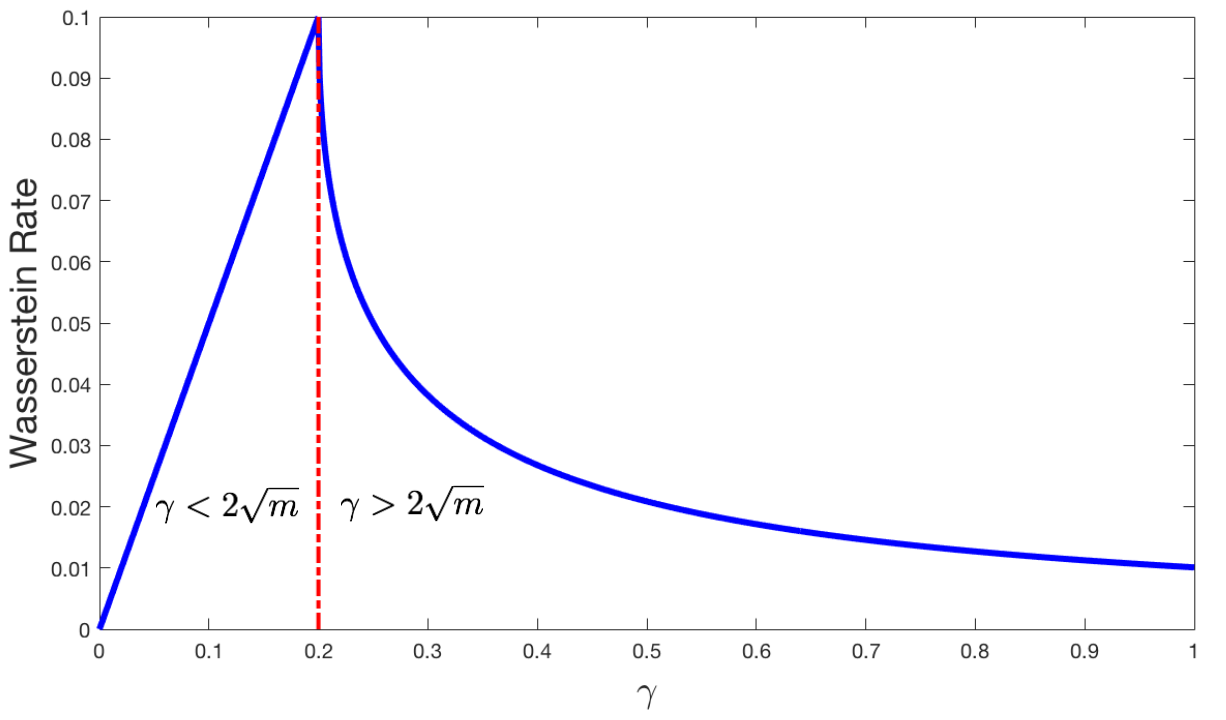}
\caption{\label{fig-2}
The norm $\Vert e^{-tH_{\gamma}}\Vert$ is optimized for the choice of $\gamma = 2\sqrt{m}$. 
This is illustrated in the figure for $m = 0.01$.}
%The exponent of the Wasserstein rate is optimized for the choice of $\gamma = 2\sqrt{m}$, see Remark \ref{remark:ULD} in Appendix~\ref{sec:quadratic}. This is illustrated in the figure for $m = 0.01$. }
\end{center}
\end{figure} 

%%%%%%%%%%%%%%%%%%%%%%%%%%%%%%%%%%%%%%%

\begin{figure}[htb] 
\begin{center}
\includegraphics[scale=0.40]{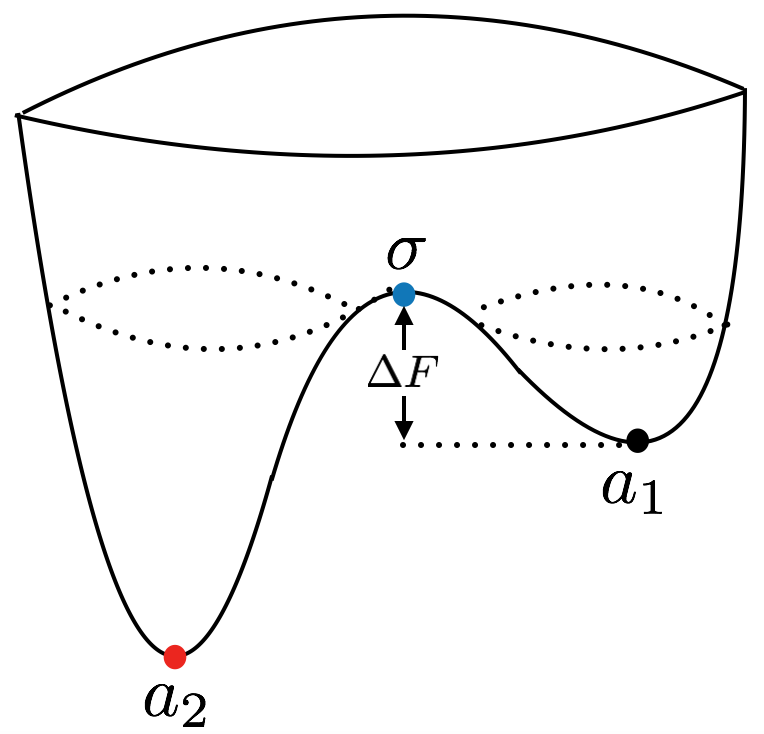}
\caption{\label{fig-1}A double-well example. 
Here, $\Delta F=F(\sigma)-F(a_{1})$.
There are exactly two local minima $a_1$ and $a_2$ which are separated with a saddle point $\sigma$.}
\end{center} \label{fig:double-well}
\end{figure}

\section{Proof of results in Section~\ref{sec:rec:escape}} \label{sec:proof-underdamped}

\subsection{Proof of Lemma~\ref{lem:keylemma}}
\begin{proof}
Let $H$ be a symmetric positive definite matrix with eigenvalue decomposition
$H= Q D Q^T$,
where $D$ is diagonal with eigenvalues in increasing order  
$m:=\lambda_{1}\leq\lambda_{2}\leq\cdots\leq\lambda_{d}=:M$ 
of the matrix $H$.
%where $m$ and $M$ are the lower and upper bounds on the eigenvalues.
%where we used the fact that $L$ is the largest eigenvalue of the Hessian at the local minimum. 
Recall $H_{\gamma}$ from \eqref{def-Hgamma0}.
%\begin{equation*}
%H_{\gamma}=
%\left[
%\begin{array}{cc}
%\gamma I & H
%\\
%-I & 0
%\end{array}
%\right].
%\end{equation*}
Note that 
\begin{equation*}
H_\gamma = \left[
\begin{array}{cc}
Q & 0
\\
0 & Q
\end{array}
\right]
G_\gamma   
\left[
\begin{array}{cc}
Q^T & 0
\\
0 & Q^T
\end{array}
\right], \quad G_\gamma : = \left[
\begin{array}{cc}
\gamma I & D
\\
-I & 0
\end{array}
\right].
\end{equation*}
Therefore $H_\gamma$ and $G_\gamma$ have the same eigenvalues. 
Due to the structure of $G_\gamma$, it can be seen that there exists a permutation matrix $P$ such that 
\begin{equation}\label{eq-perm-diagonalize}
T_\gamma : = P G_\gamma P^T  = \left[
\begin{array}{cccc}
T_1(\gamma) & 0 & 0 & 0 
\\
0 & T_2(\gamma) & 0  & 0 
\\
\vdots & \cdots & \ddots & \vdots
\\
0 & 0 & 0 & T_d(\gamma) 
\end{array}
\right]\,,
\quad
\text{where}
\quad
T_i(\gamma)  := \left[\begin{array}{cc}
\gamma & \lambda_i 
\\
-1  & 0 
\end{array} 
\right]\,,
\end{equation}
with $i=1,2,\ldots,d$, and $T_{i}(\gamma)$ are $2\times 2$ block matrices with the eigenvalues: 
\begin{equation}\label{mu:i:pm}
\mu_{i,\pm} := \frac{\gamma \pm \sqrt{\gamma^2 - 4\lambda_i}}{2}\,.\qquad i=1,2,\ldots,d\,. 
\end{equation}
We observe that $T_\gamma$ and $G_\gamma$ (and therefore $H_\gamma$) have the same eigenvalues
and the eigenvalues of 
$T_\gamma$ are determined by the eigenvalues of the $2\times 2$ block matrices $T_i(\gamma)$.

Since
$H_\gamma$ is unitarily equivalent to the matrix $T_\gamma$, i.e. there exists a unitary matrix $U$ such that $H_\gamma = U T_\gamma U^*$, we have 
 $ \left\| e^{-tH_\gamma} \right\| = \left\| U e^{-t T_\gamma} U^*\right\| = \left\| e^{-t T_\gamma}\right\|$.
Since $T_\gamma $ is a block diagonal matrix with $2\times 2$ blocks $T_i(\gamma)$
we have 
$\left\| e^{-t T_\gamma}\right\| = \max_{1\leq i\leq d} \left\| e^{-t T_i(\gamma) }\right\|$. 
Assume that $\gamma^2 - 4 \lambda_1 =\gamma^{2}-4m\leq 0$ so that the eigenvalues $\mu_{i,\pm}$ of $T_i(\gamma)$
(see Eqn. \eqref{mu:i:pm}) 
are real when $\gamma = 2\sqrt{m}$ and complex when $\lambda < 2\sqrt{m}$. 
Note that
\begin{equation}\label{eq-exp-in-t}
\left\| e^{-t T_i(\gamma)}\right\|  = e^{-t \gamma / 2} \left\|e^{-t\tilde{T}_i(\gamma)}\right\|\,,  
\qquad
\text{where}\quad \tilde{T}_i(\gamma) := T_i(\gamma) - \frac{\gamma}{2} I,\qquad 1\leq i\leq d.
\end{equation}

We consider $\gamma \in (0,2\sqrt{m}]$. Depending on the value of $\lambda_i$ and $\gamma$, there are two cases: 

\textbf{Case 1}. If $\gamma < 2\sqrt{m}$ or ($\lambda_i > m$ and $\gamma = 2\sqrt{m}$), then $\tilde{T}_i(\gamma)$ has purely imaginary eigenvalues that are complex 
conjugates which we denote by 
$\tilde{\mu}_{i,\pm} = \pm i \frac{\sqrt{4\lambda_i - \gamma^2}}{2}$, $1\leq i\leq d$.
We will show that the last term in \eqref{eq-exp-in-t} stays bounded due to the imaginariness of the eigenvalues of  $\tilde{T}_i(\gamma)$. It is easy to check that $2\times 2$ matrix $\tilde{T}_i(\gamma)$ have the eigenvectors $v_{i,\pm} = [{\mu}_{i,\pm}, -1]^T $. 
If we set
	$ G_i := \begin{bmatrix} v_{i,+} & v_{i,-} 
	\end{bmatrix}  \in \mathbb{C}^{2\times 2},
    $ 
the eigenvalue decomposition of $\tilde{T}_i(\gamma)$  is given by 

	$$ \tilde{T}_i(\gamma) = G_i  \begin{bmatrix} \tilde{\mu}_{i,+} & 0 \\ 
																	0            &  \tilde{\mu}_{i,-} 
										 \end{bmatrix}G_i^{-1}, \quad \mbox{where} \quad 
										 G_i^{-1} = \frac{1}{\det G_i}\begin{bmatrix} -1 & - \mu_{i,-} \\
										 											 1 & \mu_{i,+}
										 \end{bmatrix}\,,
    $$
and $\det G_i = i \sqrt{4\lambda_i - \gamma^2}$.
We can compute that 
	\begin{align*} 
	e^{-t \tilde{T}_i(\gamma)} &= G_i  \begin{bmatrix}  e^{-it \sqrt{4\lambda_i - \gamma^2}/2} & 0 \\ 0 &  e^{it \sqrt{4\lambda_i - \gamma^2}/2} \end{bmatrix}  G_i^{-1} \\
				 &= \frac{1}{\det G_i} \begin{bmatrix}
				 			\mu_{i,+} & \mu_{i,-} \\
				 			     -1       &       -1
				 \end{bmatrix} \begin{bmatrix} - e^{-it \sqrt{4\lambda_i - \gamma^2}/2} & - \mu_{i,-}  e^{-it \sqrt{4\lambda_i - \gamma^2}/2} \\ 
				 e^{it \sqrt{4\lambda_i - \gamma^2}/2} & \mu_{i,+}e^{it \sqrt{4\lambda_i - \gamma^2}/2} 
				 \end{bmatrix}  \\
				 &= \frac{1}{i\sqrt{4\lambda_i - \gamma^2}} 
				 \begin{bmatrix}
				 		2 \mbox{Imag} \left( \mu_{i,-}  e^{it \sqrt{4\lambda_i - \gamma^2}/2}\right)  &  2i |\mu_{i,+}|^2 \sin\left(t \sqrt{4\lambda_i - \gamma^2} /2\right) \\
				 		-2i \sin \left(t\sqrt{4\lambda_i - \gamma^2}/2\right)  & 2 \mbox{Imag}\left(\mu_{i,+} e^{it \sqrt{4\lambda_i - \gamma^2}/2}\right)\nonumber
				 \end{bmatrix}\,,
	\end{align*}
where $\mbox{Imag}(a+ib) := ib$ denotes the imaginary part of a complex number.	As a consequence, by taking componentwise absolute values  
\begin{align}
  \left\| e^{- t \tilde{T}_i(\gamma)} \right\| &\leq \frac{1}{\sqrt{4\lambda_i - \gamma^2}} \left \| \begin{bmatrix} 
  								2 | \mu_{i,-}|  & 2 | \mu_{i,+}|^2 \\
  								2                      & 2 | \mu_{i,+} |
  							\end{bmatrix} \right\| 
							 =   \frac{1}{\sqrt{4\lambda_i - \gamma^2}} \left \| \begin{bmatrix} 
  								2 \sqrt{\lambda_i}  & 2 \lambda_i \\
  								2                      & 2 \sqrt{\lambda_i}
  							\end{bmatrix} \right\|  \nonumber  \\
  							&= \frac{1}{\sqrt{4\lambda_i - \gamma^2}} \left \| \begin{bmatrix} 
  								2 \sqrt{\lambda_i} \nonumber \\
  								2                     
  							\end{bmatrix}  
  							\begin{bmatrix} 
  							1 & \sqrt{\lambda_i}
  							\end{bmatrix} \right \|
  	                                   = \frac{1}{\sqrt{4\lambda_i - \gamma^2}} \left \| \begin{bmatrix} 
  								2 \sqrt{\lambda_i}  \nonumber \\
  								2                     
  							\end{bmatrix} \right \| 
  							\left\| \begin{bmatrix} 
  							1 & \sqrt{\lambda_i}
  							\end{bmatrix} \right \| \nonumber \\ 
  							&= \frac{2(1+\lambda_i)}{\sqrt{4\lambda_i - \gamma^2}}\,, \label{ineq-resolvent-bound}
\end{align} 	
where the second from last equality used the fact that the 2-norm of a rank-one matrix is equal to its Frobenius norm. 
\footnote{The 2-norm of a rank-one matrix $R = uv^*$ should be exactly equal to $\sigma = \| u\| \| v\|$.
This follows from the fact that we can write $R = \sigma \tilde{u} \tilde{v}^T$ where $\tilde{u}$ and $\tilde{v}$ have unit norm. This would be a singular value decomposition of $R$, showing that all the singular values are zero except a singular value at $\sigma$. Because the 2-norm is equal to the largest singular value, the 2-norm of $R$ is equal to $\sigma$.}
Then, it follows from \eqref{eq-exp-in-t} that
$\left\| e^{-t T_i(\gamma)}\right\|  = e^{-t \gamma / 2} \left\|e^{-t\tilde{T}_i(\gamma)}\right\| \leq  \frac{2(1+\lambda_i)}{\sqrt{4\lambda_i - \gamma^2}}
e^{-t \gamma / 2}$,
which implies 
$\left\| e^{-t H_\gamma}\right\| =\left\| e^{-t T_\gamma}\right\| 
\leq \max_{1\leq i\leq d}\left\| e^{-t T_i(\gamma)}\right\|  \leq \frac{2 (1+M)}{\sqrt{4m-
\gamma^2}}  e^{-t \gamma / 2}$,
provided that $\gamma^2 - 4m < 0$. 
In particular, if we choose $\hat{\varepsilon}=1-\frac{\gamma}{2\sqrt{m}}$ for any $\hat{\varepsilon}> 0$, we obtain
\begin{equation*}
\left\| e^{-t H_\gamma}\right\| \leq \frac{1+M}{\sqrt{m(1-(1-\hat{\varepsilon})^{2})}} e^{-\sqrt{m}(1-\hat{\varepsilon})t}. 
\end{equation*}
The proof for \textbf{Case 1} is complete.

\textbf{Case 2}. If $\gamma = 2\sqrt{m}$ and $\lambda_i = m$, then $\tilde{T}_i(\gamma)$ has double eigenvalues at zero and is not diagonalizable. It admits the Jordan decomposition
$$\tilde{T}_i(\gamma) = G_i  \begin{bmatrix} 0 & 1 \\ 
0 & 0
\end{bmatrix}G_i^{-1} \quad \mbox{with} \quad 
G_i = \begin{bmatrix} 
    \sqrt{m} & 1 \\
    -1 & 0 
\end{bmatrix} \quad \mbox{and} \quad G_i^{-1} = \begin{bmatrix} 
    0 & -1 \\
    1 & \sqrt{m} 
\end{bmatrix}.
$$
By a direct computation, we obtain
$$e^{-t\tilde{T}_i(\gamma)} = G_i  \begin{bmatrix} 1 & -t \\ 
0 & 1
\end{bmatrix}G_i^{-1} = \begin{bmatrix}
1-t\sqrt{m} & -tm \\
t & 1+t\sqrt{m}
\end{bmatrix}.
$$
A simple computation reveals 
\begin{equation}\label{eq:exp-T-tilde}
\left\| e^{-t\tilde{T}_i(\gamma)}\right\|  \leq \sqrt{\mbox{Tr}\left(e^{-t\tilde{T}_i(\gamma)} e^{-t\tilde{T}_i(\gamma)^T}\right)}=\sqrt{2 + (m+1)^2 t^2 }.
\end{equation}
To finish the proof of \textbf{Case 2}, let $\gamma = 2\sqrt{m}$. We compute 
\begin{align*}
\max_{1\leq i \leq d} \left\| e^{-t\tilde{T}_i(\gamma)}\right\|  &= \max\left\{ \max_{i: \lambda_i = m} \left\| e^{-t\tilde{T}_i(\gamma)}\right\|  , \max_{i: \lambda_i > m} \left\| e^{-t\tilde{T}_i(\gamma)}\right\| \right\}\\
   &\leq  \max\left\{ \sqrt{2 + (m+1)^2 t^2 }, \max_{i: \lambda_i > m}  \frac{(1+\lambda_i)}{\sqrt{\lambda_i-m}}\right\},
\end{align*}
where we used \eqref{ineq-resolvent-bound} and \eqref{eq:exp-T-tilde} in the last inequality. We conclude from \eqref{eq-exp-in-t} for \textbf{Case 2}. 
\end{proof}

%%%%%%%%%%%%%%
%%%%%%%%%%%%%%

\subsection{Proof of Theorem \ref{MainThm1}}

The main result we use to prove Theorem \ref{MainThm1} is the following proposition. 
The proof of the following result will be presented later in Section \ref{proof:two:results}.

\begin{proposition}\label{prop:key}
Assume $\gamma=2\sqrt{m}$.
Fix any $r>0$ and 
\begin{equation*}
0<\varepsilon
<\min\left\{\overline{\varepsilon}_{1}^{U},\overline{\varepsilon}_{2}^{U},
\overline{\varepsilon}_{3}^{U}\right\},
\end{equation*}
where
\begin{align}
&\overline{\varepsilon}_{1}^{U}:=\sqrt{\frac{C_{H}+2+(m+1)^{2}}{(C_{H}+2)m+(m+1)^{2}}}r,\label{eqn:e1}
\\
&\overline{\varepsilon}_{2}^{U}:=2\sqrt{2}\left(C_{H}+2+(m+1)^{2}\right)^{1/4}\frac{e^{-1/2}r}{m^{1/4}},\label{eqn:e2}
\\
&\overline{\varepsilon}_{3}^{U}:=\frac{\sqrt{m}}{4L\left(\sqrt{C_{H}+2}+\frac{m+1}{\sqrt{m}}
+\frac{\sqrt{(C_{H}+2)m}+(m+1)}{8\sqrt{C_{H}+2+(m+1)^{2}}}\right)}.\label{eqn:e3}
\end{align}
Consider the stopping time:
\begin{equation*}
\tau:=\inf\left\{t\geq 0:\Vert X(t)-x_{\ast}\Vert\geq\varepsilon+re^{-\sqrt{m}t}\right\}.
\end{equation*}
For any initial point $X(0)=x$ with $\Vert x-x_{\ast}\Vert \le r$, and 
\begin{equation*}
\beta\geq\frac{256(2C_{H}m+4m+(m+1)^{2})}{m\varepsilon^{2}}
\left(d\log(2)+\log\left(\frac{2\Vert H_{2\sqrt{m}}\Vert \mathcal{T}+1}{\delta}\right)\right),
\end{equation*}
we have
\begin{equation*}
\mathbb{P}_{x}\left(\tau\in[\mathcal{T}_{\text{rec}}^{U},\mathcal{T}_{\text{esc}}^{U}]\right)\leq\delta.
\end{equation*}
\end{proposition}

%%%%%%%%%%%%%%%%%%%%%%%%%%%%%%%%%%%%%%%%%%%%
We are now ready to complete the proof of Theorem \ref{MainThm1}.

%%%%%%%%%%%%%%%%%%%%%%%%%%%%%%%%
\subsubsection{Completing the proof of Theorem \ref{MainThm1}}\label{sec:complete:proof}

Assume that $\gamma=2\sqrt{m}$. 
Let us compare the discrete dynamics \eqref{eq:V-iterate}-\eqref{eq:X-iterate} 
and the continuous dynamics \eqref{eq:VL}-\eqref{eq:XL}.
Define: 
\begin{align}
&\tilde{V}(t)=V_0-\int_{0}^{t}\gamma\tilde{V}(s)ds
-\int_{0}^{t}\nabla F\left(\tilde{X}\left(\lfloor s/\eta\rfloor\eta\right)\right)ds+\sqrt{2\gamma \beta^{-1}}\int_{0}^{t}dB_{s},\label{eqn:tilde:1}
\\
&\tilde{X}(t)=X_0+\int_{0}^{t}\tilde{V}(s)ds.\label{eqn:tilde:2}
\end{align}
The process $(\tilde{V}, \tilde{X})$ defined in \eqref{eqn:tilde:1} and \eqref{eqn:tilde:2} 
is the continuous-time interpolation of the iterates $\{(V_{k},X_{k}) \}$. In particular, 
the joint distribution of $\{(V_{k},X_{k}): k = 1, 2, \ldots, K\}$ is the same as $\{(\tilde{V}(t),\tilde{X}(t)): t = \eta, 2 \eta, \ldots, K \eta\}$ for any positive integer $K.$ 

{We can follow the the proof of Lemma 18 in \cite{GGZ} 
to provide an upper bound on the relative entropy $D(\cdot\Vert\cdot)$
between 
the law $\tilde{\mathbb{P}}^{K\eta}$ of $((\tilde{V}(t),\tilde{X}(t)):t\leq K\eta)$
and the law $\mathbb{P}^{K\eta}$ of $((V(t),X(t)):t\leq K\eta)$
as follows.
For the sake of convenience of the readers, we provide
a self-contained derivation here.
Let $\mathbb{P}$ be the probability measure associated with the underdamped Langevin diffusion $(X(t),V(t))$ in \eqref{eq:VL}--\eqref{eq:XL}
and $\tilde{\mathbb{P}}$ be the probability measure associated with the $(\tilde{X}(t),\tilde{V}(t))$ process defined in \eqref{eqn:tilde:1}-\eqref{eqn:tilde:2}.
Let $\mathcal{F}_{t}$ be the natural filtration up to time $t$.
Then, the Radon-Nikodym derivative of $\mathbb{P}$ w.r.t. $\tilde{\mathbb{P}}$
is given by the Girsanov formula (see e.g. Section 7.6 in \cite{liptser2013statistics}):
\begin{equation*}
\frac{d\mathbb{P}}{d\tilde{\mathbb{P}}}
\bigg|_{\mathcal{F}_{t}}
=e^{-\sqrt{\frac{\beta}{2\gamma}}\int_{0}^{t}(\nabla F(\tilde{X}(s))-\nabla F(\tilde{X}(\lfloor s/\eta\rfloor\eta)))\cdot dB(s)
-\frac{\beta}{4\gamma}\int_{0}^{t}\Vert\nabla F(\tilde{X}(s))-\nabla F(\tilde{X}(\lfloor s/\eta\rfloor\eta))\Vert^{2}ds}.
\end{equation*}
Then by writing $\mathbb{P}_{t}$ and $\tilde{\mathbb{P}}_{t}$
as the probability measures $\mathbb{P}$ and $\tilde{\mathbb{P}}$ conditional on the filtration $\mathcal{F}_{t}$, 
\begin{align*}
D(\tilde{\mathbb{P}}_{t}\Vert\mathbb{P}_{t})
:=-\int d\tilde{\mathbb{P}}_{t}\log\frac{d\mathbb{P}_{t}}{d\tilde{\mathbb{P}}_{t}}
=\frac{\beta}{4\gamma}\int_{0}^{t}\mathbb{E}\left\Vert\nabla F(\tilde{X}(s))-\nabla F(\tilde{X}(\lfloor s/\eta\rfloor\eta))\right\Vert^{2}ds.
\end{align*}
Then, we get
\begin{align}
D(\tilde{\mathbb{P}}_{k\eta}\Vert\mathbb{P}_{k\eta})
=\frac{\beta}{4\gamma}\sum_{j=0}^{k-1}
\int_{j\eta}^{(j+1)\eta}\mathbb{E}\left\Vert\nabla F(\tilde{X}(s))-\nabla F(\tilde{X}(\lfloor s/\eta\rfloor\eta))\right\Vert^{2}ds. 
\end{align}
Notice that for any $k\eta\leq s<(k+1)\eta$,
\begin{equation}
\tilde{X}(s)=\tilde{X}_{k}+\psi_{1}(s-k\eta)\tilde{V}_{k}-\psi_{2}(s-k\eta)\nabla F(\tilde{X}_{k})
+\sqrt{2\gamma\beta^{-1}}\xi'_{k+1,s-k\eta},
\end{equation}
in distribution, where $\xi'_{k+1,s-k\eta}$
is centered Gaussian independent of $\mathcal{F}_{k}$ and 
$\mathbb{E}\Vert\xi'_{k+1,s-k\eta}\Vert^{2}\leq\frac{d}{3}(s-k\eta)^{3}\leq\frac{d}{3}\eta^{3}$.
Moreover, $\psi_{1}(s-k\eta)=\int_{0}^{s-k\eta}e^{-\gamma t}dt\leq(s-k\eta)\leq\eta$,
and $\psi_{2}(s-k\eta)=\int_{0}^{s-k\eta}\psi_{1}(t)dt\leq\int_{0}^{s-k\eta}tdt\leq\eta^{2}$.
Therefore, we can compute that
\begin{align*}
&\frac{\beta}{4\gamma}\sum_{j=0}^{k-1}
\int_{j\eta}^{(j+1)\eta}\mathbb{E}\left\Vert\nabla F(\tilde{X}(s))-\nabla F(\tilde{X}(\lfloor s/\eta\rfloor\eta))\right\Vert^{2}ds
\\
&
\leq
\frac{\beta M^{2}}{4\gamma}\sum_{j=0}^{k-1}
\int_{j\eta}^{(j+1)\eta}\mathbb{E}\left\Vert\tilde{X}(s)-\tilde{X}(\lfloor s/\eta\rfloor\eta)\right\Vert^{2}ds
\\
&
=\frac{\beta M^{2}}{4\gamma}\sum_{j=0}^{k-1}
\int_{j\eta}^{(j+1)\eta}\mathbb{E}\left\Vert\psi_{1}(s-j\eta)\tilde{V}_{j}-\psi_{2}(s-j\eta)\nabla F(\tilde{X}_{j})
+\sqrt{2\gamma\beta^{-1}}\xi'_{j+1,s-j\eta}\right\Vert^{2}ds
\\
&\leq
\frac{3\beta M^{2}}{4\gamma}\sum_{j=0}^{k-1}
\int_{j\eta}^{(j+1)\eta}\bigg(\mathbb{E}\left\Vert\psi_{1}(s-j\eta)\tilde{V}_{j}\right\Vert^{2}
+\mathbb{E}\left\Vert\psi_{2}(s-j\eta)\nabla F(\tilde{X}_{j})\right\Vert^{2}
\\
&\qquad\qquad\qquad\qquad\qquad\qquad
+\mathbb{E}\left\Vert\sqrt{2\gamma\beta^{-1}}\xi'_{j+1,s-j\eta}\right\Vert^{2}\bigg)ds
\\
&\leq
\frac{3\beta M^{2}}{4\gamma}(k\eta)
\bigg(\eta^{2}\sup_{j\geq 0}\mathbb{E}\Vert\tilde{V}_{j}\Vert^{2}
+\eta^{4}\left(2M^{2}\sup_{j\geq 0}\mathbb{E}\Vert\tilde{X}_{j}\Vert^{2}
+2B^{2}\right)
+\frac{d\eta^{3}}{3}2\gamma\beta^{-1}\bigg)
\\
&\leq
\frac{3\beta M^{2}}{4\gamma}(k\eta)\eta^{2}
\bigg(C_{v}^{d}+\left(2M^{2}C_{x}^{d}+2B^{2}\right)+\frac{2d\gamma\beta^{-1}}{3}\bigg),
\end{align*}
provided that $\eta\leq\min\left\{1,\frac{\gamma}{\hat{K}_{2}}(d/\beta+\overline{A}/\beta),\frac{\gamma\lambda}{2\hat{K}_{1}},\frac{2}{\gamma\lambda}\right\}$,
with $C_{v}^{d}$ defined in Lemma \ref{lem:L2bound}, 
where we applied Lemma~\ref{lem:gradient-bound}.
Hence, we have
\begin{equation*}
D\left(\tilde{\mathbb{P}}^{K\eta}\Big\Vert\mathbb{P}^{K\eta}\right)
\leq\frac{3\beta M^{2}}{4\gamma}K\eta^{3}
\left(C_{v}^{d}+2M^{2}C_{x}^{d}+2B^{2}+\frac{2d\gamma\beta^{-1}}{3}\right).
\end{equation*}
Using Pinsker's inequality, we obtain an upper bound on the total variation $\Vert\cdot\Vert_{TV}$: 
\begin{equation*}
\left\Vert\tilde{\mathbb{P}}^{K\eta}-\mathbb{P}^{K\eta}\right\Vert_{TV}^{2}
\leq\frac{3\beta M^{2}}{8\gamma}K\eta^{3}
\left(C_{v}^{d}+2M^{2}C_{x}^{d}+2B^{2}+\frac{2d\gamma\beta^{-1}}{3}\right).
\end{equation*}}
Using a result about an optimal coupling (Theorem 5.2., \cite{Lindvall}),
that is, given any two random elements $\mathcal{X}, \mathcal{Y}$ of a common standard Borel space,
there exists a coupling $\mathcal{P}$ of $\mathcal{X}$ and $\mathcal{Y}$ such that
\begin{equation*}
\mathcal{P}(\mathcal{X}\neq \mathcal{Y})\leq\Vert\mathcal{L}(\mathcal{X})-\mathcal{L}( \mathcal{Y})\Vert_{TV}.
\end{equation*}
Hence, given any $\beta>0$ and $K \eta\leq\mathcal{T}_{\text{esc}}^{U}$,
we can choose 
\begin{equation}\label{eq:chooseeta}
\eta^{2}\leq\frac{{8\gamma\delta^{2}}}{3\beta M^{2}(C_{v}^{d}+2M^{2}C_{x}^{d}+2B^{2}+\frac{2d\gamma\beta^{-1}}{3})\mathcal{T}_{\text{esc}}^{U}},
\end{equation}
so that there is a coupling of $\left\{(V(k\eta),X(k\eta)):k=1,2,\ldots,K\right\}$
and $\left\{(V_{k},X_{k}):k=1,2,\ldots,K\right\}$ such that 
\begin{equation}\label{eq:chooseeta2}
\mathcal{P}(((V(\eta),X(\eta)),\ldots,(V(K\eta),X(K\eta)))
\neq
((V_{1},X_{1}),\ldots,(V_{K},X_{K}))
\leq\delta.
\end{equation}
It follows that
\begin{equation*}
\mathbb{P}(((V_{1},X_{1}),\ldots,(V_{K},X_{K}))\in\cdot)
\leq\mathbb{P}(((V(\eta),X(\eta)),\ldots,(V(K\eta),X(K\eta)))\in\cdot)+\delta.
\end{equation*}
%which proves Lemma \ref{lem:path-coupling}.

Let us now complete the proof of Theorem \ref{MainThm1}.
We need to show that
\begin{equation*}
\mathbb{P}\left(\left(X_{1},\ldots,X_{K}\right)\in\mathcal{A}\right)\leq\delta,
\end{equation*}
where $K=\lfloor\eta^{-1}\mathcal{T}_{\text{esc}}^{U}\rfloor$
and $\mathcal{A}:=\mathcal{A}_{1}\cap\mathcal{A}_{2}$, where 
\begin{align*}
&\mathcal{A}_{1}:=\left\{(x_{1},\ldots,x_{K})\in(\mathbb{R}^{d})^{K}:\max_{k\leq\eta^{-1}\mathcal{T}_{\text{rec}}^{U}}
\frac{\Vert x_{k}-x_{\ast}\Vert}{\varepsilon+re^{-\sqrt{m} k\eta}}\leq\frac{1}{2}\right\},
\\
&\mathcal{A}_{2}:=\left\{(x_{1},\ldots,x_{K})\in(\mathbb{R}^{d})^{K}:\max_{\eta^{-1}\mathcal{T}_{\text{rec}}^{U}\leq k\leq K}
\frac{\Vert x_{k}-x_{\ast}\Vert}{\varepsilon+re^{-\sqrt{m} k\eta}}\geq 1\right\}.
\end{align*}

We can choose $\beta$ sufficiently large so that with probability at least $1-\delta/3$,
we have either $\Vert X(t)-x_{\ast}\Vert\geq\varepsilon+re^{-\sqrt{m} t}$
for some $t\leq\mathcal{T}_{\text{rec}}^{U}$
or $\Vert X(t)-x_{\ast}\Vert\leq\varepsilon+re^{-\sqrt{m} t}$
for all $t\leq\mathcal{T}_{\text{esc}}^{U}$. 
Moreover, for any $K,\eta$ and $\beta$ satisfying the conditions
of the theorem, there exists a coupling of $(X(\eta),\ldots,X(K\eta))$
and $(X_{1},\ldots,X_{K})$ so that
with probability $1-\delta/3$, $X_{k}=X(k\eta)$ for all $k=1,2,\ldots,K$.
Then, by \eqref{eq:chooseeta} and \eqref{eq:chooseeta2}, we get
\begin{equation}\label{in:view:1}
\mathbb{P}((X_{1},\ldots,X_{K})\in\mathcal{A})
\leq\mathbb{P}((X(\eta),\ldots,X(K\eta))\in\mathcal{A})+\frac{\delta}{3},
\end{equation}
provided that
\begin{equation}\label{eqn:eta3}
\eta\leq\overline{\eta}_{3}^{U}:=\frac{{2\sqrt{2}}\gamma^{1/2}\delta}{3\sqrt{3\beta} M(C_{v}^{d}+2M^{2}C_{x}^{d}+2B^{2}+\frac{2d\gamma\beta^{-1}}{3})^{1/2}(\mathcal{T}_{\text{esc}}^{U})^{1/2}}.
\end{equation} 
It remains to estimate the probability of $\mathbb{P}((X(\eta),\ldots,X(K\eta))\in\mathcal{A}_{1}\cap\mathcal{A}_{2})$ 
for the underdamped Langevin diffusion.
Partition the interval $[0,\mathcal{T}_{\text{rec}}^{U}]$
using the points $0=t_{1}<t_{1}<\cdots<t_{\lceil\eta^{-1}\mathcal{T}_{\text{rec}}^{U}\rceil}=\mathcal{T}_{\text{rec}}^{U}$
with $t_{k}=k\eta$ for $k=0,1,\ldots,\lceil\eta^{-1}\mathcal{T}_{\text{rec}}^{U}\rceil-1$,
and consider the event: 
\begin{equation*}
\mathcal{B}:=\left\{\max_{0\leq k\leq \lceil\eta^{-1}\mathcal{T}_{\text{rec}}^{U}\rceil-1}\max_{t\in[t_{k},t_{k+1}]}\Vert X(t)-X(t_{k+1})\Vert\leq\frac{\varepsilon}{2}\right\}.
\end{equation*}
On the event $\left\{(X(\eta),\ldots,X(K\eta))\in\mathcal{A}_{1}\right\}\cap\mathcal{B}$,
\begin{align*}
\sup_{t\in[0,\mathcal{T}_{\text{rec}}^{U}]}
\frac{\Vert X(t)-x_{\ast}\Vert}{\varepsilon+re^{-\sqrt{m} t}}
&=\max_{0\leq k\leq \lceil\eta^{-1}\mathcal{T}_{\text{rec}}^{U}\rceil-1}\sup_{t\in[t_{k},t_{k+1}]}\frac{\Vert X(t)-x_{\ast}\Vert}{\varepsilon+re^{-\sqrt{m} t}}
\\
&\leq\frac{1}{2}+\max_{0\leq k\leq \lceil\eta^{-1}\mathcal{T}_{\text{rec}}^{U}\rceil-1}\max_{t\in[t_{k},t_{k+1}]}\frac{1}{\varepsilon}\Vert X(t)-X(t_{k+1})\Vert<1,
\end{align*}
and thus
\begin{align}
\mathbb{P}((X(\eta),\cdots,X(K\eta))\in\mathcal{A})
&\leq\mathbb{P}(\left\{(X(\eta),\cdots,X(K\eta))\in\mathcal{A}\right\}\cap\mathcal{B})
+\mathbb{P}(B^{c})
\nonumber
\\
&\leq
\mathbb{P}(\tau\in[\mathcal{T}_{\text{rec}}^{U},\mathcal{T}_{\text{esc}}^{U}])
+\mathbb{P}(\mathcal{B}^{c})
\nonumber
\\
&\leq
\frac{\delta}{3}+\mathbb{P}(\mathcal{B}^{c})\,,\label{in:view:2}
\end{align}
provided that (by applying Proposition \ref{prop:key} and Lemma \ref{lem:H:gamma}) (with $\gamma=2\sqrt{m}$): 
\begin{equation}\label{eqn:beta1}
\beta\geq\underline{\beta}_{1}^{U}:=
\frac{256(2C_{H}m+4m+(m+1)^{2})}{m\varepsilon^{2}}
\left(d\log(2)+\log\left(\frac{6\sqrt{4m+M^{2}+1}\mathcal{T}+3}{\delta}\right)\right).
\end{equation}

To complete the proof, we need to show that
$\mathbb{P}(\mathcal{B}^{c})\leq\frac{\delta}{3}$ in view of \eqref{in:view:1} and \eqref{in:view:2}.
For any $t\in[t_{k},t_{k+1}]$, where $t_{k+1} - t_k = \eta$, we have
\begin{equation}\label{addingtwo:1}
\Vert X(t)-X(t_{k+1})\Vert
\leq
\int_{t}^{t_{k+1}}\Vert V(s)\Vert ds
\leq\eta\Vert V(t_{k+1})\Vert
+\int_{t}^{t_{k+1}}\Vert V(s)-V(t_{k+1})\Vert ds,
\end{equation}
and
\begin{align}
&\Vert V(t)-V(t_{k+1})\Vert
\nonumber
\\
&\leq
\gamma\int_{t}^{t_{k+1}}\Vert V(s)\Vert ds
+\int_{t}^{t_{k+1}}\Vert\nabla F(X(s))\Vert ds
+\sqrt{2\gamma\beta^{-1}}\Vert B_{t}-B_{t_{k+1}}\Vert
\nonumber
\\
&\leq
\gamma\eta\Vert V(t_{k+1})\Vert
+\gamma\int_{t}^{t_{k+1}}\Vert V(s)-V(t_{k+1})\Vert ds
\nonumber
\\
&\qquad
+M\int_{t}^{t_{k+1}}\Vert X(s)-X(t_{k+1})\Vert ds
+\eta\Vert\nabla F(X(t_{k+1}))\Vert
+\sqrt{2\gamma\beta^{-1}}\Vert B_{t}-B_{t_{k+1}}\Vert
\nonumber
\\
&\leq
\gamma\eta\Vert V(t_{k+1})\Vert
+\gamma\int_{t}^{t_{k+1}}\Vert V(s)-V(t_{k+1})\Vert ds
\nonumber
\\
&\qquad
+M\int_{t}^{t_{k+1}}\Vert X(s)-X(t_{k+1})\Vert ds
+M\eta\Vert X(t_{k+1})\Vert+B\eta
+\sqrt{2\gamma\beta^{-1}}\Vert B_{t}-B_{t_{k+1}}\Vert\,,
\label{addingtwo:2}
\end{align}
where the second inequality above used $M$-Lipschitz property of $\nabla F$
and the last inequality above used Lemma \ref{lem:gradient-bound}.
By adding the above two inequalities \eqref{addingtwo:1} and \eqref{addingtwo:2} together, we get
\begin{align*}
&\Vert X(t)-X(t_{k+1})\Vert
+\Vert V(t)-V(t_{k+1})\Vert
\\
&\leq
(1+\gamma)\eta\Vert V(t_{k+1})\Vert
+(1+\gamma)\int_{t}^{t_{k+1}}\Vert V(s)-V(t_{k+1})\Vert ds
\\
&\qquad
+M\int_{t}^{t_{k+1}}\Vert X(s)-X(t_{k+1})\Vert ds
+M\eta\Vert X(t_{k+1})\Vert+B\eta
+\sqrt{2\gamma\beta^{-1}}\Vert B_{t}-B_{t_{k+1}}\Vert
\\
&\leq
(1+\gamma+M)\int_{t}^{t_{k+1}}\left(\Vert V(s)-V(t_{k+1})\Vert+\Vert X(s)-X(t_{k+1})\Vert\right)ds
\\
&\qquad
+(1+\gamma)\eta\Vert V(t_{k+1})\Vert+M\eta\Vert X(t_{k+1})\Vert
+B\eta
+\sqrt{2\gamma\beta^{-1}}\sup_{t\in[t_{k},t_{k+1}]}\Vert B_{t}-B_{t_{k+1}}\Vert.
\end{align*}
By applying Gronwall's inequality, we get
\begin{align}
&\sup_{t\in[t_{k},t_{k+1}]}\left[\Vert X(t)-X(t_{k+1})\Vert+\Vert V(t)-V(t_{k+1})\Vert\right]\nonumber
\\
&\leq
e^{(1+\gamma+M)\eta}
\left[(1+\gamma)\eta\Vert V(t_{k+1})\Vert+M\eta\Vert X(t_{k+1})\Vert+B\eta
+\sqrt{2\gamma\beta^{-1}}\sup_{t\in[t_{k},t_{k+1}]}\Vert B_{t}-B_{t_{k+1}}\Vert\right]. \label{eq:gronwall}
\end{align}
We have from Lemma~\ref{lem:L2bound} that for any $u>0$,
\begin{equation} 
\mathbb{P}(\Vert V(t_{k+1})\Vert\geq u)
\leq\frac{\sup_{t>0}\mathbb{E}\Vert V(t)\Vert^{2}}{u^{2}}
\leq\frac{C_{v}^{c}}{u^{2}}, \label{eq:L2-1}
\end{equation}
and
\begin{equation}
\mathbb{P}(\Vert X(t_{k+1})\Vert\geq u)
\leq\frac{\sup_{t>0}\mathbb{E}\Vert X(t)\Vert^{2}}{u^{2}}
\leq\frac{C_{x}^{c}}{u^{2}}, \label{eq:L2-2}
\end{equation}
where $C_{v}^{c}$, $C_{x}^{c}$ are defined in Lemma \ref{lem:L2bound}.
By Lemma \ref{lem:Brownian}, we have
\begin{equation*}
\mathbb{P}\left(\sup_{t\in[t_{k},t_{k+1}]}\Vert B_{t}-B_{t_{k+1}}\Vert\geq u\right)
\leq
2^{1/4}e^{1/4}e^{-\frac{u^{2}}{4d\eta}}.  
\end{equation*}
Therefore, we can infer from \eqref{eq:gronwall} that with $K_{0}:=\lceil\eta^{-1}\mathcal{T}_{\text{rec}}^{U}\rceil$,
\begin{align}
&\mathbb{P}\left(\mathcal{B}^{c}\right)
\nonumber
\\
&\leq
\sum_{k=0}^{K_{0}-1}
\mathbb{P}\left(\Vert X(t_{k+1})\Vert\geq\frac{\varepsilon e^{-(1+\gamma+M)\eta}}{8M\eta}\right)
+\sum_{k=0}^{K_{0}-1}
\mathbb{P}\left(\Vert V(t_{k+1})\Vert\geq\frac{\varepsilon e^{-(1+\gamma+M)\eta}}{8(1+\gamma)\eta}\right)
\nonumber
\\
&\qquad
+\sum_{k=0}^{K_{0}-1}
\mathbb{P}\left(B\geq\frac{\varepsilon e^{-(1+\gamma+M)\eta}}{8\eta}\right)
+\sum_{k=0}^{K_{0}-1}
\mathbb{P}\left(\sup_{t\in[t_{k},t_{k+1}]}\Vert B_{t}-B_{t_{k+1}}\Vert\geq\frac{\varepsilon e^{-(1+\gamma+M)\eta}\sqrt{\beta}}{8\sqrt{2\gamma}}\right)
\nonumber
\\
&\leq
\frac{64K_{0}}{\varepsilon^{2}}
\left(M^{2}C_{x}^{c}+(1+\gamma)^{2}C_{v}^{c}\right)\cdot\eta^{2}e^{2(1+\gamma+M)\eta}
\label{last:line:one}
\\
&\qquad\qquad\qquad
+2^{1/4}e^{1/4}K_{0}\cdot\exp\left(-\frac{1}{4d\eta}\frac{\varepsilon^{2}e^{-2(1+\gamma+M)\eta}\beta}{128\gamma}\right)
\label{last:line:two}
\\
&\qquad\qquad\qquad\qquad\qquad
+K_{0}\mathbb{P}\left(B\geq\frac{\varepsilon e^{-(1+\gamma+M)\eta}}{8\eta}\right)\,,
\label{last:line:three}
\end{align}
where the last inequality follows from \eqref{eq:L2-1}, \eqref{eq:L2-2} and Lemma \ref{lem:Brownian}. 
We can choose $\eta\leq 1$ so that
\begin{equation}\label{eqn:eta2}
\eta\leq
\overline{\eta}_{2}^{U}:=
\frac{\delta\varepsilon^{2}e^{-2(1+\gamma+M)}}{384(M^{2}C_{x}^{c}+(1+\gamma)^{2}C_{v}^{c})\mathcal{T}_{\text{rec}}^{U}}\,,
\end{equation}
so that the term in \eqref{last:line:one} is less than $\delta/6$,
where $C_{v}^{c}$, $C_{x}^{c}$ are defined in Lemma \ref{lem:L2bound},
and then we choose $\beta$ so that
\begin{equation}\label{eqn:beta2}
\beta\geq\underline{\beta}_{2}^{U}:=
\frac{512d\eta\gamma\log(2^{1/4}e^{1/4}6\delta^{-1}\mathcal{T}_{\text{rec}}^{U}/\eta)}
{\varepsilon^{2}e^{-2(1+\gamma+M)\eta}}\,,
\end{equation} 
so that the term in \eqref{last:line:two} is also less than $\delta/6$, and we can choose $\eta$ so that
$\eta\leq 1$ and
\begin{equation}\label{eqn:eta1}
\eta\leq\overline{\eta}_{1}^{U}
:=\frac{\varepsilon e^{-(1+\gamma+M)}}{8B}\,,
\end{equation}
so that the term in \eqref{last:line:three} is zero.

To complete the proof, let us work on the leading orders of the constants.
For the sake of convenience, we hide the dependence on $M$ and $L$
and assume that $M,L=\mathcal{O}(1)$.
We also assume that $C_{H}=\mathcal{O}(1)$.
Recall that $0<\varepsilon\leq\min\{\overline{\varepsilon}_{1}^{U},\overline{\varepsilon}_{2}^{U},\overline{\varepsilon}_{3}^{U}\}$, where it is easy to check that
It is easy to check that
\begin{equation*}
\overline{\varepsilon}_{1}^{U}=\sqrt{\frac{C_{H}+2+(m+1)^{2}}{(C_{H}+2)m+(m+1)^{2}}}r
\geq\Omega\left(\frac{C_{H}^{1/2}r}{C_{H}^{1/2}m^{1/2}+m+1}\right)
\geq\Omega(r),
\end{equation*}
where we used $m\leq M=\mathcal{O}(1)$ and
\begin{equation*}
\overline{\varepsilon}_{2}^{U}=2\sqrt{2}(C_{H}+2+(m+1)^{2})^{1/4}\frac{e^{-1/2}r}{m^{1/4}}
\geq\Omega\left(\frac{(1+C_{H}^{1/4})r}{m^{1/4}}\right)
\geq\Omega\left(\frac{r}{m^{1/4}}\right),
\end{equation*}
and
\begin{equation*}
\overline{\varepsilon}_{3}^{U}=\frac{\sqrt{m}}{4L\left(\sqrt{C_{H}+2}+\frac{m+1}{\sqrt{m}}
+\frac{\sqrt{(C_{H}+2)m}+(m+1)}{8\sqrt{C_{H}+2+(m+1)^{2}}}\right)}
\geq
\Omega\left(\frac{\sqrt{m}}{L\left(1+\frac{m+1}{\sqrt{m}}+\frac{\sqrt{m}}{m+1}\right)}\right)
\geq
\Omega(m),
\end{equation*}
where we used the fact that $m+1\geq 2\sqrt{m}$.
Hence, we can take
\begin{equation*}
\varepsilon\leq\min\left\{\mathcal{O}\left(r\right),
\mathcal{O}\left(\frac{r}{m^{1/4}}\right),\mathcal{O}(m)\right\}.
\end{equation*}
Moreover, $m\leq M=\mathcal{O}(1)$.
Hence, we can take
\begin{equation*}
\varepsilon\leq\min\left\{\mathcal{O}(r),\mathcal{O}(m)\right\}.
\end{equation*}

Next, we recall the recurrence time:
\begin{equation*}
\mathcal{T}_{\text{rec}}^{U}=-\frac{1}{\sqrt{m}}
W_{-1}\left(\frac{-\varepsilon^{2}\sqrt{m}}{8r^{2}\sqrt{C_{H}+2+(m+1)^{2}}}\right),
\end{equation*}
and since $W_{-1}(-x)\sim\log(1/x)$ for $x\rightarrow 0^{+}$, 
and we assume $C_{H}=\mathcal{O}(1)$, we get
\begin{equation*}
\mathcal{T}_{\text{rec}}^{U}=\mathcal{O}\left(\frac{1}{\sqrt{m}}\log\left(\frac{r}{\varepsilon m}\right)\right)
\leq\mathcal{O}\left(\frac{|\log(m)|}{\sqrt{m}}\log\left(\frac{r}{\varepsilon}\right)\right).
\end{equation*}

Next, we recall that stepsize $\eta$ satisfies 
$\eta\leq\min\{1,\overline{\eta}_{1}^{U},\overline{\eta}_{2}^{U},\overline{\eta}_{3}^{U},\overline{\eta}_{4}^{U}\}$
and it is easy to check that
\begin{equation*}
\overline{\eta}_{1}^{U}=\frac{\varepsilon e^{-(1+2\sqrt{m}+M)}}{8B}
\geq\Omega\left(\varepsilon e^{-(2m^{1/2}+M)}\right)
\geq\Omega(\varepsilon), 
\end{equation*}
and
\begin{equation*}
\overline{\eta}_{2}^{U}=\frac{\delta\varepsilon^{2}e^{-2(1+2\sqrt{m}+M)}}
{384(M^{2}C_{x}^{c}+(1+2\sqrt{m})^{2}C_{v}^{c})\mathcal{T}_{\text{rec}}^{U}}
\geq\Omega\left(\frac{\delta\varepsilon^{2}e^{-(4m^{1/2}+2M)}}
{(M^{2}C_{x}^{c}+(1+m)C_{v}^{c})\mathcal{T}_{\text{rec}}^{U}}\right).
\end{equation*}
Moreover, we have (note that $R=\sqrt{b/m}$ in the definition of $C_{x}^{c},C_{v}^{c}$)
\begin{equation*}
C_{x}^{c}
\leq\mathcal{O}\left(\frac{1+\frac{1}{m}+\frac{d}{\beta}}{m}\right),
\qquad
C_{v}^{c}
\leq\mathcal{O}\left(1+\frac{1}{m}+\frac{d}{\beta}\right),
\end{equation*}
together with $m\leq M=\mathcal{O}(1)$ implies that
\begin{equation*}
\overline{\eta}_{2}^{U}=\frac{\delta\varepsilon^{2}e^{-2(1+2\sqrt{m}+M)}}
{384(M^{2}C_{x}^{c}+(1+2\sqrt{m})^{2}C_{v}^{c})\mathcal{T}_{\text{rec}}^{U}}
\geq\Omega\left(\frac{m^{2}\beta\delta\varepsilon^{2}}
{(md+\beta)\mathcal{T}_{\text{rec}}^{U}}\right).
\end{equation*}
Moreover, 
\begin{equation*}
\overline{\eta}_{3}^{U}
=
\frac{{4m^{1/4}\delta}}{3\sqrt{3\beta} M(C_{v}^{d}+2M^{2}C_{x}^{d}+2B^{2}+\frac{4d\sqrt{m}\beta^{-1}}{3})^{1/2}(\mathcal{T}_{\text{esc}}^{U})^{1/2}}
\geq
\Omega\left(\frac{m^{5/4}\delta}{(d+\beta)^{1/2}(\mathcal{T}_{\text{esc}}^{U})^{1/2}}\right),
\end{equation*}
where we used $C_{x}^{d}\leq\mathcal{O}\left(\frac{d+\beta}{\beta m^{2}}\right)$
and $C_{v}^{d}\leq\mathcal{O}\left(\frac{d+\beta}{\beta m}\right)$,
and
\begin{equation*}
\overline{\eta}_{4}^{U}=\min\left\{1,\frac{2\sqrt{m}}{\hat{K}_{2}}\frac{d+\overline{A}}{\beta},
\frac{\sqrt{m}\lambda}{\hat{K}_{1}}\right\}
\geq\min\left\{\Omega\left(\frac{m^{1/2}(d+\beta)}{dm^{1/2}+\beta}\right),\Omega(m^{5/2})\right\}\,,
\end{equation*}
where we used $\lambda=\Omega(m)$, $\overline{A}=\Omega(\beta)$,
$K_{1}=\mathcal{O}(\frac{1}{\beta m})$, 
$K_{2}=\mathcal{O}(1)$, 
$\hat{K}_{1}=\mathcal{O}(\frac{1}{m})$,
$\hat{K}_{2}=\mathcal{O}(1+\frac{d}{\beta}\sqrt{m})$,
and the minimum between $\frac{m^{1/2}(d+\beta)}{dm^{1/2}+\beta}$ and $m^{5/2}$
is $m^{5/2}$. Hence, we can take
\begin{equation*}
\eta\leq\min\left\{\mathcal{O}(\varepsilon),
\mathcal{O}\left(\frac{m^{2}\beta\delta\varepsilon^{2}}
{(md+\beta)\mathcal{T}_{\text{rec}}^{U}}\right),
\mathcal{O}\left(\frac{m^{5/4}\delta}{(d+\beta)^{1/2}(\mathcal{T}_{\text{esc}}^{U})^{1/2}}\right),
\mathcal{O}(m^{5/2})\right\}\,.
\end{equation*}

Finally, $\beta$ satisfies $\beta\geq\max\{\underline{\beta}_{1}^{U},\underline{\beta}_{2}^{U}\}$,
and We have
\begin{align*}
\underline{\beta}_{1}^{U}&=\frac{256(2C_{H}m+4m+(m+1)^{2})}{m\varepsilon^{2}}\left(d\log(2)
+\log\left(\frac{6(4m + M^{2} + 1)^{1/2}\mathcal{T}+3}{\delta}\right)\right)
\\
&\leq
\mathcal{O}\left(\frac{d+\log((\mathcal{T}+1)/\delta)}{m\varepsilon^{2}}\right),
\end{align*}
and
\begin{equation*}
\underline{\beta}_{2}^{U}=\frac{1024d\eta\sqrt{m}\log(2^{1/4}e^{1/4}6\delta^{-1}\mathcal{T}_{\text{rec}}^{U}/\eta)}
{\varepsilon^{2}e^{-2(1+2\sqrt{m}+M)\eta}}
\leq
\mathcal{O}\left(\frac{d\eta m^{1/2}\log(\delta^{-1}\mathcal{T}_{\text{rec}}^{U}/\eta)}{\varepsilon^{2}}\right),
\end{equation*}
where we used $e^{2(1+2\sqrt{m}+M)\eta}=e^{\mathcal{O}(\varepsilon)}=\mathcal{O}(1)$.

Hence, we can take
\begin{equation*}
\beta\geq
\max\left\{\Omega\left(\frac{d+\log((\mathcal{T}+1)/\delta)}{m\varepsilon^{2}}\right),
\Omega\left(\frac{d\eta m^{1/2}\log(\delta^{-1}\mathcal{T}_{\text{rec}}^{U}/\eta)}{\varepsilon^{2}}\right)\right\}.
\end{equation*}

The proof is now complete.

%%%%%%%%%%%%%%%%%%%%%%%%%%%%%%%%%%%%%%%%%
\subsubsection{Proof of Proposition~\ref{prop:key}}\label{proof:two:results}

In this section, we focus on the proof of Proposition \ref{prop:key}.
We adopt some ideas from \cite{Berglund03,pmlr-v75-tzen18a}.
We recall $x_{\ast}$ is a local minimum of $F$ and $H$ is the Hessian matrix: 
$H=\nabla^{2}F(x_{\ast})$, and we write 
\begin{equation*}
X(t)=Y(t)+x_{\ast}.
\end{equation*}
Thus, we have the decomposition
\begin{equation*}
\nabla F(X(t))=HY(t)-\rho(Y(t)),
\end{equation*}
where $\Vert\rho(Y(t))\Vert\leq\frac{1}{2}L\Vert Y(t)\Vert^{2}$
since the Hessian of $F$ is $L$-Lipschitz (Lemma 1.2.4. \cite{Nesterov}).
Then, we have
\begin{align*}
&dV(t)=-\gamma V(t)dt- (H(Y(t))-\rho(Y(t)))dt+\sqrt{2\gamma \beta^{-1}}dB_{t},
\\
&dY(t)=V(t)dt.
\end{align*}
We can write it in terms of matrix form as:
\begin{equation*}
d
\left[
\begin{array}{c}
V(t)
\\
Y(t)
\end{array}
\right]
=
\left[
\begin{array}{cc}
-\gamma I & -H
\\
I & 0
\end{array}
\right]
\left[
\begin{array}{c}
V(t)
\\
Y(t)
\end{array}
\right]
dt
+
\sqrt{2\gamma\beta^{-1}}\left[
\begin{array}{cc}
I & 0
\\
0 & 0
\end{array}
\right]
dB_{t}^{(2)}
+
\left[
\begin{array}{c}
\rho(V(t))
\\
0
\end{array}
\right]dt,
\end{equation*}
where $B_{t}^{(2)}$ is a $2d$-dimensional standard Brownian motion. 
Therefore, we have
\begin{equation*}
\left[
\begin{array}{c}
V(t)
\\
Y(t)
\end{array}
\right]
=e^{-tH_{\gamma}}
\left[
\begin{array}{c}
V(0)
\\
Y(0)
\end{array}
\right]
+\sqrt{2\gamma\beta^{-1}}
\int_{0}^{t}e^{(s-t)H_{\gamma}}I^{(2)}dB_{s}^{(2)}
+\int_{0}^{t}e^{(s-t)H_{\gamma}}
\left[
\begin{array}{c}
\rho(V(s))
\\
0
\end{array}
\right]ds,
\end{equation*}
where
\begin{equation}\label{eqn:H:I}
H_{\gamma}=
\left[
\begin{array}{cc}
\gamma I & H
\\
-I & 0
\end{array}
\right],
\qquad
I^{(2)}=\left[
\begin{array}{cc}
I & 0
\\
0 & 0
\end{array}
\right].
\end{equation}

Given $0\leq t_{0}\leq t_{1}$, 
we define the matrix flow 
\begin{equation}\label{eqn:matrix:flow}
Q_{t_{0}}(t):=e^{(t_{0}-t)H_{\gamma}}
\end{equation}
and we also define
\begin{equation*}
Z(t):=e^{(t-t_{0})H_{\gamma}}\left[
\begin{array}{c}
V(t)
\\
Y(t)
\end{array}
\right]
=Z_{t}^{0}+Z_{t}^{1},
\end{equation*}
where
\begin{align}
&Z_{t}^{0}=e^{-t_{0}H_{\gamma}}
\left[
\begin{array}{c}
V(0)
\\
Y(0)
\end{array}
\right]
+\sqrt{2\gamma\beta^{-1}}
\int_{0}^{t}e^{(s-t_{0})H_{\gamma}}I^{(2)}dB_{s}^{(2)},\label{eqn:Z0}
\\
&Z_{t}^{1}=\int_{0}^{t}e^{(s-t_{0})H_{\gamma}}
\left[
\begin{array}{c}
\rho(V(s))
\\
0
\end{array}
\right]ds.\label{eqn:Z1}
\end{align}
Note that 
\begin{equation*}
Q_{t_{0}}(t_{1})Z_{t}^{0}
=e^{-t_{1}H_{\gamma}}
\left[
\begin{array}{c}
V(0)
\\
Y(0)
\end{array}
\right]
+\sqrt{2\gamma\beta^{-1}}
\int_{0}^{t}e^{(s-t_{1})H_{\gamma}}I^{(2)}dB_{s}^{(2)}
\end{equation*}
is a martingale. 
Before we proceed to the proof of Proposition \ref{prop:key}, 
we state the following lemma, which will be used in the proof of Proposition \ref{prop:key}.

\begin{lemma}\label{lem:ineq}
Assume $\gamma=2\sqrt{m}$. Define:
\begin{align}
&\mu_{t}:=e^{-tH_{\gamma}}(V(0),Y(0))^{T},\label{eqn:mu:t}
\\
&\Sigma_{t}:= 2\gamma\beta^{-1}\int_{0}^{t}e^{(s-t)H_{\gamma}}I^{(2)}e^{(s-t)H_{\gamma}^{T}}ds.\label{eqn:Sigma:t}
\end{align}
For any $\theta \in \left(0,\frac{2m\sqrt{m}}{\gamma(2C_{H}m+4m+(m+1)^{2})}\right)$,
and $h>0$ and any $(V(0),Y(0))$,
\begin{align*}
&\mathbb{P}\left(\sup_{t_{0}\leq t\leq t_{1}}\Vert Q_{t_{0}}(t_{1})Z_{t}^{0}\Vert\geq h\right)
\\
&\leq \left(1-\theta\frac{\gamma(2C_{H}m+4m+(m+1)^{2})}{2m\sqrt{m}}\right)^{-d} 
e^{-\frac{\beta\theta}{2}[h^{2}-\langle\mu_{t_{1}},(I-\beta\theta\Sigma_{t_{1}})^{-1}\mu_{t_{1}}\rangle]}.
\end{align*}
\end{lemma}

%Notice that $Z_{t}^{0}$ process is exactly the Ornstein-Uhlenbeck process \eqref{eqn:OU}
%that we have studied in Section \ref{sec:quadratic}. There, the emphasis is the convergence speed
%of this Ornstein-Uhlenbeck process as time $t$ goes to infinity, and the above Lemma \ref{lem:ineq}
%is about the tail estimate on a finite time interval.

Finally, let us complete the proof of Proposition \ref{prop:key}.

%%%%%%%%%%%%
\begin{proof}[Proof of Proposition~\ref{prop:key}]
Since $\Vert Y(0)\Vert=\Vert X(0)-x_{\ast}\Vert\leq r$, we know that $\tau>0$. 
Fix some $\mathcal{T}_{\text{rec}}^{U}\leq t_{0}\leq t_{1}$,
such that $t_{1}-t_{0}\leq\frac{1}{2\Vert H_{\gamma}\Vert}$.
Then, for every $t\in[t_{0},t_{1}]$,
\begin{equation*}
\Vert Y(t)\Vert
\leq\left\Vert e^{(t_{1}-t)H_{\gamma}}Q_{t_{0}}(t_{1})Z_{t}\right\Vert
\leq e^{\frac{1}{2}}\left\Vert Q_{t_{0}}(t_{1})Z_{t}\right\Vert.
\end{equation*}

It follows that (with $e^{-1/2}\geq 1/2$)
\begin{align}
&\mathbb{P}(\tau\in[t_{0},t_{1}])
\nonumber
\\
&=\mathbb{P}\left(\sup_{t_{0}\leq t\leq t_{1}\wedge\tau}\frac{\Vert Y(t)\Vert}{\varepsilon+re^{-\sqrt{m} t}}\geq 1,\tau\geq t_{0}\right)
\nonumber
\\
&\leq\mathbb{P}\left(\sup_{t_{0}\leq t\leq t_{1}\wedge\tau}\frac{\Vert Q_{t_{0}}(t_{1})Z_{t}\Vert}
{\varepsilon+re^{-\sqrt{m} t}}\geq\frac{1}{2},\tau\geq t_{0}\right)
\nonumber
\\
&\leq
\mathbb{P}\left(\sup_{t_{0}\leq t\leq t_{1}\wedge\tau}\frac{\Vert Q_{t_{0}}(t_{1})Z_{t}^{0}\Vert}
{\varepsilon+re^{-\sqrt{m} t}}\geq c_{0},\tau\geq t_{0}\right)
+\mathbb{P}\left(\sup_{t_{0}\leq t\leq t_{1}\wedge\tau}\frac{\Vert Q_{t_{0}}(t_{1})Z_{t}^{1}\Vert}
{\varepsilon+re^{-\sqrt{m} t}}\geq c_{1},\tau\geq t_{0}\right),
\label{bound:two:terms}
\end{align}
where $c_{0}+c_{1}=\frac{1}{2}$ and $c_{0},c_{1}>0$.
We will first bound the second term in \eqref{bound:two:terms} which will turn out
to be zero, and then use Lemma \ref{lem:ineq} to bound the first term in \eqref{bound:two:terms}.

First, notice that $Z_{t}^{1}\equiv 0$ in the quadratic case %(Sec. \ref{sec:quadratic})
and the second term in \eqref{bound:two:terms} is automatically zero.
In the more general case, we will show that the second term in \eqref{bound:two:terms}
is also zero.
On the event $\tau\in[t_{0},t_{1}]$, for any $0\leq s\leq t_{1}\wedge\tau$,
we have
\begin{equation*}
\Vert\rho(Y(s))\Vert\leq\frac{L}{2}\Vert Y(s)\Vert^{2}
\leq\frac{L}{2}\left(\varepsilon+re^{-\sqrt{m}s}\right)^{2}.
\end{equation*}
Therefore, for any $t\in[t_{0},t_{1}\wedge\tau]$, by Lemma \ref{lem:keylemma}, we get
\begin{align*}
&\left\Vert Q_{t_{0}}(t_{1})Z_{t}^{1}\right\Vert
\\
&\leq\int_{0}^{t}\left\Vert e^{ (s-t_{1}) H_{\gamma}}\right\Vert\cdot\Vert\rho(Y(s))\Vert ds
\\
&\leq\frac{L}{2}\int_{0}^{t}\sqrt{C_{H}+2+(m+1)^{2}(t_{1}-s)^{2}}e^{ (s-t_{1})\sqrt{m}}
\left(\varepsilon+re^{-\sqrt{m} s}\right)^{2}ds
\\
&\leq L\int_{0}^{t}\left(\sqrt{C_{H}+2}+(m+1)(t_{1}-s)\right)e^{(s-t_{1})\sqrt{m}}\left(\varepsilon^{2}+r^{2}e^{-2\sqrt{m}s}\right)ds
\\
&\leq L\int_{0}^{t_{1}}\left(\sqrt{C_{H}+2}+(m+1)(t_{1}-s)\right)e^{(s-t_{1})\sqrt{m}}\left(\varepsilon^{2}+r^{2}e^{-2\sqrt{m}s}\right)ds
\\
&\leq\frac{L}{\sqrt{m}}\left(\left(\sqrt{C_{H}+2}+\frac{m+1}{\sqrt{m}}\right)\varepsilon^{2}
+\sqrt{C_{H}+2}r^{2}e^{-\sqrt{m} t_{1}}\right)
\\
&\qquad\qquad\qquad\qquad
+L(m+1)r^{2}\int_{0}^{t_{1}}(t_{1}-s)e^{(s-t_{1})\sqrt{m}}e^{-2\sqrt{m}s}ds
\\
&\leq\frac{L}{\sqrt{m}}\left(\left(\sqrt{C_{H}+2}+\frac{m+1}{\sqrt{m}}\right)\varepsilon^{2}
+\sqrt{C_{H}+2}r^{2}e^{-\sqrt{m} t_{1}}+(m+1)r^{2}t_{1}e^{-t_{1}\sqrt{m}}\right)
\\
&\leq\frac{L}{\sqrt{m}}\left(\left(\sqrt{C_{H}+2}+\frac{m+1}{\sqrt{m}}\right)\varepsilon^{2}
+\left(\sqrt{(C_{H}+2)m}+(m+1)\right)r^{2}t_{1}e^{-t_{1}\sqrt{m}}\right)
\\
&\leq\frac{L}{\sqrt{m}}\left(\sqrt{C_{H}+2}+\frac{m+1}{\sqrt{m}}
+\frac{\sqrt{(C_{H}+2)m}+(m+1)}{8\sqrt{C_{H}+2+(m+1)^{2}}}\right)\varepsilon^{2}
\end{align*}
where we used $t_{1}\geq t\geq t_{0}\geq\mathcal{T}_{\text{rec}}^{U}\geq\frac{1}{\sqrt{m}}$, 
and $t_{1}e^{-t_{1}\sqrt{m}}\leq\mathcal{T}_{\text{rec}}^{U}e^{-\mathcal{T}_{\text{rec}}^{U}\sqrt{m}}$
and the definition of $\mathcal{T}_{\text{rec}}^{U}$:
\begin{equation*}
\sqrt{C_{H}+2+(m+1)^{2}}\mathcal{T}_{\text{rec}}^{U}e^{-\sqrt{m}\mathcal{T}_{\text{rec}}^{U}}
=\frac{\varepsilon^{2}}{8r^2}.
\end{equation*}
Consequently, if we take $c_{1}=\frac{L}{\sqrt{m}}\left(\sqrt{C_{H}+2}+\frac{m+1}{\sqrt{m}}
+\frac{\sqrt{(C_{H}+2)m}+(m+1)}{8\sqrt{C_{H}+2+(m+1)^{2}}}\right)\varepsilon$, 
then, 
\begin{equation*}
\sup_{t_{0}\leq t\leq t_{1}\wedge\tau}\frac{\Vert Q_{t_{0}}(t_{1})Z_{t}\Vert}{\varepsilon+re^{-\sqrt{m} t}}
\leq\frac{1}{\varepsilon}\sup_{t_{0}\leq t\leq t_{1}\wedge\tau}\Vert Q_{t_{0}}(t_{1})Z_{t}\Vert\leq c_{1},
\end{equation*}
which implies that
\begin{equation*}
\mathbb{P}\left(\sup_{t_{0}\leq t\leq t_{1}\wedge\tau}\frac{\Vert Q_{t_{0}}(t_{1})Z_{t}^{1}\Vert}
{\varepsilon+re^{-\sqrt{m} t}}\geq c_{1},\tau\geq t_{0}\right)=0.
\end{equation*}
Moreover, $c_{0}=\frac{1}{2}-c_{1}=\frac{1}{2}-\frac{L}{\sqrt{m}}\left(\sqrt{C_{H}+2}+\frac{m+1}{\sqrt{m}}
+\frac{\sqrt{(C_{H}+2)m}+(m+1)}{8\sqrt{C_{H}+2+(m+1)^{2}}}\right)\varepsilon>\frac{1}{4}$ since it is assumed 
that $\varepsilon <\frac{\sqrt{m}}{4L\left(\sqrt{C_{H}+2}+\frac{m+1}{\sqrt{m}}
+\frac{\sqrt{(C_{H}+2)m}+(m+1)}{8\sqrt{C_{H}+2+(m+1)^{2}}}\right)}$. 

Second, we will apply Lemma \ref{lem:ineq} to bound the first term in \eqref{bound:two:terms}.
By using $V(0)=0$ and $\Vert Y(0)\Vert\leq r$ and the definition
of $\mu_{t_{1}}$ and $\Sigma_{t_{1}}$ in \eqref{eqn:mu:t} and \eqref{eqn:Sigma:t},
we get
\begin{align*}
&\left\langle\mu_{t_{1}},(I-\beta\theta\Sigma_{t_{1}})^{-1}\mu_{t_{1}}\right\rangle
\\
&=
\left\langle e^{-t_{1}H_{\gamma}}(V(0),Y(0))^{T},(I-\beta\theta\Sigma_{t_{1}})^{-1}e^{-t_{1}H_{\gamma}}(V(0),Y(0))^{T}\right\rangle
\\
&\leq\left(1-\theta\frac{\gamma(2C_{H}m+4m+(m+1)^{2})}{2m\sqrt{m}}\right)^{-1}
\left(C_{H}+2+(m+1)^{2}t_{1}^{2}\right)e^{-2\sqrt{m}t_{1}}r^{2}
\\
&\leq
2\left((C_{H}+2)m+(m+1)^{2}\right)t_{1}^{2}e^{-2\sqrt{m}t_{1}}r^{2}
\\
&\leq\frac{1}{32}\frac{(C_{H}+2)m+(m+1)^{2}}{C_{H}+2+(m+1)^{2}}\frac{\varepsilon^{4}}{r^{2}}
\leq\frac{1}{32}\varepsilon^{2},
\end{align*}
by choosing $\theta=\frac{m\sqrt{m}}{\gamma(2C_{H}m+4m+(m+1)^{2})}$ and $t_{1}\geq\mathcal{T}_{\text{rec}}^{U}\geq\frac{1}{\sqrt{m}}$,
and $t_{1}e^{-t_{1}\sqrt{m}}\leq\mathcal{T}_{\text{rec}}^{U}e^{-\mathcal{T}_{\text{rec}}^{U}}$,
and using the definition
$\sqrt{C_{H}+2+(m+1)^{2}}\mathcal{T}_{\text{rec}}^{U}e^{-\sqrt{m}\mathcal{T}_{\text{rec}}^{U}}
=\frac{\varepsilon^{2}}{8r^{2}}$, and we also used 
$\varepsilon\leq\sqrt{\frac{C_{H}+2+(m+1)^{2}}{(C_{H}+2)m+(m+1)^{2}}}r$.

Then with the choice of $h=(\varepsilon+re^{-\sqrt{m}t_{1}})c_{0}$
and $\theta=\frac{m\sqrt{m}}{\gamma(2C_{H}m+4m+(m+1)^{2})}$ in Lemma \ref{lem:ineq}, and
using the fact that $h=(\varepsilon+re^{-\sqrt{m}t_{1}})c_{0}\geq\varepsilon c_{0}$, we get 
\begin{align*}
&\mathbb{P}\left(\sup_{t_{0}\leq t\leq t_{1}\wedge\tau}\frac{\Vert Q_{t_{0}}(t_{1})Z_{t}^{0}\Vert}
{\varepsilon+re^{-\sqrt{m} t}}\geq c_{0},\tau\geq t_{0}\right)
\\
&\leq
\mathbb{P}\left(\sup_{t_{0}\leq t\leq t_{1}}\Vert Q_{t_{0}}(t_{1})Z_{t}^{0}\Vert
\geq\left(\varepsilon+re^{-\sqrt{m} t_{1}}\right)c_{0}\right)
\\
&\leq\left(1-\theta\frac{\gamma(2C_{H}m+4m+(m+1)^{2})}{2m\sqrt{m}}\right)^{-\frac{2d}{2}} 
\cdot\exp\left(-\frac{\beta\theta}{2}\left[h^{2}-\langle\mu_{t_{1}},(I-\beta\theta\Sigma_{t_{1}})^{-1}\mu_{t_{1}}\rangle\right]\right)
\\
&\leq
2^{d}\cdot\exp\left(-\frac{\beta\gamma^{-1}m\sqrt{m}\varepsilon^{2}}{2(2C_{H}+4m+(m+1)^{2})}
\left(c_{0}^{2}-\frac{1}{32}\right)\right)
\\
&\leq
2^{d}\cdot\exp\left(-\frac{\beta\gamma^{-1}m\sqrt{m}\varepsilon^{2}}{128(2C_{H}+4m+(m+1)^{2})}\right).
\end{align*}
Thus for any $t_{0}\geq\mathcal{T}_{\text{rec}}^{U}$ and $t_{0}\leq t_{1}\leq t_{0}+\frac{1}{2\Vert H_{\gamma}\Vert}$, 
\begin{equation*}
\mathbb{P}(\tau\in[t_{0},t_{1}])
\leq
2^{d}\cdot\exp\left(-\frac{\beta\gamma^{-1}m\sqrt{m}\varepsilon^{2}}{128(2C_{H}m+4m+(m+1)^{2})}
\right).
\end{equation*}
Fix any $\mathcal{T}>0$ and recall the definition of the escape time $\mathcal{T}_{\text{esc}}^{U}=\mathcal{T}+\mathcal{T}_{\text{rec}}^{U}$.
Partition the interval $[\mathcal{T}_{\text{rec}}^{U},\mathcal{T}_{\text{esc}}^{U}]$ using the points
$\mathcal{T}_{\text{rec}}^{U}=t_{0}<t_{1}<\cdots<t_{\lceil 2\Vert H_{\gamma}\Vert \mathcal{T}\rceil}=\mathcal{T}_{\text{esc}}^{U}$
with $t_{j}=j/(2\Vert H_{\gamma}\Vert)$, then we have
\begin{align*}
\mathbb{P}\left(\tau\in\left[\mathcal{T}_{\text{rec}}^{U},\mathcal{T}_{\text{esc}}^{U}\right]\right)
&=\sum_{j=0}^{\lceil 2\Vert H_{\gamma}\Vert \mathcal{T}\rceil}\mathbb{P}(\tau\in[t_{j},t_{j+1}])
\\
&\leq(2\Vert H_{\gamma}\Vert \mathcal{T}+1)\cdot 2^{d}
\cdot\exp\left(-\frac{\beta\gamma^{-1}m\sqrt{m}\varepsilon^{2}}{128(2C_{H}m+4m+(m+1)^{2})}\right)\leq\delta,
\end{align*}
provided that
\begin{equation*}
\beta\geq\frac{128(2C_{H}m+4m+(m+1)^{2})\gamma }{m\sqrt{m}\varepsilon^{2}}\left(d\log(2)+\log\left(\frac{2\Vert H_{\gamma}\Vert \mathcal{T}+1}{\delta}\right)\right).
\end{equation*}
Finally, plugging $\gamma=2\sqrt{m}$ into the above formulas and applying 
the bound on $\Vert H_{\gamma}$ from Lemma \ref{lem:H:gamma}, 
the conclusion follows.
\end{proof}

%%%%%%%%%%%%%%%%%%%%%%%%%%%%%%%%%%%%%%%%%%
\subsubsection{Uniform $L^2$ bounds for underdamped Langevin dynamics}

In this section, we state the uniform $L^{2}$ bounds for the continuous time
underdamped Langevin dynamics (\eqref{eq:VL} and \eqref{eq:XL}) 
and the discrete time iterates (\eqref{eq:V-iterate} and \eqref{eq:X-iterate}) 
in Lemma \ref{lem:L2bound}, which is a modification of Lemma 8 in \cite{GGZ}.
The uniform $L^{2}$ bound for the discrete dynamics \eqref{eq:V-iterate}-\eqref{eq:X-iterate} 
is used to derive the relative entropy
to compare the laws of the continuous time dynamics and the discrete time dynamics,
and the uniform $L^{2}$ bound for the continuous dynamics \eqref{eq:VL}-\eqref{eq:XL} 
is used to control
the tail of the continuous dynamics in Section \ref{sec:complete:proof}.

Before we proceed, let us first introduce
the following Lyapunov function (from the paper \cite{Eberle}) 
which will be used 
in the proof the uniform $L^{2}$ boundedness results
for both the continuous and discrete underdamped Langevin dynamics.
We define the Lyapunov function $\mathcal{V}$ as:
\begin{equation} \label{eq:lyapunov}
\mathcal{V}(x,v):=\beta F(x)
+\frac{\beta}{4}\gamma^{2}\left(\Vert x+\gamma^{-1}v\Vert^{2}+\Vert\gamma^{-1}v\Vert^{2}-\lambda\Vert x\Vert^{2}\right)\,,
\end{equation}
and $\lambda$ is a positive constant less than $1/4$ according to \cite{Eberle}.
We will first show in the following lemma
that we can find explicit constants $\lambda\in(0,\min(1/4,m/(M+\gamma^{2}/2)))$ and $\overline{A}\in(0,\infty)$ so that
the drift condition \eqref{eq:drift}
is satisfied. The drift condition
is needed in \cite{Eberle}, 
which is applied to obtain the uniform 
$L^{2}$ bounds in \cite{GGZ} 
that implies
the uniform $L^{2}$ bounds
in our current setting (the following Lemma \ref{lem:L2bound}).

\begin{lemma}\label{lem:lambda:A}
Let us define: 
\begin{align}
&\lambda=\frac{1}{2}\min(1/4,m/(M+\gamma^{2}/2)),\label{eqn:lambda}
\\
&\overline{A}=\frac{\beta}{2}\frac{m}{M+\frac{1}{2}\gamma^{2}}
\left(\frac{B^{2}}{2M+\gamma^{2}}
+\frac{b}{m}\left(M+\frac{1}{2}\gamma^{2}\right)+A\right),
\label{eqn:overline:A}
\end{align}
then the following drift condition holds:
\begin{align}
x\cdot\nabla F(x)
\geq 2\lambda( F(x)+ \gamma^{2}\Vert x\Vert^{2}/4)-2\overline{A}/\beta\,.
\label{eq:drift}
\end{align}
\end{lemma}

The following lemma provides uniform $L^{2}$ bounds for the continuous-time
underdamped Langevin diffusion process $(X(t),V(t))$ defined in \eqref{eq:VL}-\eqref{eq:XL} 
and discrete-time underdamped Langevin dynamics $(X_{k},V_{k})$ defined in \eqref{eq:V-iterate}-\eqref{eq:X-iterate}.

\begin{lemma}[Uniform $L^{2}$ bounds]\label{lem:L2bound}
Suppose parts $(i)$, $(ii)$, $(iii)$, $(iv)$ of Assumption \ref{assumptions}
and the drift condition \eqref{eq:drift} hold.
$\gamma>0$ is arbitrary and $\lambda$, $\overline{A}$ are defined in \eqref{eqn:lambda} and \eqref{eqn:overline:A}.
\begin{itemize} 
\item [$(i)$] It holds that
\begin{align}
&\sup_{t\geq 0}\mathbb{E}\Vert X(t)\Vert^{2}\leq 
C_{x}^{c}
:=\frac{\left(\frac{\beta M}{2}+\frac{\beta\gamma^{2}(2-\lambda)}{4}\right)R^{2}+\beta BR+\beta A
+\frac{3}{4}\beta\Vert V(0)\Vert^{2}+\frac{d+\overline{A}}{\lambda}}{\frac{1}{8} (1-2 \lambda) \beta \gamma^2}, \label{def-Cx-c}
\\
&\sup_{t\geq 0}\mathbb{E}\Vert V(t)\Vert^{2}\leq C_{v}^{c}
:=\frac{\left(\frac{\beta M}{2}+\frac{\beta\gamma^{2}(2-\lambda)}{4}\right)R^{2}+\beta BR+\beta A
+\frac{3}{4}\beta\Vert V(0)\Vert^{2}+\frac{d+\overline{A}}{\lambda}}{\frac{\beta}{4}(1-2\lambda)},
\label{def-Cxv-c}
\end{align}
\item [$(ii)$] For any stepsize $\eta$ satisfying:
\begin{equation}\label{eqn:eta4}
0<\eta\leq
\overline{\eta}_{4}^{U}
:=\min\left\{1,\frac{\gamma}{\hat{K}_{2}}(d/\beta+\overline{A}/\beta),\frac{\gamma\lambda}{2\hat{K}_{1}},\frac{2}{\gamma\lambda}\right\},
\end{equation}
where
\begin{align}
&\hat{K}_{1}:=K_{1}+Q_{1}\frac{4}{1-2\lambda}+Q_{2}\frac{8}{(1-2\lambda)\gamma^{2}},
\label{def-hat-K1}
\\
&\hat{K}_{2}:=K_{2}+Q_{3},
\label{def-hat-K2}
\end{align}
where
\begin{align}
&K_{1}
:=\max\left\{\frac{32M^{2}\left(\frac{1}{2}+\gamma\right)}{ (1-2 \lambda) \beta \gamma^2 },
\frac{8\left(\frac{1}{2}M+\frac{1}{4}\gamma^{2}-\frac{1}{4}\gamma^{2}\lambda+\gamma\right)}{\beta(1-2\lambda)}\right\},
\label{def-K1}
\\
&K_{2}:=2B^{2}\left(\frac{1}{2}+\gamma\right),\label{def-K2}
\end{align}
and
\begin{align}
&Q_{1}:=\frac{1}{2}c_{0}
\Bigg((5M+4-2\gamma+(c_{0}+\gamma^{2}))
+(1+\gamma)\left(\frac{5}{2}+c_{0}(1+\gamma)\right)
+2\gamma^{2}\lambda\Bigg),\label{def-Q1}
\\
&Q_{2}:=\frac{1}{2}c_{0}
\Bigg[\Bigg((1+\gamma)\left(c_{0}(1+\gamma)+\frac{5}{2}\right)
+c_{0}+2
+\lambda\gamma^{2}
+2(Mc_{0}+M+1)\Bigg)\cdot 2M^{2}\nonumber
\\
&\qquad\qquad\qquad
+\left(2M^{2}+\gamma^{2}\lambda+\frac{3}{2}\gamma^{2}(1+\gamma)\right)\Bigg],\label{def-Q2}
\\
&Q_{3}:=c_{0}
\Bigg((1+\gamma)\left(c_{0}(1+\gamma)+\frac{5}{2}\right)
+c_{0}+2
+\lambda\gamma^{2}
+2(Mc_{0}+M+1)\Bigg)B^{2}+c_{0}B^{2}\nonumber
\\
&\qquad\qquad\qquad
+\frac{1}{2}\gamma^{3}\beta^{-1}c_{22}
+\gamma^{2}\beta^{-1}c_{12}+M\gamma\beta^{-1}c_{22},\label{def-Q3}
\end{align}
where 
\begin{equation}\label{def-c0-c12-c22}
c_{0}:=1+\gamma^{2},\qquad c_{12}:=\frac{d}{2},\qquad c_{22}:=\frac{d}{3},
\end{equation}
we have
\begin{align}
&\sup_{j\geq 0}\mathbb{E}\Vert X_{j}\Vert^{2}
\leq
C_{x}^{d}
:=\frac{\left(\frac{\beta M}{2}+\frac{\beta\gamma^{2}(2-\lambda)}{4}\right)R^{2}+\beta BR+\beta A
+\frac{3}{4}\beta\Vert V(0)\Vert^{2}
+\frac{4(d+\overline{A})}{\lambda}}{\frac{1}{8} (1-2 \lambda) \beta \gamma^2 },\label{def-Cx-d}
\\
&\sup_{j\geq 0}\mathbb{E}\Vert V_{j}\Vert^{2}
\leq
C_{v}^{d}
:=\frac{\left(\frac{\beta M}{2}+\frac{\beta\gamma^{2}(2-\lambda)}{4}\right)R^{2}+\beta BR+\beta A
+\frac{3}{4}\beta\Vert V(0)\Vert^{2}
+\frac{4(d+\overline{A})}{\lambda}}{\frac{\beta}{4}(1-2\lambda)}.\label{def-Cxv-d}
\end{align}
\end{itemize}
\end{lemma}

%%%%%%%%%%%%%%%%%%%%%%%%%%%%%%%%%%%
\subsubsection{Proofs of auxiliary results}

\begin{proof}[Proof of Lemma~\ref{lem:ineq}]
Note that $Q_{t_{0}}(t_{1})Z_{t}^{0}$
is a $2d$-dimensional martingale and by Doob's martingale inequality,
for any $h>0$,
\begin{align}
\mathbb{P}\left(\sup_{t_{0}\leq t\leq t_{1}}\Vert Q_{t_{0}}(t_{1})Z_{t}^{0}\Vert\geq h\right)
&\leq e^{-\beta\theta h^{2}/2}
\mathbb{E}\left[e^{(\beta\theta/2)\Vert Q_{t_{0}}(t_{1})Z_{t_{1}}^{0}\Vert^{2}}\right]
\nonumber \\
&=e^{-\beta\theta h^{2}/2}
\frac{1}{\sqrt{\det(I-\beta\theta\Sigma_{t_{1}})}}e^{\frac{\beta\theta}{2}\langle\mu_{t_{1}},(I-\beta\theta\Sigma_{t_{1}})^{-1}\mu_{t_{1}}\rangle}, \label{eq:est1}
\end{align}
where the last line above uses the fact that $Q_{t_{0}}(t_{1})Z_{t_{1}}$
is a Gaussian random vector with mean
\begin{equation*}
\mu_{t_{1}}=e^{-t_{1}H_{\gamma}}(V(0),Y(0))^{T},
\end{equation*}
and covariance matrix 
\begin{align*}
\Sigma_{t_{1}}
&=2\gamma\beta^{-1}\int_{0}^{t_{1}}\left(e^{(s-t_{1})H_{\gamma}}I^{(2)}\right)
\left(e^{(s-t_{1})H_{\gamma}}I^{(2)}\right)^{T}ds
\\
&=2\gamma\beta^{-1}\int_{0}^{t_{1}}e^{-sH_{\gamma}}I^{(2)}e^{-sH_{\gamma}^{T}}ds.
\end{align*}
We next estimate $\det(I-\beta\theta\Sigma_{t_{1}})$ fron \eqref{eq:est1}. 
Let us recall from Lemma \ref{lem:keylemma}
that if $\gamma=2\sqrt{m}$, then we recall from Lemma \ref{lem:keylemma} that,
\begin{equation*}
\left\| e^{-t H_\gamma}\right\| \leq \sqrt{C_{H}+2 + (m+1)^2 t^2 }\cdot e^{-\sqrt{m}t},
\end{equation*} 
and thus, we have
\begin{equation*}
\Vert\Sigma_{t_{1}}\Vert
\leq 2\gamma\beta^{-1}\int_{0}^{t_{1}}\left(C_{H}+2+(m+1)^{2}t^{2}\right)e^{-2\sqrt{m}t}dt
\leq
\gamma\beta^{-1}\frac{2C_{H}m+4m+(m+1)^{2}}{2m\sqrt{m}}.
\end{equation*}
Therefore we infer that 
the eigenvalues of $I-\beta\theta\Sigma$ are bounded below by 
$1-\theta\frac{\gamma(2C_{H}m+4m+(m+1)^{2})}{2m\sqrt{m}}$. 
The conclusion then follows from \eqref{eq:est1}.
\end{proof}

%%%%%%%%%%%%%%%%%%
\begin{proof}[Proof of Lemma \ref{lem:lambda:A}]
By Assumption \ref{assumptions} (iii), 
$x\cdot\nabla F(x)\geq m\Vert x\Vert^{2}-b$.
Thus in order to show 
the drift condition \eqref{eq:drift},
it suffices to show that
\begin{equation}\label{eq:drift:alter}
m\Vert x\Vert^{2}-b
-2\lambda( F(x)+ \gamma^{2}\Vert x\Vert^{2}/4)
\geq -2\overline{A}/\beta.
\end{equation}
Given the definition of $\lambda$
in \eqref{eqn:lambda}, by 
Lemma \ref{lem:gradient-bound}, we get
\begin{align*}
&m\Vert x\Vert^{2}-b
-2\lambda( F(x)+ \gamma^{2}\Vert x\Vert^{2}/4)
\\
&\geq
m\Vert x\Vert^{2}-b
-\frac{m}{M+\frac{1}{2}\gamma^{2}}( F(x)+ \gamma^{2}\Vert x\Vert^{2}/4)
\\
&\geq
\frac{mM+\frac{1}{4}m\gamma^{2}}{M+\frac{1}{2}\gamma^{2}}\Vert x\Vert^{2}-b
-\frac{m}{M+\frac{1}{2}\gamma^{2}}\left(\frac{M}{2}\Vert x\Vert^{2}+B\Vert x\Vert+A\right)
\\
&=\frac{m}{M+\frac{1}{2}\gamma^{2}}
\left(\frac{1}{2}M\Vert x\Vert^{2}+\frac{1}{4}\gamma^{2}\Vert x\Vert^{2}-B\Vert x\Vert
-\frac{b}{m}\left(M+\frac{1}{2}\gamma^{2}\right)-A\right)
\\
&\geq\frac{m}{M+\frac{1}{2}\gamma^{2}}
\left(-\frac{B^{2}}{2M+\gamma^{2}}
-\frac{b}{m}\left(M+\frac{1}{2}\gamma^{2}\right)-A\right)
=-2\overline{A}/\beta,
\end{align*}
by the definition of $\overline{A}$
in \eqref{eqn:overline:A}. 
Hence, \eqref{eq:drift:alter} holds
and the proof is complete.
\end{proof}

%%%%%%%%%%%%%%%%%%%%%%%%
\begin{proof}[Proof of Lemma \ref{lem:L2bound}]
According to {Lemma 12 (i)} in \cite{GGZ}, 
\begin{align*}
&\sup_{t\geq 0}\mathbb{E}\Vert X(t)\Vert^{2}\leq 
\frac{\int_{\mathbb{R}^{2d}}\mathcal{V}(x,v)d\mu_{0}(x,v)+\frac{d+\overline{A}}{\lambda}}{\frac{1}{8} (1-2 \lambda) \beta \gamma^2}, \\
&\sup_{t\geq 0}\mathbb{E}\Vert V(t)\Vert^{2}
\leq\frac{\int_{\mathbb{R}^{2d}}\mathcal{V}(x,v)d\mu_{0}(x,v)
+\frac{d+\overline{A}}{\lambda}}{\frac{\beta}{4}(1-2\lambda)},
\end{align*}
where $\mathcal{V}$ is the Lyapunov function defined in \eqref{eq:lyapunov}
and $\mu_{0}$ is the initial distribution of $(X(0),V(0))$
and in our case, $\mu_{0}=\delta_{(X(0),V(0))}$ and $\Vert X(0)\Vert\leq R$ 
and $V(0)\in\mathbb{R}^{d}$,
and for any 
$0<\eta\leq\min\left\{1,\frac{\gamma}{\hat{K}_{2}}(d/\beta+\overline{A}/\beta),\frac{\gamma\lambda}{2\hat{K}_{1}},\frac{2}{\gamma\lambda}\right\}$ with $\hat{K}_{1}$ and $\hat{K}_{2}$ given in \eqref{def-hat-K1} and \eqref{def-hat-K2},
\footnote{Note that in the definition of $\hat{K}_{1},\hat{K}_{2}$ in \cite{GGZ}, 
there is a constant $\delta$, which is simply zero, in the context of the current paper.}
and according to {Lemma 17} in \cite{GGZ}, we also have
\begin{align*}
&\sup_{j\geq 0}\mathbb{E}\Vert X_{j}\Vert^{2}
\leq
\frac{\int_{\mathbb{R}^{2d}}\mathcal{V}(x,v)\mu_{0}(dx,dv)
+\frac{4(d+\overline{A})}{\lambda}}{\frac{1}{8} (1-2 \lambda) \beta \gamma^2 },
\\
&\sup_{j\geq 0}\mathbb{E}\Vert V_{j}\Vert^{2}
\leq
\frac{\int_{\mathbb{R}^{2d}}\mathcal{V}(x,v)\mu_{0}(dx,dv)
+\frac{4(d+\overline{A})}{\lambda}}{\frac{\beta}{4}(1-2\lambda)}.
\end{align*}
We recall from \eqref{eq:lyapunov} that
$\mathcal{V}(x,v)=\beta F(x)
+\frac{\beta}{4}\gamma^{2}(\Vert x+\gamma^{-1}v\Vert^{2}+\Vert\gamma^{-1}v\Vert^{2}-\lambda\Vert x\Vert^{2})$,
and $\Vert X(0)\Vert\leq R$ and $V(0)\in\mathbb{R}^{d}$. By Lemma \ref{lem:gradient-bound}, 
we get
\begin{equation*} 
\mathcal{V}(x,v)\leq\frac{\beta M}{2}\Vert x\Vert^{2}+\beta B\Vert x\Vert+\beta A
+\frac{\beta}{4}\gamma^{2}(\Vert x+\gamma^{-1}v\Vert^{2}+\Vert\gamma^{-1}v\Vert^{2}-\lambda\Vert x\Vert^{2})\,,
\end{equation*}
so that
\begin{align*}
&\mathcal{V}(X(0),V(0))
\\
&=\frac{\beta M}{2}\Vert X(0)\Vert^{2}+\beta B\Vert X(0)\Vert+\beta A
+\frac{\beta}{4}\gamma^{2}(2\Vert X(0)\Vert^{2}+3\gamma^{-2}\Vert V(0)\Vert^{2}-\lambda\Vert X(0)\Vert^{2})
\\
&\leq\left(\frac{\beta M}{2}+\frac{\beta\gamma^{2}(2-\lambda)}{4}\right)R^{2}+\beta BR+\beta A
+\frac{3}{4}\beta\Vert V(0)\Vert^{2}.
\end{align*}
Hence, the conclusion follows.
\end{proof}

%%%%%%%%%%%%%%%%%%%%%%%%%%%%%%%%%%%%%%%%%%%%%%%%%%%

%%%%%%%%%%%%%%%%%%%%%%%%%%%%%%%%%%%%%%%%%%%%
%\section{Proofs of results in Section \ref{sec:anti}}\label{sec:proof:anti}

%%%%%%%%%%%%%%%%%%%%%%%%%%%%%%%%%%%%%%%%%%%
\subsection{Proof of Theorem \ref{MainThm1:Q}}

The proof of Theorem \ref{MainThm1:Q} is similar to the proof of Theorem \ref{MainThm1}. 
For brevity, we omit some of the details, and only outline the key steps and the propositions and lemmas
used for the proof of Theorem \ref{MainThm1:Q}.

\begin{proposition}\label{prop:key:Q}
Fix any $r>0$ and $0<\varepsilon<\min\{\overline{\varepsilon}_{1}^{J},\overline{\varepsilon}_{2}^{J}\}$, where
\begin{equation}\label{eqn:e:J}
\overline{\varepsilon}_{1}^{J}:=\frac{m_{J}(\tilde{\varepsilon})}{4C_{J}(\tilde{\varepsilon})(1+\Vert J\Vert) L(1+\frac{1}{64C_{J}(\tilde{\varepsilon})^{2}})},
\qquad
\overline{\varepsilon}_{2}^{J}:=8rC_{J}(\tilde{\varepsilon}).
\end{equation}
Consider the stopping time:
\begin{equation*}
\tau:=\inf\left\{t\geq 0:\Vert X(t)-x_{\ast}\Vert\geq\varepsilon+re^{-m_{J}(\tilde{\varepsilon}) t}\right\}.
\end{equation*}
For any initial point $X(0)=x$ with $\Vert x-x_{\ast}\Vert \le r$, and 
\begin{equation*}
\beta\geq\frac{128C_{J}(\tilde{\varepsilon})^{2}}{m_{J}(\tilde{\varepsilon})\varepsilon^{2}}\left(\frac{d}{2}\log(2)
+\log\left(\frac{2(1+\Vert J\Vert)M\mathcal{T}+1}{\delta}\right)\right),
\end{equation*}
we have
\begin{equation*}
\mathbb{P}_{x}\left(\tau\in[\mathcal{T}_{\text{rec}}^{J},\mathcal{T}_{\text{esc}}^{J}]\right)\leq\delta.
\end{equation*}
\end{proposition}

%%%%%%%%%%%%%%%%%%%%%%%%%%%%%%%%%%%

\subsubsection{Completing the proof of Theorem \ref{MainThm1:Q}}\label{sec:complete:proof:Q}

We first compare the discrete dynamics \eqref{eqn:nonreversible:discrete} 
and the continuous dynamics \eqref{eqn:nonreversible1}.
Define: 
\begin{equation}\label{eqn:tilde:X}
\tilde{X}(t)=X_0
-\int_{0}^{t}A_{J}\left(\nabla F(\tilde{X}(\lfloor s/\eta\rfloor\eta))\right)ds+\sqrt{2\gamma \beta^{-1}}\int_{0}^{t}dB_{s}.
\end{equation}
The process $\tilde{X}$ defined in \eqref{eqn:tilde:X} is the continuous-time interpolation of the iterates $\{X_{k}\}$. 
In particular, the joint distribution of $\{X_{k}: k = 1, 2, \ldots, K\}$ 
is the same as $\{\tilde{X}(t): t = \eta, 2 \eta, \ldots, K \eta\}$ for any positive integer $K$. 

By following Lemma 7 in \cite{Raginsky} and apply the uniform $L^{2}$
bounds for $X_{k}$ in Corollary \ref{cor:L2:discrete} provided
that the stepsize $\eta$ is sufficiently small (we apply
the bound $\Vert A_{J}\Vert\leq 1+\Vert J\Vert$ to Corollary \ref{cor:L2:discrete})
\begin{equation}\label{eqn:eta4:J}
\eta\leq\overline{\eta}_{4}^{J}:=\frac{1}{M(1+\Vert J\Vert)^{2}},
\end{equation} 
we will obtain an upper bound on the
relative entropy $D(\cdot\Vert\cdot)$ between 
the law $\tilde{\mathbb{P}}^{K\eta}$ of $(\tilde{X}(t):t\leq K\eta)$
and the law $\mathbb{P}^{K\eta}$ of $(X(t):t\leq K\eta)$, 
and by Pinsker's inequality an upper bound on 
the total variation $\Vert\cdot\Vert_{TV}$ as well. More precisely, we have
\begin{equation}\label{eqn:entropy}
\left\Vert\tilde{\mathbb{P}}^{K\eta}-\mathbb{P}^{K\eta}\right\Vert_{TV}^{2}
\leq\frac{1}{2}D\left(\tilde{\mathbb{P}}^{K\eta}\Big\Vert\mathbb{P}^{K\eta}\right)
\leq\frac{1}{2}C_{1}K\eta^{2},
\end{equation}
where (we use the bound $\Vert A_{J}\Vert\leq 1+\Vert J\Vert$)
\begin{equation}\label{eqn:C1}
C_{1}:=6(\beta((1+\Vert J\Vert)^{2}M^{2}C_{d}+B^{2})+d)(1+\Vert J\Vert)^{2}M^{2},
\end{equation}
where $C_{d}$ is defined in \eqref{eqn:Cd}.

Let us now complete the proof of Theorem \ref{MainThm1:Q}.
We need to show that
\begin{equation*}
\mathbb{P}\left((X_{1},\ldots,X_{K})\in\mathcal{A}\right)\leq\delta,
\end{equation*}
where 
$K=\lfloor\eta^{-1}\mathcal{T}_{\text{esc}}^{J}\rfloor$ and $\mathcal{A}:=\mathcal{A}_{1}\cap\mathcal{A}_{2}$:
\begin{align*}
&\mathcal{A}_{1}:=\left\{(x_{1},\ldots,x_{K})\in(\mathbb{R}^{d})^{K}:\max_{k\leq\eta^{-1}\mathcal{T}_{\text{rec}}^{J}}
\frac{\Vert x_{k}-x_{\ast}\Vert}{\varepsilon+re^{-m_{J}(\tilde{\varepsilon}) k\eta}}\leq\frac{1}{2}\right\},
\\
&\mathcal{A}_{2}:=\left\{(x_{1},\ldots,x_{K})\in(\mathbb{R}^{d})^{K}:\max_{\eta^{-1}\mathcal{T}_{\text{rec}}^{J}\leq k\leq K}
\frac{\Vert x_{k}-x_{\ast}\Vert}{\varepsilon+re^{-m_{J}(\tilde{\varepsilon})k\eta}}\geq 1\right\}.
\end{align*}
Similar to the proof in Section \ref{sec:complete:proof} and by \eqref{eqn:entropy}, we get
\begin{equation}\label{in:view:1:J}
\mathbb{P}((X_{1},\ldots,X_{K})\in\mathcal{A})
\leq\mathbb{P}((X(\eta),\ldots,X(K\eta))\in\mathcal{A})+\frac{\delta}{3},
\end{equation}
provided that
\begin{equation}\label{eqn:eta3:J}
\eta\leq
\overline{\eta}_{3}^{J}:=
\frac{2\delta^{2}}{9C_{1}\mathcal{T}_{\text{esc}}^{J}}.
\end{equation} 
It remains to estimate the probability of $\mathbb{P}((X(\eta),\ldots,X(K\eta))\in\mathcal{A}_{1}\cap\mathcal{A}_{2})$ 
for the non-reversible Langevin diffusion.
Partition the interval $[0,\mathcal{T}_{\text{rec}}^{J}]$
using the points $0=t_{1}<t_{1}<\cdots<t_{\lceil\eta^{-1}\mathcal{T}_{\text{rec}}^{J}\rceil}
=\mathcal{T}_{\text{rec}}^{J}$
with $t_{k}=k\eta$ for $k=0,1,\ldots,\lceil\eta^{-1}\mathcal{T}_{\text{rec}}^{J}\rceil-1$,
and consider the event: 
\begin{equation*}
\mathcal{B}:=\left\{\max_{0\leq k\leq
\lceil\eta^{-1}\mathcal{T}_{\text{rec}}^{J}\rceil-1}\max_{t\in[t_{k},t_{k+1}]}\Vert X(t)-X(t_{k+1})\Vert
\leq\frac{\varepsilon}{2}\right\}.
\end{equation*}
Similar to the proof in Section \ref{sec:complete:proof}, we get
\begin{equation}\label{in:view:2:J}
\mathbb{P}((X(\eta),\cdots,X(K\eta))\in\mathcal{A})
\leq
\frac{\delta}{3}+\mathbb{P}(\mathcal{B}^{c})\,,
\end{equation}
provided that (by applying Proposition \ref{prop:key:Q}): 
\begin{equation}\label{eqn:beta1:J}
\beta\geq
\underline{\beta}_{1}^{J}:=\frac{128C_{J}(\tilde{\varepsilon})^{2}}{m_{J}(\tilde{\varepsilon})\varepsilon^{2}}\left(\frac{d}{2}\log(2)+\log\left(\frac{6(1+\Vert J\Vert)M\mathcal{T}+3}{\delta}\right)\right).
\end{equation}

To complete the proof, we need to show that
$\mathbb{P}(\mathcal{B}^{c})\leq\frac{\delta}{3}$ in view of \eqref{in:view:1:J} and \eqref{in:view:2:J}.
For any $t\in[t_{k},t_{k+1}]$, where $t_{k+1}-t_{k}=\eta$, we have
\begin{align*}
&\Vert X(t)-X(t_{k+1})\Vert
\nonumber
\\
&\leq
\int_{t}^{t_{k+1}}\Vert A_{J}\nabla F(X(s))\Vert ds
+\sqrt{2\beta^{-1}}\Vert B_{t}-B_{t_{k+1}}\Vert
\nonumber
\\
&\leq 
\Vert A_{J}\Vert M\int_{t}^{t_{k+1}}\Vert X(s)-X(t_{k+1})\Vert ds
+\eta\Vert A_{J}\nabla F(X(t_{k+1}))\Vert
+\sqrt{2\beta^{-1}}\Vert B_{t}-B_{t_{k+1}}\Vert
\nonumber
\\
&\leq
\Vert A_{J}\Vert M\int_{t}^{t_{k+1}}\Vert X(s)-X(t_{k+1})\Vert ds
\nonumber
\\
&\qquad\qquad\qquad
+\eta\Vert A_{J}\Vert\cdot(M\Vert X(t_{k+1})\Vert+B)
+\sqrt{2\beta^{-1}}\Vert B_{t}-B_{t_{k+1}}\Vert\,.
\end{align*}
By Gronwall's inequality, we get the key estimate:
\begin{align*}
&\sup_{t\in[t_{k},t_{k+1}]}
\Vert X(t)-X(t_{k+1})\Vert
\\
&\leq
e^{\eta\Vert A_{J}\Vert M}\left[\eta\Vert A_{J}\Vert\cdot(M\Vert X(t_{k+1})\Vert+B)
+\sqrt{2\beta^{-1}}\sup_{t\in[t_{k},t_{k+1}]}\Vert B_{t}-B_{t_{k+1}}\Vert\right].
\end{align*}
Then, by following the same argument as in Section \ref{sec:complete:proof}
and also apply $\Vert A_{J}\Vert\leq 1+\Vert J\Vert$, 
we can show that $\mathbb{P}(\mathcal{B}^{c})\leq\frac{\delta}{3}$
provided that 
$\eta\leq 1$ and
\begin{equation}\label{eqn:eta1:J}
\eta\leq
\overline{\eta}_{1}^{J}:=\frac{\varepsilon e^{-(1+\Vert J\Vert)M}}{8(1+\Vert J\Vert)B},
\end{equation}
and
\begin{equation}\label{eqn:eta2:J}
\eta\leq
\overline{\eta}_{2}^{J}:=
\frac{\delta\varepsilon^{2}e^{-2(1+\Vert J\Vert)M}}{384(1+\Vert J\Vert)^{2}M^{2}C_{c}\mathcal{T}_{\text{rec}}^{J}}\,,
\end{equation}
where $C_{c}$ is defined in \eqref{eqn:Cc} and
\begin{equation}\label{eqn:beta2:J}
\beta\geq
\underline{\beta}_{2}^{J}:=\frac{512d\eta\log(2^{1/4}e^{1/4}6\delta^{-1}\mathcal{T}_{\text{rec}}^{J}/\eta)}
{\varepsilon^{2}e^{-2(1+\Vert J\Vert)M\eta}}\,.
\end{equation} 

To complete the proof, we need work on the leading orders of the constants.
We treat $\Vert J\Vert$, $M$, $L$ as constant.
The argument is similar to the argument in the proof of Theorem \ref{MainThm1}
and is thus omitted here. The proof is now complete.
%%%%%%%%%%%%%%%%%%%%%%%%%%%

\subsubsection{Proof of Proposition \ref{prop:key:Q}}

Before we proceed to the proof of Proposition \ref{prop:key:Q}, let 
us first state the following two lemmas that will be used in the proof of Proposition \ref{prop:key:Q}.

\begin{lemma}\label{lem:ineq:Q}
For any $\theta\in (0,\frac{\lambda_{1}^{J}-\tilde{\varepsilon}}{(C_{J}(\tilde{\varepsilon}))^{2}})$, $h>0$ and $y_{0}\in\mathbb{R}^{d}$,
\begin{equation*}
\mathbb{P}\left(\sup_{t_{0}\leq t\leq t_{1}}\left\Vert Q_{t_{0}}(t_{1})Z_{t}^{0}\right\Vert\geq h\right)
\leq\left(1-\theta\frac{(C_{J}(\tilde{\varepsilon}))^{2}}{\lambda_{1}^{J}-\tilde{\varepsilon}}\right)^{-d/2}
e^{-\frac{\beta\theta}{2}[h^{2}-\langle\mu_{t_{1}},(I-\beta\theta\Sigma_{t_{1}})^{-1}\mu_{t_{1}}\rangle]},
\end{equation*}
where $Q_{t_{0}}(t_{1})$ is defined in \eqref{eqn:matrix:flow:Q},
$Z_{t}^{0}$ is defined in \eqref{eqn:Z0:Q}, and
\begin{equation}\label{defn:mu:Sigma:J}
\mu_{t}:=e^{-tA_{J}H}y_{0},
\qquad
\Sigma_{t}:=2\beta^{-1}\int_{0}^{t}e^{-s(A_{J}H)}e^{-s(A_{J}H)^{T}}ds.
\end{equation}
\end{lemma}

\begin{lemma}\label{lem:ineq:2:Q}
Given $t_{0}\leq t\leq (t_{1}\wedge\tau)$, 
where $\tau$ is the stopping time
defined in Proposition \ref{prop:key:Q},
we have
\begin{equation*}
\left\Vert Q_{t_{0}}(t_{1})Z_{t}^{1}\right\Vert
\leq\frac{C_{J}(\tilde{\varepsilon})\Vert A_{J}\Vert L}{2}\int_{0}^{t}e^{ (s-t_{1})m_{J}(\tilde{\varepsilon})}\left(\varepsilon+re^{-m_{J}(\tilde{\varepsilon})s}\right)^{2}ds,
\end{equation*}
where $Q_{t_{0}}(t_{1})$ is defined in \eqref{eqn:matrix:flow:Q},
and $Z_{t}^{1}$ is defined in \eqref{eqn:Z1:Q}.
\end{lemma}

\begin{proof}[Proof of Proposition \ref{prop:key:Q}]
We recall $x_{\ast}$ is a local minimum of $F$ and $H$ is the Hessian matrix:
$H=\nabla^{2}F(x_{\ast})$, and we write 
\begin{equation*}
X(t)=Y(t)+x_{\ast}.
\end{equation*}
Thus, we have the decomposition
\begin{equation*}
\nabla F(X(t))=HY(t)-\rho(Y(t)),
\end{equation*}
where $\Vert\rho(Y(t))\Vert\leq\frac{1}{2}L\Vert Y(t)\Vert^{2}$
since the Hessian of $F$ is $L$-Lipschitz (Lemma 1.2.4. \cite{Nesterov}).
This implies that
\begin{equation*}
dY(t)=-A_{J}HY(t)dt
+A_{J}\rho(Y(t))dt+\sqrt{2\beta^{-1}}dB_{t}.
\end{equation*}
Thus, we get
\begin{equation*}
Y(t)=e^{-tA_{J}H}Y(0)
+\sqrt{2\beta^{-1}}\int_{0}^{t}e^{(s-t)A_{J}H}dB_{s}
+\int_{0}^{t}e^{(s-t)A_{J}H}A_{J}\rho (Y(s))ds.
\end{equation*}
Given $0\leq t_{0}\leq t_{1}$, we define the matrix flow 
\begin{equation}\label{eqn:matrix:flow:Q}
Q_{t_{0}}(t):=e^{(t_{0}-t)A_{J}H},
\end{equation}
and $Z_{t}:=e^{(t-t_{0})A_{J}H}Y_{t}$ so that
\begin{equation*}
Z_{t}=e^{-t_{0}A_{J}H}Y(0)
+\sqrt{2\beta^{-1}}\int_{0}^{t}e^{(s-t_{0})A_{J}H}dB_{s}
+\int_{0}^{t}e^{(s-t_{0})A_{J}H}A_{J}\rho (Y(s))ds.
\end{equation*}
We define the decomposition $Z_{t}=Z_{t}^{0}+Z_{t}^{1}$, where
\begin{align}
&Z_{t}^{0}=e^{-t_{0}A_{J}H}Y(0)
+\sqrt{2\beta^{-1}}\int_{0}^{t}e^{(s-t_{0})A_{J}H}dB_{s},\label{eqn:Z0:Q}
\\
&Z_{t}^{1}=\int_{0}^{t}e^{(s-t_{0})A_{J}H}A_{J}\rho (Y(s))ds.\label{eqn:Z1:Q}
\end{align}
It follows that for any $t_{0}\leq t\leq t_{1}$,
\begin{align*}
&Q_{t_{0}}(t_{1})Z_{t}^{1}=\int_{0}^{t}e^{(s-t_{1})A_{J}H}A_{J}\rho (Y(s))ds,
\\
&Q_{t_{0}}(t_{1})Z_{t}^{0}=e^{-t_{1}A_{J}H}Y(0)
+\sqrt{2\beta^{-1}}\int_{0}^{t}e^{(s-t_{1})A_{J}H}dB_{s}.
\end{align*}
The rest of the proof is similar to the proof of Proposition \ref{prop:key}.
We apply Lemma \ref{lem:ineq:2:Q} to bound the term $Q_{t_{0}}(t_{1})Z_{t}^{1}$
and apply Lemma \ref{lem:ineq:Q} to bound the term $Q_{t_{0}}(t_{1})Z_{t}^{0}$.
By letting $\gamma=1$ in Proposition \ref{prop:key}
and replacing $d$ by $d/2$ due to Lemma \ref{lem:ineq:Q}, 
and $\Vert H_{\gamma}\Vert$ by $\Vert A_{J}H\Vert$ and using
the bounds $\Vert A_{J}\Vert\leq(1+\Vert J\Vert)$ and $\Vert A_{J}H\Vert\leq(1+\Vert J\Vert)M$,
we obtain the desired result in Proposition \ref{prop:key:Q}.
\end{proof}

%%%%%%%%%%%%%%%%%%%%%%%%%%%%%%%%%%%%%%%%%%
\subsubsection{Uniform $L^2$ bounds for NLD}
In this section we establish uniform $L^2$ bounds for both the continuous time dynamics \eqref{eqn:nonreversible1}
and discrete time dynamics \eqref{eqn:nonreversible:discrete}. 
The main idea of the proof is to use Lyapunov functions. 
Our local analysis result relies on the approximation of the continuous time dynamics \eqref{eqn:nonreversible1}
by the discrete time dynamics \eqref{eqn:nonreversible:discrete}. 
The uniform $L^{2}$ bound for the discrete dynamics \eqref{eqn:nonreversible:discrete}
is used to derive the relative entropy
to compare the laws of the continuous time dynamics and the discrete time dynamics,
and the uniform $L^{2}$ bound for the continuous dynamics \eqref{eqn:nonreversible1}
is used to control
the tail of the continuous dynamics in Section \ref{sec:complete:proof:Q}.
We first recall the continuous-time dynamics from \eqref{eqn:nonreversible1}:
\begin{equation*}
dX(t)=-A_{J}(\nabla F(X(t)))dt+\sqrt{2\beta^{-1}}dB_{t},\qquad A_{J}=I+J,
\end{equation*}
where $J$ is a $d\times d$ anti-symmetric matrix, 
i.e. $J^{T}=-J$. 
The generator of this continuous time process is given by
\begin{equation}\label{eqn:generator}
\mathcal{L} = -A_J \nabla F  \cdot \nabla + \beta^{-1} \Delta
\end{equation}

\begin{lemma} \label{lem:cont-anti-L2}
Given $X(0)=x\in\mathbb{R}^{d}$,
\begin{equation*}
\mathbb{E} [F(X(t))]
\leq
F(x) 
+\frac{B}{2}+A+\frac{b(M+B)}{m}+\frac{2M\beta^{-1}d(M+B)}{m^{2}}.
\end{equation*}
\end{lemma}

Since $F$ has at most the quadratic growth (due to Lemma \ref{lem:gradient-bound}), 
we immediately have the following corollary.

\begin{corollary}\label{cor:L2:continuous}
Given $\Vert X(0)\Vert\leq R=\sqrt{b/m}$,
\begin{equation}\label{eqn:Cc}
\mathbb{E} [\Vert X(t)\Vert^{2}]
\leq
C_{c}:=\frac{MR^{2}+2BR+B+4A}{m}
+\frac{2b(M+B)}{m^{2}}+\frac{4M\beta^{-1}d(M+B)}{m^{3}}+\frac{b}{m}\log 3.
\end{equation}
\end{corollary}

%%%%%%%%%%%%%%%%%%%%%%%%%%%%%%%%%%%%%%
We next show uniform $L^2$ bounds for the discrete iterates $X_k$, where
we recall from \eqref{eqn:nonreversible:discrete} that the non-reversible Langevin dynamics is given by:
\begin{equation*}
X_{k+1}=X_{k}-\eta A_{J}(\nabla F(X_{k}))+\sqrt{2\eta\beta^{-1}}\xi_{k}.
\end{equation*}

\begin{lemma}\label{lem:disc-anti-L2}
Given that $\eta\leq\min\left\{\frac{1}{M\Vert A_{J}\Vert^{2}},\frac{4(M+B)}{m^{2}}\right\}$, we have
\begin{equation*}
\mathbb{E}_x[F(X_k)] \le F(x) 
+\frac{B}{2}+A
+\frac{4(M+B)M\beta^{-1}d}{m^{2}}
+\frac{(M+B)b}{m}.
\end{equation*}
\end{lemma}

Since $F$ has at most the quadratic growth (due to Lemma \ref{lem:gradient-bound}), 
we immediately have the following corollary.

\begin{corollary}\label{cor:L2:discrete}
Given that $\eta\leq\min\left\{\frac{1}{M\Vert A_{J}\Vert^{2}},\frac{4(M+B)}{m^{2}}\right\}$ and $\Vert X(0)\Vert\leq R=\sqrt{b/m}$, we have
\begin{equation}\label{eqn:Cd}
\mathbb{E}[\Vert X_{k}\Vert^{2}] 
\leq
C_{d}:=\frac{MR^{2}+2BR+B+4A}{m}
+\frac{8(M+B)M\beta^{-1}d}{m^{3}}
+\frac{2(M+B)b}{m^{2}}
+\frac{b}{m}\log 3.
\end{equation}
\end{corollary}

%%%%%%%%%%%%%%%%%%%%%%%%%%%%%%%%%%%%%%%%%%%%%

\subsubsection{Proofs of auxiliary results}

%%%%%%%%%%%%%%%%%%%%%%%%%%%%%%%%%%%%%%%%%%%%%%
\begin{proof}[Proof of Lemma \ref{lem:ineq:Q}]
By following the proof of Lemma \ref{lem:ineq}. We get
\begin{equation*}
\mathbb{P}\left(\sup_{t_{0}\leq t\leq t_{1}}\left\Vert Q_{t_{0}}(t_{1})Z_{t}^{0}\right\Vert\geq h\right)
\leq\frac{1}{\sqrt{\det(I-\beta\theta\Sigma_{t_{1}})}}
e^{-\frac{\beta\theta}{2}[h^{2}-\langle\mu_{t_{1}},(I-\beta\theta\Sigma_{t_{1}})^{-1}\mu_{t_{1}}\rangle]},
\end{equation*}
Recall from \eqref{eq:norm_J} that for any $\tilde{\varepsilon}>0$, 
there exists some $C_{J}(\tilde{\varepsilon})$ such that for every $t\geq 0$,
\begin{equation*}
\left\| e^{-tA_{J}H}\right\| \leq C_{J}(\tilde{\varepsilon}) e^{-(\lambda_{1}^{J}-\tilde{\varepsilon})t},
\end{equation*}
Hence, by the definition of $\Sigma_{t}$ from \eqref{defn:mu:Sigma:J}, we get
\begin{equation*}
\Vert\Sigma_{t}\Vert
\leq 2\beta^{-1}\int_{0}^{\infty}(C_{J}(\tilde{\varepsilon}))^{2} e^{-2(\lambda_{1}^{J}-\tilde{\varepsilon})t}dt
=\frac{\beta^{-1}(C_{J}(\tilde{\varepsilon}))^{2}}{\lambda_{1}^{J}-\tilde{\varepsilon}}.
\end{equation*}
The rest of the proof follows similarly as in the proof of Lemma \ref{lem:ineq}.
\end{proof}

%%%%%%%%%%%%%%%%%%%%%%%%%%%%%%%%%%%%%%%%%%%%%%%%
\begin{proof}[Proof of Lemma \ref{lem:ineq:2:Q}]
Note that
\begin{equation*}
\left\Vert Q_{t_{0}}(t_{1})Z_{t}^{1}\right\Vert
\leq\int_{0}^{t}\left\Vert e^{ (s-t_{1})A_{J}H}\right\Vert \left\Vert A_{J}\right\Vert 
\left\Vert\rho(Y(s))\right\Vert ds,
\end{equation*}
and by applying $\Vert\rho(Y(t))\Vert\leq\frac{1}{2}L\Vert Y(t)\Vert^{2}$
and \eqref{eq:norm_J}, 
and $t_{0}\leq t\leq (t_{1}\wedge\tau)$
and the definition of the stopping
time $\tau$ in Proposition \ref{prop:key:Q},
we get the desired result.
\end{proof}
%%%%%%%%%%%%%%%%%%%%%%%%%%%%%%%%%%%%%%%%%%%%%

\begin{proof}[Proof of Lemma~\ref{lem:cont-anti-L2}]
Note that if we can show that $F(x)$ is a Lyapunov function for $X(t)$:
\begin{equation}\label{eq:drift-b}
\mathcal{L} F(x) \le -\epsilon_{1} F(x) +b_{1},
\end{equation}
for some $\epsilon_{1}, b_{1} >0$, then 
\begin{equation*}
\mathbb{E} [F(X(t))]
\leq
F(x) + \frac{b_{1}}{\epsilon_{1}}.
\end{equation*}
Let us first prove this.
Applying Ito formula to $e^{\epsilon_{1}t} F(X(t))$, we obtain from Dynkin formula and the drift condition \eqref{eq:drift-b} that
for $t_K:=\min\{t, \tau_K\}$ with $\tau_K$ be the exit time of $X(t)$ from a ball 
centered at $0$ with radius $K$ with $X(0)=x$,
\begin{equation*}
\mathbb{E}[e^{\epsilon_{1} t_K } F(X(t_K))]
\leq F(x) + \mathbb{E}\left[\int_{0}^{t_K} b_{1} e^{\epsilon_{1} s} ds\right] 
\leq F(x) + \int_{0}^t b_{1} e^{\epsilon_{1} s} ds
\leq F(x) + \frac{b_{1}}{\epsilon_{1}} \cdot e^{\epsilon_{1} t}.
\end{equation*}
Let $K \rightarrow \infty$, then we can infer from Fatou's lemma that
for any $t$:
\begin{equation*}
\mathbb{E}\left[e^{\epsilon_{1} t}F(X(t))\right] \le F(x) + \frac{b_{1}}{\epsilon_{1}} \cdot e^{\epsilon_{1} t}.
\end{equation*}
Hence, we have
\begin{equation*}
\mathbb{E} [F(X(t))]
\leq
F(x) + \frac{b_{1}}{\epsilon_{1}}.
\end{equation*}

Next, let us prove \eqref{eq:drift-b}. 
By the definition of $\mathcal{L}$ in \eqref{eqn:generator}, 
we can compute that
\begin{align*}
\mathcal{L}F(x)
&=-A_{J}\nabla F(x)\cdot\nabla F(x)
+\beta^{-1}\Delta F(x)
\\
&=-\Vert\nabla F(x)\Vert^{2}+\beta^{-1}\Delta F(x),
\end{align*}
since $J$ is anti-symmetric so that $\langle\nabla F(x),J\nabla F(x)\rangle=0$.
Moreover, 
\begin{equation}\label{eqn:similar:1}
\Vert x\Vert\cdot\Vert\nabla F(x)\Vert
\geq\langle x,\nabla F(x)\rangle
\geq m\Vert x\Vert^{2}-b,
\end{equation}
implies that
\begin{equation}\label{eqn:similar:2}
\Vert\nabla F(x)\Vert
\geq m\Vert x\Vert-\frac{b}{\Vert x\Vert}\geq\frac{1}{2}m\Vert x\Vert,
\end{equation}
provided that $\Vert x\Vert\geq\sqrt{2b/m}$, 
and thus 
\begin{equation}\label{eqn:similar:3}
\mathcal{L}F(x)
\leq-\frac{m^{2}}{4}\Vert x\Vert^{2}+\beta^{-1}\Delta F(x)
\leq
-\frac{m^{2}}{4}\Vert x\Vert^{2}+\frac{mb}{2}+\beta^{-1}\Delta F(x),
\end{equation}
for any $\Vert x\Vert\geq\sqrt{2b/m}$.
On the other hand, for any $\Vert x\Vert\leq\sqrt{2b/m}$, 
we have 
\begin{equation}\label{eqn:similar:4}
\mathcal{L}F(x)\leq\beta^{-1}\Delta F(x)
\leq-\frac{m^{2}}{4}\Vert x\Vert^{2}+\frac{mb}{2}+\beta^{-1}\Delta F(x).
\end{equation}
Hence, for any $x\in\mathbb{R}^{d}$,
\begin{equation}\label{eqn:similar:5}
\mathcal{L}F(x)
\leq-\frac{m^{2}}{4}\Vert x\Vert^{2}+\frac{mb}{2}+\beta^{-1}\Delta F(x).
\end{equation}
Next, recall that $F$ is $M$-smooth, and thus
\begin{equation*}
\Delta F(x)\leq Md.
\end{equation*}
Finally, by Lemma \ref{lem:gradient-bound},
\begin{equation*}
F(x)
\leq\frac{M}{2}\Vert x\Vert^{2}+B\Vert x\Vert+A
\leq\frac{M+B}{2}\Vert x\Vert^{2}+\frac{B}{2}+A.
\end{equation*}
Therefore, we have
\begin{equation*}
\mathcal{L}F(x)
\leq
-\frac{m^{2}}{2(M+B)}F(x)+\frac{m^{2}(\frac{B}{2}+A)}{2(M+B)}+\frac{mb}{2}+M\beta^{-1}d.
\end{equation*}
Hence, the proof is complete.
\end{proof}

%%%%%%%%%%

\begin{proof}[Proof of Corollary~\ref{cor:L2:continuous}]
Recall from Lemma \ref{lem:gradient-bound} that 
\begin{equation*}
F(x)\geq\frac{m}{2}\Vert x\Vert^{2}-\frac{b}{2}\log 3,
\end{equation*}
which implies that
\begin{equation*}
\Vert x\Vert^{2}\leq\frac{2}{m}F(x)+\frac{b}{m}\log 3.
\end{equation*}
It then follows from Lemma~\ref{lem:cont-anti-L2} that
\begin{equation*}
\mathbb{E} [\Vert X(t)\Vert^{2}]
\leq
\frac{2}{m}F(x) 
+\frac{B}{m}+\frac{2A}{m}
+\frac{2b(M+B)}{m^{2}}+\frac{4M\beta^{-1}d(M+B)}{m^{3}}+\frac{b}{m}\log 3.
\end{equation*}
Recall that $\Vert X(0)\Vert=\Vert x\Vert\leq R$ and by Lemma \ref{lem:gradient-bound} we get
$F(x)\leq\frac{M}{2}\Vert x\Vert^{2}+B\Vert x\Vert+A$, and thus
\begin{equation*}
\mathbb{E} [\Vert X(t)\Vert^{2}]
\leq
C_{c}=\frac{MR^{2}+2BR+B+4A}{m}
+\frac{2b(M+B)}{m^{2}}+\frac{4M\beta^{-1}d(M+B)}{m^{3}}+\frac{b}{m}\log 3.
\end{equation*}
\end{proof}

%%%%%%%%%%%
\begin{proof}[Proof of Lemma~\ref{lem:disc-anti-L2}]
Suppose we have
\begin{equation}  \label{eq:drift-d}
\frac{\mathbb{E }_x [ F(X_1) ] - F(x)}{\eta} \le -\epsilon_{2} F(x) + b_{2},
\end{equation}
uniformly for small $\eta$, where $\epsilon_{2}, b_{2}$ are positive constants that are independent of $\eta$, 
then we will first show below that 
\begin{equation*}
\mathbb{E}_x[F(X_k)] \le F(x) + \frac{b_{2}}{\epsilon_{2}}.
\end{equation*}
We will use the discrete Dynkin's formula (see, e.g. Section 4.2 in \cite{Meyn92}).
Let $\mathbb{F}_{i}$ denote the filtration generated by $X_0, \ldots, X_i$. 
Note $\{X_k: k \ge 0\}$ is a time-homogeneous Markov process, so
the drift condition \eqref{eq:drift-d} implies that
\begin{equation*} 
\mathbb{E} [ F(X_i) | \mathbb{F}_{i-1}] \le (1- \eta \epsilon_{2}) F(X_{i-1}) + b_{2}.
\end{equation*}
Then by letting $r=1/(1- \eta \epsilon_{2})\geq 1$, which requires $\eta\leq\frac{1}{\varepsilon_{2}}$, we obtain
\begin{equation*} 
\mathbb{E}\left [ r F(X_i) | \mathbb{F}_{i-1}\right] \le  F(X_{i-1}) + rb_{2}.
\end{equation*}
Then we can compute that
\begin{equation} \label{ineq:dft}
\mathbb{E} \left[ r^{i} F(X_i) | \mathbb{F}_{i-1}\right] - r^{i-1}F(X_{i-1})
= r^{i-1} \cdot \left[ \mathbb{E } [ r F(X_i) | \mathbb{F}_{i-1}] - F(X_{i-1}) \right] 
\le   r^i b_{2}.
\end{equation}
Define the stopping time $\tau_{k,K} = \min \{k, \inf\{i: |X_i| \ge K\}\}$, where $K$ is a positive integer, so that
$X_i$ is essentially bounded for $i \le \tau_{k,K}$. 
Applying the discrete Dynkin's formula (see, e.g. Section 4.2 in \cite{Meyn92}), we have
\begin{equation*}
\mathbb{E}_x\left[r^{\tau_{k,K}} F(X_{\tau_{k,K}})\right] = \mathbb{E}_x\left[F(X_0)\right] +
 \mathbb{E } \left[ \sum_{i=1}^{\tau_{k,K}}\left(\mathbb{E} [ r^i F(X_i) | \mathbb{F}_{i-1}] - r^{i-1}F(X_{i-1})\right)
\right].
\end{equation*}
Then it follows from \eqref{ineq:dft} that
\begin{equation*}
\mathbb{E}_x\left[r^{\tau_{k,K}} F(X_{\tau_{k,K}})\right] \leq F(x) +b_{2} \eta \sum_{i=1}^{k}  r^i.
\end{equation*}
As $\tau_{k,K} \rightarrow k$ almost surely as $K \rightarrow \infty$, we infer from Fatou's Lemma that
\begin{equation*}
\mathbb{E}_x\left[r^{k} F(X_{k})\right] \le F(x) +b_{2} \eta \sum_{i=1}^{k}  r^i,
\end{equation*}
which implies that for all $k$,
\begin{equation*}
\mathbb{E}_x\left[F(X_k)\right] \le F(x) + \frac{b_{2} \eta}{r-1} 
= F(x) + \frac{b_{2} (1-\eta_{2} \epsilon_{2})}{\epsilon_{2}} 
\le F(x) + \frac{b_{2}}{\epsilon_{2}},
\end{equation*}
as $r=1/(1- \eta_{2} \epsilon_{2})$. 
Hence we have
\begin{equation*}
\mathbb{E}_x\left[F(X_k)\right] \le F(x) + \frac{b_{2}}{\epsilon_{2}}.
\end{equation*}

It remains to prove \eqref{eq:drift-d}. 
Note that as $\nabla F$ is Lipschitz continuous with constant $M$ so that:
\begin{equation*}
F(y) \le F(x) + \nabla F(x) (y-x) + \frac{M}{2} \Vert y-x\Vert^2.
\end{equation*}
Therefore,
\begin{align*}
\frac{\mathbb{E }_x [ F(X_1) ] - F(x)}{\eta} 
&=\frac{1}{\eta}\left(\mathbb{E }_x \left[ F(x-\eta A_{J}(\nabla F(x))+\sqrt{2\eta\beta^{-1}}\xi_{0}) \right] - F(x)\right)
\\
&\leq
-\nabla F(x)A_{J}\nabla F(x)
+\frac{M}{2\eta}\mathbb{E}_{x}\left[\left\Vert-\eta A_{J}(\nabla F(x))+\sqrt{2\eta\beta^{-1}}\xi_{0}\right\Vert^{2}\right]
\\
&=-\Vert\nabla F(x)\Vert^{2}
+\frac{M}{2}\eta\Vert A_{J}\nabla F(x)\Vert^{2}
+M\beta^{-1}d
\\
&\leq
-\frac{1}{2}\Vert\nabla F(x)\Vert^{2}+M\beta^{-1}d\,,
\end{align*}
provided that $\frac{M}{2}\Vert A_{J}\Vert^{2}\eta\leq\frac{1}{2}$.
Similar to the arguments in \eqref{eqn:similar:1}-\eqref{eqn:similar:5}, we get
\begin{equation*}
\frac{\mathbb{E }_x [ F(X_1) ] - F(x)}{\eta} 
\leq
-\frac{m^{2}}{8}\Vert x\Vert^{2}+M\beta^{-1}d+\frac{mb}{4}.
\end{equation*}
Finally, by Lemma \ref{lem:gradient-bound},
\begin{equation*}
F(x)
\leq\frac{M}{2}\Vert x\Vert^{2}+B\Vert x\Vert+A
\leq\frac{M+B}{2}\Vert x\Vert^{2}+\frac{B}{2}+A.
\end{equation*}
Therefore, we have
\begin{equation*}
\frac{\mathbb{E }_x [ F(X_1) ] - F(x)}{\eta} 
\leq
-\frac{m^{2}}{4(M+B)}F(x)+\frac{m^{2}(\frac{B}{2}+A)}{4(M+B)}+M\beta^{-1}d+\frac{mb}{4},
\end{equation*}
where we assumed that $1-\eta\frac{m^{2}}{4(M+B)}\in[0,1)$.
Hence, the proof is complete.
\end{proof}

\begin{proof}[Proof of Corollary \ref{cor:L2:discrete}]
The proof is similar to the proof of Corollary \ref{cor:L2:continuous}
and is thus omitted.
\end{proof}

%%%%%%%%%%%%%%%%%%%%%%%%%%%%%%%%%
\section{Proof of Proposition~\ref{prop:eigenvalue-compare}
and Proposition~\ref{prop:anti-disc-cont}}

\begin{proof}[Proof of Proposition~\ref{prop:eigenvalue-compare}]
Write $u$ as the corresponding eigenvector of $A_{J} \mathbb{L}^{\sigma}$ for the eigenvalue $- \mu_J^{\ast}<0$, so we have
\begin{equation}\label{eq-eigvec-product}
A_{J} \mathbb{L}^{\sigma}  u = - \mu_J^{\ast} u. 
\end{equation} 
Then it follows that
\begin{equation*}
(-\mu_J^{\ast}) u^{\ast}\mathbb{L}^{\sigma} u = u^{\ast} \mathbb{L}^{\sigma} (- \mu_J^{\ast} u) = u^{\ast}\mathbb{L}^{\sigma} A_{J} \mathbb{L}^{\sigma} u=  u^{\ast} (\mathbb{L}^{\sigma})^{T} A_{J}\mathbb{L}^{\sigma}u  = |\mathbb{L}^{\sigma} u|^2 + u^{\ast}(\mathbb{L}^{\sigma})^{T} J \mathbb{L}^{\sigma} u ,
\end{equation*}
where $u^{\ast}$ denotes the conjugate transpose of $u$, $(\mathbb{L}^{\sigma})^{T}$ denotes the transpose of $\mathbb{L}^{\sigma}$, and $(\mathbb{L}^{\sigma})^{T} = \mathbb{L}^{\sigma}$ as $\mathbb{L}^{\sigma}$ is a real symmetric matrix. 
It is easy to see that $u^{\ast}\mathbb{L}^{\sigma} u$ is a real number as $(u^{\ast}\mathbb{L}^{\sigma} u)^{\ast} = u^{\ast}\mathbb{L}^{\sigma}u$. In addition, $u^{\ast}(\mathbb{L}^{\sigma})^{T} J \mathbb{L}^{\sigma} u$ is pure imaginary, since
$(u^{\ast}(\mathbb{L}^{\sigma})^{T} J \mathbb{L}^{\sigma}u)^{\ast} = u^{\ast}(\mathbb{L}^{\sigma})^{T} J^T \mathbb{L}^{\sigma}u =- u^{\ast}(\mathbb{L}^{\sigma})^{T} J\mathbb{L}^{\sigma} u$ by the fact that $J$ is an anti-symmetric real matrix. Hence, we deduce that
\begin{equation*}
u^{\ast}(\mathbb{L}^{\sigma})^{T} J \mathbb{L}^{\sigma} u=0,
\end{equation*}
and it implies that 
\begin{equation} \label{eq:1}
(-\mu_J^{\ast}) u^{\ast}\mathbb{L}^{\sigma} u  = |\mathbb{L}^{\sigma} u|^2.
\end{equation}
Note $u^{\ast}\mathbb{L}^{\sigma}u \ne 0$ as otherwise $0$ 
becomes an eigenvalue of $\mathbb{L}^{\sigma}$ from \eqref{eq:1}, which is a contradiction. 
In fact, we obtain from \eqref{eq:1} that $- u^{\ast}\mathbb{L}^{\sigma} u > 0$ 
as $\mu_J^{\ast}>0$ and $|\mathbb{L}^{\sigma} u|^2>0$.

Since $\mathbb{L}^{\sigma}$ is a real symmetric matrix, we have 
\begin{equation}\label{eqn:S}
\mathbb{L}^{\sigma}=S^{T}DS,    
\end{equation}
for a real orthogonal matrix $S$, where $D= \text{diag} (\mu_1, \mu_2, \ldots, \mu_d)$ 
with $\mu_1< 0 < \mu_2 < \ldots < \mu_d$ being the eigenvalues of $\mathbb{L}^{\sigma}$. 
Then we obtain
\begin{equation} \label{eq:compare}
\mu_J^{\ast} = \frac{|\mathbb{L}^{\sigma} u|^2}{- u^{\ast} \mathbb{L}^{\sigma} u} = \frac{u^{\ast} S^{\ast} D^2 Su}{ -u^{\ast}S^{\ast} D Su} = \frac{\sum_{i=1}^d \mu_i^2 |(Su)_i|^2 }{ \sum_{i=1}^d - \mu_i |(Su)_i|^2},
\end{equation}
where $(Su)_i$ denotes the $i$-th component of the vector $Su$.
Since $\mu_1< 0 < \mu_2 < \ldots < \mu_d$, we then have $(Su)_1 \ne 0$ as otherwise $- u^{\ast}\mathbb{L}^{\sigma} u =  \sum_{i=1}^n - \mu_i |(Su)_i|^2 \le 0$, which is a contradiction. Therefore, we conclude from \eqref{eq:compare} that 
\beq\label{eq-mu-star-bound} \mu_J^{\ast} \ge |\mu_1| = \mu^{\ast}(\sigma).
\eeq
The equality $\mu_J^{\ast} = |\mu_1|= \mu^{\ast}(\sigma)$ is attained if and only if $(Su)_i =0$ for $i=2, \ldots, n$. Or equivalently if and only if the vector $Su = a e_1$ where $a$ is a non-zero constant and $e_1 = [1 ~ 0 ~ \dots~ 0]^T$ is the first basis vector. Since  $S^{-1} = S^T$, this is also equivalent to $u = a v $ where $v = S^Te_1$ is an eigenvector of $\mathbb{L}^\sigma$ corresponding to the eigenvalue $\mu_1$. Since $u$ and $v$ are related up to a constant, this is the same as saying $v$ is an eigenvector of $A_J \mathbb{L}^\sigma$ satisfying \eqref{eq-eigvec-product}. Since $v$ is also an eigenvalue of $\mathbb{L}^\sigma$ and $J$ being anti-symmetric, has only purely imaginary eigenvalues except a zero eigenvalue, this is if and only if $Jv = 0$. In other words, the equality $\mu_J^{\ast} = |\mu_1|= \mu^{\ast}(\sigma)$ is attained if and only if the eigenvector of  $\mathbb{L}^\sigma$ corresponding to the negative eigenvalue $\mu_1$ is an eigenvector of $J$ for the eigenvalue 0. 

We note finally that  Equation~\eqref{eq:exit-compare-cont} then readily follows from \eqref{eq:exit-time} and \eqref{eq-mu-star-bound}.
\end{proof}

%%%%%%%%%%%%%%%%%%%%%%%%%%%%%%%%%%%%%%%%%%%%%%%%%%%%

\begin{proof}[Proof of Proposition~\ref{prop:anti-disc-cont}]
Write $\tau_{a_1 \rightarrow a_2}^{\beta, n}$ for the first time that the continuous-time dynamics $\{X(t)\}$ starting from $a_1$ to exit the region $D_n$.
Then by monotone convergence theorem, we have
\begin{equation*} 
\lim_{R\rightarrow \infty} \mathbb{E}\left[ \tau_{a_1 \rightarrow a_2}^{\beta, n} \right]  = \mathbb{E}\left[ \tau_{a_1 \rightarrow a_2}^{\beta}\right].
\end{equation*}
Hence, for fixed $\epsilon>0$, one can choose a sufficiently large $n$ such that 
\begin{equation} \label{eq:largeR}
 \left|  \mathbb{E}\left[\tau_{a_1 \rightarrow a_2}^{\beta, n} \right] - \mathbb{E}\left[ \tau_{a_1 \rightarrow a_2}^{\beta}\right] \right| < \epsilon.
\end{equation}
We next control the expected difference between the exit times $\hat \tau_{a_1 \rightarrow a_2}^{\beta, n}$ of the discrete dynamics, and $\tau_{a_1 \rightarrow a_2}^{\beta, n}$ of the continuous dynamics, from the bounded domain $D_n$. For fixed $\epsilon$ and large $n$, we can infer from Theorem 4.2 in \cite{Gobet2005} that\footnote{The Assumption (H2') in Theorem 4.2 of \cite{Gobet2005} can be readily verified in our setting: for both reversible and non-reversible SDE, the drift and diffusion coefficients are clearly Lipschitz; the diffusion matrix is uniformly elliptic; and the domain $D_n$ is bounded and it satisfies the exterior cone condition.}, for sufficiently small stepsize $\eta \le \bar \eta(\epsilon, n, \beta)$,
\begin{equation}\label{eq:disc-con}
\left|  \mathbb{E}\left[\hat \tau_{a_1 \rightarrow a_2}^{\beta, n} \right] - \mathbb{E}\left[\tau_{a_1 \rightarrow a_2}^{\beta, n}\right] \right| < \epsilon.
\end{equation}
Together with \eqref{eq:largeR}, we obtain for $\eta$ sufficiently small,
\begin{equation*}
\left|  \mathbb{E}\left[ \hat \tau_{a_1 \rightarrow a_2}^{\beta, n} \right] - \mathbb{E}\left[\tau_{a_1 \rightarrow a_2}^{\beta}\right] \right| < 2 \epsilon.
\end{equation*}
The proof is therefore complete.
\end{proof}

%%%%%%%%%%%%%%%%%%%%%%%%%%%%%%%%%%%%%%%%%%%
%\begin{comment}
\section{Recurrence and escape times for underdamped Langevin dynamics with small friction}\label{sec:small}

In this section, we investigate the local analysis results for 
the underdamped Langevin dynamics \eqref{eq:V-iterate}-\eqref{eq:X-iterate}
when the friction coefficient $\gamma$ is small, 
and in particular, we assume that $0<\gamma<2\sqrt{m}$. 

%%%%%%%%%%%%%%%%%%%%%%%%%%%%%%%%%%%%%%%%%%
%\subsection{Main results}

\begin{theorem}\label{MainThm1:small}
Fix $\gamma<2\sqrt{m}$, $\delta\in(0,1)$ and $r>0$. Assume 
\begin{equation*}
0<\varepsilon<\overline{\varepsilon}^{U}
:=\min\left\{\frac{\sqrt{m}(1-\hat{\varepsilon})}{4C_{\hat{\varepsilon}}L(1+\frac{1}{64C_{\hat{\varepsilon}}^{2}})},8rC_{\hat{\varepsilon}}\right\}
=\min\left\{\mathcal{O}\left(m\right),\mathcal{O}\left(\frac{r}{\sqrt{m}}\right)\right\},
\end{equation*}
where $\hat{\varepsilon}$ and $C_{\hat{\varepsilon}}$ are defined in Lemma \ref{lem:keylemma}, and
we assumed that $\hat{\varepsilon}=\Omega(1)$ and thus $C_{\hat{\varepsilon}}$ is of order $m^{-1/2}$.
Define the recurrence time
\begin{equation}\label{rec:time:U:small}
\mathcal{T}_{\text{rec}}^{U}:=\frac{2}{\sqrt{m}(1-\hat{\varepsilon})}\log\left(\frac{8rC_{\hat{\varepsilon}}}{\varepsilon}\right)
=\mathcal{O}\left(\frac{1}{\sqrt{m}}\log\left(\frac{r}{\varepsilon m}\right)\right),
\end{equation}
%which is well-defined and positive since $\varepsilon<8rC_{\hat{\varepsilon}}$, 
and we also define the escape time
\begin{equation*}
\mathcal{T}_{\text{esc}}^{U}:=\mathcal{T}_{\text{rec}}^{U}+\mathcal{T},
\end{equation*}
for any arbitrary $\mathcal{T}>0$. 

Consider an arbitrary initial point $x$ for the underdamped Langevin dynamics and a local minimum $x_{\ast}$ at a distance at most $r$. 
Assume that the stepsize $\eta$ satisfies
\begin{equation*}
\eta\leq
\overline{\eta}^{U}
=
\min\left\{\mathcal{O}(\varepsilon),
\mathcal{O}\left(\frac{m^{2}\beta\delta\varepsilon^{2}}
{(md+\beta)\mathcal{T}_{\text{rec}}^{U}}\right),
\mathcal{O}\left(\frac{m^{5/4}\delta}{(d+\beta)^{1/2}(\mathcal{T}_{\text{esc}}^{U})^{1/2}}\right),
\mathcal{O}(m^{5/2})\right\}\,,
\end{equation*}
where %$\overline{\eta}^{U}$ is more formally defined in Table \ref{table_constants_1} and 
$\beta$ satisfies
\begin{equation*}
\beta\geq
\underline{\beta}^{U}=
\max\left\{\Omega\left(\frac{(d+\log((\mathcal{T}+1)/\delta))C_{\hat{\varepsilon}}^{2}}{\sqrt{m}(1-\hat{\varepsilon})\varepsilon^{2}}\right),
\Omega\left(\frac{d\eta m^{1/2}\log(\delta^{-1}\mathcal{T}_{\text{rec}}^{U}/\eta)}{\varepsilon^{2}}\right)\right\},
\end{equation*}  
%where more formally $\underline{\beta}^{U}=\max\left\{\underline{\beta}_{3}^{U},\underline{\beta}_{2}^{U}\right\}$ 
%and $\underline{\beta}_{2}^{U}$
%is defined in Table \ref{table_constants_1}
%and
%\begin{equation*}
%\underline{\beta}_{3}^{U}:=
%\frac{128\gamma C_{\hat{\varepsilon}}^{2}}{\sqrt{m}(1-\hat{\varepsilon})\varepsilon^{2}}\left(d\log(2)
%+\log\left(\frac{6(\gamma^{2} + M^{2} + 1)^{1/2}\mathcal{T}+3}{\delta}\right)\right),
%\end{equation*}  
for any realization of $Z$, with probability at least $1-\delta$
w.r.t. the Gaussian noise, at least one of the following events will occur:
(1) $\Vert X_{k}-x_{\ast}\Vert\geq\frac{1}{2}\left(\varepsilon+re^{-\sqrt{m}(1-\hat{\varepsilon}) k\eta}\right)$
for some $k\leq\eta^{-1}\mathcal{T}_{\text{rec}}^{U}$;
(2) $\Vert X_{k}-x_{\ast}\Vert\leq\varepsilon+re^{-\sqrt{m}(1-\hat{\varepsilon}) k\eta}$
for every $\eta^{-1}\mathcal{T}_{\text{rec}}^{U}\leq k\leq\eta^{-1}\mathcal{T}_{\text{esc}}^{U}$.
\end{theorem}

%%%%%%%%%%%%%%%%%%%%%%%%%%%%%%%%%%%%%%%%%%%%

\begin{remark}
Notice that in Theorem \ref{MainThm1:small}, the definition of $\eta$ and $\beta$ are coupled
since $\overline{\eta}^{U}$ depends on $\beta$ and $\underline{\beta}^{U}$ depends on $\eta$.
A closer look reveals that when $\eta$ is sufficiently small, 
the first term in the definition of $\underline{\beta}^{U}$ dominates the second term
and $\underline{\beta}^{U}$ is independent of $\eta$. So to satisfy
the constraints in Theorem \ref{MainThm1:small}, it suffices to first choose
$\beta$ to be larger than the first term in $\underline{\beta}^{U}$ and then choose
$\eta$ to be sufficiently small.
\end{remark}

%\begin{remark}
%We can apply Theorem \ref{MainThm1:small} to obtain Theorem \ref{MainThm2} for the population risk
%in Section \ref{sec:population}, similar to Theorem 3 in \cite{pmlr-v75-tzen18a}.
%\end{remark}

%\begin{remark} 
%In \cite{pmlr-v75-tzen18a}, they used $\Vert X_{k}-x_{\ast}\Vert_{H}$.
%In Theorem \ref{MainThm1:small}, we used $\Vert X_{k}-x_{\ast}\Vert$ instead. 
%Note that the relation $\Vert\cdot\Vert_{H}\geq\sqrt{m}\Vert\cdot\Vert$ holds.
%\end{remark}

\begin{remark} 
In \cite{pmlr-v75-tzen18a}, the overdamped Langevin algorithm is used
and the recurrence time 
$\mathcal{T}_{\text{rec}}=\mathcal{O}\left(\frac{1}{m}\log(\frac{r}{\varepsilon})\right)$,
and thus our recurrence time 
$\mathcal{T}_{\text{rec}}^{U}=\mathcal{O}\left(\frac{1}{\sqrt{m}}\log\left(\frac{r}{\varepsilon m}\right)\right)$
for the underdamped Langevin algorithm with the choice of $\gamma<2\sqrt{m}$, 
which has a square root factor improvement ignoring the logarithmic factor. This recurrence time
is worse than $\mathcal{T}_{\text{rec}}^{U}=\mathcal{O}\left(\frac{1}{\sqrt{m}}\log\left(\frac{r}{\varepsilon }\right)\right)$
for the underdamped Langevin algorithm with the choice of $\gamma=2\sqrt{m}$
by a logarithmic factor assuming $C_{H}=\mathcal{O}(1)$.
\end{remark} 

%%%%%%%%%%%%%%%%%%%%%%%%%%%%%%%%%%%%%%%%%%%%

\begin{remark}
Let us compare the case $\gamma=2\sqrt{m}$ with the case $\gamma<2\sqrt{m}$ (to be discussed in Section \ref{sec:small}).
When $\gamma=2\sqrt{m}$, 
since $W_{-1}(-x)\sim\log(1/x)$ for $x\rightarrow 0^{+}$, assuming $r,\varepsilon,C_{H}=\mathcal{O}(1)$,
we have
\begin{equation*}
\mathcal{T}_{\text{rec}}^{U}\sim
\frac{1}{2\sqrt{m}}\log(1/m),
\qquad\text{as $m\rightarrow 0^{+}$}.
\end{equation*}
When $\gamma<2\sqrt{m}$, we have
$\mathcal{T}_{\text{rec}}^{U}=\frac{2}{\sqrt{m}(1-\hat{\varepsilon})}\log(8rC_{\hat{\varepsilon}}/\varepsilon)$,
where $C_{\hat{\varepsilon}}=\frac{1+M}{\sqrt{m(1-(1-\hat{\varepsilon})^{2})}}$,
and $\hat{\varepsilon}=1-\frac{\gamma}{2\sqrt{m}} \in (0,1)$.
For example, if $\gamma=m^{\chi}$,
where $\chi\in(0,1/2)$ and $m\rightarrow 0^{+}$, then
$\mathcal{T}_{\text{rec}}^{U}\sim\frac{2}{m^{\chi}}\log(m)$,
as $m\rightarrow 0^{+}$,
and if $\gamma=\gamma_{0}\sqrt{m}$, where $\gamma_{0}\in(0,2)$, then
$\mathcal{T}_{\text{rec}}^{U}\sim\frac{2}{\gamma_{0}\sqrt{m}}\log(m)$,
as $m\rightarrow 0^{+}$.
To summarize, with all the parameters fixed,
as $m\rightarrow 0^{+}$, the choice $\gamma=2\sqrt{m}$
is more optimal than the choice $\gamma<2\sqrt{m}$.
On the other hand,
when the second smallest eigenvalue is close to the smallest eigenvalue $m$, 
such that $C_{H}=\max_{i:\lambda_{i}>m}\frac{(1+\lambda_{i})^{2}}{\lambda_{i}-m}$ is large, 
it is more desirable to use the underdamped Langevin
algorithm with $\gamma<2\sqrt{m}$ instead.
\end{remark}

%%%%%%%%%%%%%%%%%%%%%%%%%%%%%%%%%%%%%%%
\subsection{Proof of Theorem \ref{MainThm1:small}}

The proof of Theorem \ref{MainThm1:small} is similar
to the proof of Theorem \ref{MainThm1} and the following proposition, 
and the similar arguments in Section \ref{sec:complete:proof}.

\begin{proposition}\label{prop:key:small}
Assume $\gamma<2\sqrt{m}$.
Fix any $r>0$ and 
\begin{equation*}
\varepsilon<
\min\left\{\frac{\sqrt{m}(1-\hat{\varepsilon})}{4C_{\hat{\varepsilon}}L(1+\frac{1}{64C_{\hat{\varepsilon}}^{2}})},8rC_{\hat{\varepsilon}}\right\}.
\end{equation*}
Consider the stopping time:
\begin{equation*}
\tau:=\inf\left\{t\geq 0:\Vert X(t)-x_{\ast}\Vert\geq\varepsilon+re^{-\sqrt{m}(1-\hat{\varepsilon}) t}\right\}.
\end{equation*}
For any initial point $X(0)=x$ with $\Vert x-x_{\ast}\Vert \le r$, and 
\begin{equation*}
\beta\geq\frac{128\gamma C_{\hat{\varepsilon}}^{2}}{\sqrt{m}(1-\hat{\varepsilon})\varepsilon^{2}}\left(d\log(2)+\log\left(\frac{2\Vert H_{\gamma}\Vert \mathcal{T}+1}{\delta}\right)\right),
\end{equation*}
we have
\begin{equation*}
\mathbb{P}_{x}\left(\tau\in\left[\mathcal{T}_{\text{rec}}^{U},\mathcal{T}_{\text{esc}}^{U}\right]\right)\leq\delta.
\end{equation*}
\end{proposition}

The term $\Vert H_{\gamma}\Vert$ in Proposiiton \ref{prop:key:small}
can be bounded using Lemma \ref{lem:H:gamma}.
Based on Proposition \ref{prop:key:small}, the proof of 
Theorem \ref{MainThm1:small} is similar to the proof of Theorem \ref{MainThm1}.
So in the rest of the section, we will only focus
on the proof of Proposition \ref{prop:key:small}.

%%%%%%%%%%%%%%%%%%%%%%%%%%%%%%%%%%%%%%%%%
\subsubsection{Proof of Proposition~\ref{prop:key:small}}\label{proof:two:results:small}

In this section, we focus on the proof of Proposition \ref{prop:key:small}.
We recall some definitions from Section \ref{proof:two:results}.
We recall the matrices $H_{\gamma}$ and $I^{(2)}$ from \eqref{eqn:H:I},
the matrix flow
$Q_{t_{0}}(t)$
from \eqref{eqn:matrix:flow} 
and the processes $Z_{t}^{0}$ and $Z_{t}^{1}$ from \eqref{eqn:Z0}-\eqref{eqn:Z1},
and also $\mu_{t}$ and $\Sigma_{t}$ from \eqref{eqn:mu:t}-\eqref{eqn:Sigma:t}.

\begin{lemma}\label{lem:ineq:small}
Assume $\gamma<2\sqrt{m}$. 
For any $\theta \in \left(0,\frac{\sqrt{m}(1-\hat{\varepsilon})}{\gamma C_{\hat{\varepsilon}}^{2}}\right)$,
and $h>0$ and any $(V(0),Y(0))$,
\begin{equation*}
\mathbb{P}\left(\sup_{t_{0}\leq t\leq t_{1}}\Vert Q_{t_{0}}(t_{1})Z_{t}^{0}\Vert\geq h\right)
\leq \left(1-\frac{\theta\gamma C_{\hat{\varepsilon}}^{2}}{\sqrt{m}(1-\hat{\varepsilon})}\right)^{-d} 
e^{-\frac{\beta\theta}{2}[h^{2}-\langle\mu_{t_{1}},(I-\beta\theta\Sigma_{t_{1}})^{-1}\mu_{t_{1}}\rangle]}.
\end{equation*}
\end{lemma}

\begin{proof}[Proof of Lemma~\ref{lem:ineq:small}]
The proof is similar to the proof of Lemma \ref{lem:ineq}.
Let us recall from Lemma \ref{lem:keylemma}
that if $\gamma<2\sqrt{m}$, then
\begin{equation*}
\left\| e^{-t H_\gamma}\right\| \leq C_{\hat{\varepsilon}} e^{-\sqrt{m}(1-\hat{\varepsilon})t}\,,
\end{equation*} 
where $C_{\hat{\varepsilon}}$ and $\hat{\varepsilon}$ are defined in Lemma \ref{lem:keylemma}.
Therefore, we have
\begin{equation*}
\Vert\Sigma_{t_{1}}\Vert
\leq 2\gamma\beta^{-1}\int_{0}^{t_{1}}C_{\hat{\varepsilon}}^{2}e^{-2\sqrt{m}(1-\hat{\varepsilon})t}dt
\leq\frac{\gamma\beta^{-1}C_{\hat{\varepsilon}}^{2}}{\sqrt{m}(1-\hat{\varepsilon})}.
\end{equation*}
Therefore we infer that 
the eigenvalues of $I-\beta\theta\Sigma$ are bounded below by 
$1-\theta\frac{\gamma C_{\hat{\varepsilon}}^{2}}{\sqrt{m}(1-\hat{\varepsilon})}$. 
The conclusion then follows from \eqref{eq:est1}.
\end{proof}

%%%%%%%%%%%%%%%%%%%%%%%%%
\begin{proof}[Proof of Proposition~\ref{prop:key:small}]
Since $\Vert Y(0)\Vert\leq r$, we know that $\tau>0$. 
Fix some $\mathcal{T}_{\text{rec}}^{U}\leq t_{0}\leq t_{1}$,
such that $t_{1}-t_{0}\leq\frac{1}{2\Vert H_{\gamma}\Vert}$.
Then, for every $t\in[t_{0},t_{1}]$,
\begin{equation*}
\Vert Y(t)\Vert
\leq\left\Vert e^{(t_{1}-t)H_{\gamma}}Q_{t_{0}}(t_{1})Z_{t}\right\Vert
\leq e^{\frac{1}{2}}\left\Vert Q_{t_{0}}(t_{1})Z_{t}\right\Vert.
\end{equation*}
Recall that $\gamma<2\sqrt{m}$. 
Similar to the derivations in \eqref{bound:two:terms}, we get
\begin{align}
\mathbb{P}(\tau\in[t_{0},t_{1}])
&\leq
\mathbb{P}\left(\sup_{t_{0}\leq t\leq t_{1}\wedge\tau}\frac{\Vert Q_{t_{0}}(t_{1})Z_{t}^{0}\Vert}
{\varepsilon+re^{-\sqrt{m}(1-\hat{\varepsilon}) t}}\geq c_{0},\tau\geq t_{0}\right)
\nonumber
\\
&\qquad\qquad\qquad
+\mathbb{P}\left(\sup_{t_{0}\leq t\leq t_{1}\wedge\tau}\frac{\Vert Q_{t_{0}}(t_{1})Z_{t}^{1}\Vert}
{\varepsilon+re^{-\sqrt{m}(1-\hat{\varepsilon}) t}}\geq c_{1},\tau\geq t_{0}\right),
\label{bound:two:terms:small}
\end{align}
where $c_{0}+c_{1}=\frac{1}{2}$ and $c_{0},c_{1}>0$.

First, we show that the second term in \eqref{bound:two:terms:small} is zero. 
On the event $\tau\in[t_{0},t_{1}]$, for any $0\leq s\leq t_{1}\wedge\tau$,
we have
\begin{equation*}
\Vert\rho(Y(s))\Vert\leq\frac{L}{2}\Vert Y(s)\Vert^{2}
\leq\frac{L}{2}\left(\varepsilon+re^{-\sqrt{m}(1-\hat{\varepsilon}) s}\right)^{2}.
\end{equation*}
Therefore, for any $t\in[t_{0},t_{1}\wedge\tau]$, since $\gamma<2\sqrt{m}$,
by Lemma \ref{lem:keylemma}, we get
\begin{align*}
\left\Vert Q_{t_{0}}(t_{1})Z_{t}^{1}\right\Vert
&\leq\int_{0}^{t}\left\Vert e^{ (s-t_{1}) H_{\gamma}}\right\Vert\cdot\Vert\rho(Y(s))\Vert ds
\\
&\leq\frac{C_{\hat{\varepsilon}}L}{2}\int_{0}^{t}e^{ (s-t_{1})\sqrt{m}(1-\hat{\varepsilon})}\left(\varepsilon+re^{-\sqrt{m}(1-\hat{\varepsilon}) s}\right)^{2}ds
\\
&\leq C_{\hat{\varepsilon}}L\int_{0}^{t}e^{(s-t)\sqrt{m}(1-\hat{\varepsilon})}\left(\varepsilon^{2}+r^{2}e^{-2\sqrt{m}(1-\hat{\varepsilon})s}\right)ds
\\
&\leq\frac{C_{\hat{\varepsilon}}L}{\sqrt{m}(1-\hat{\varepsilon})}\left(\varepsilon^{2}+r^{2}e^{-\sqrt{m}(1-\hat{\varepsilon}) t}\right)
\\
&\leq\frac{C_{\hat{\varepsilon}}L}{\sqrt{m}(1-\hat{\varepsilon})}\varepsilon^{2}\left(1+\frac{1}{64C_{\hat{\varepsilon}}^{2}}\right),
\end{align*}
since $t\geq t_{0}\geq\mathcal{T}_{\text{rec}}^{U}=\frac{2}{\sqrt{m}(1-\hat{\varepsilon})}\log\left(\frac{8rC_{\hat{\varepsilon}}}{\varepsilon}\right)$. 
Consequently, if we take $c_{1}=\frac{C_{\hat{\varepsilon}}L}{\sqrt{m}(1-\hat{\varepsilon})}\varepsilon(1+\frac{1}{64C_{\hat{\varepsilon}}^{2}})$, 
then, 
\begin{equation*}
\sup_{t_{0}\leq t\leq t_{1}\wedge\tau}\frac{\Vert Q_{t_{0}}(t_{1})Z_{t}\Vert}{\varepsilon+re^{-\sqrt{m}(1-\hat{\varepsilon}) t}}
\leq\frac{1}{\varepsilon}\sup_{t_{0}\leq t\leq t_{1}\wedge\tau}\Vert Q_{t_{0}}(t_{1})Z_{t}\Vert\leq c_{1},
\end{equation*}
which implies that 
\begin{equation*}
\mathbb{P}\left(\sup_{t_{0}\leq t\leq t_{1}\wedge\tau}\frac{\Vert Q_{t_{0}}(t_{1})Z_{t}^{1}\Vert}
{\varepsilon+re^{-\sqrt{m}(1-\hat{\varepsilon}) t}}\geq c_{1},\tau\geq t_{0}\right)=0.
\end{equation*}
Moreover, $c_{0}=\frac{1}{2}-c_{1}=\frac{1}{2}-\frac{C_{\hat{\varepsilon}}L}{\sqrt{m}(1-\hat{\varepsilon})}(1+\frac{1}{64C_{\hat{\varepsilon}}^{2}})
\varepsilon>\frac{1}{4}$ since it is assumed $\varepsilon <\frac{\sqrt{m}(1-\hat{\varepsilon})}{4C_{\hat{\varepsilon}}L(1+\frac{1}{64C_{\hat{\varepsilon}}^{2}})}$. 

Second, we apply Lemma \ref{lem:ineq:small} to bound the first term in \eqref{bound:two:terms:small}.
By using $V(0)=0$ and $\Vert Y(0)\Vert\leq r$ and the definition
of $\mu_{t_{1}}$ and $\Sigma_{t_{1}}$ in \eqref{eqn:mu:t} and \eqref{eqn:Sigma:t},
we get
\begin{align*}
\left\langle\mu_{t_{1}},(I-\beta\theta\Sigma_{t_{1}})^{-1}\mu_{t_{1}}\right\rangle
&=
\left\langle e^{-t_{1}H_{\gamma}}(V(0),Y(0))^{T},(I-\beta\theta\Sigma_{t_{1}})^{-1}e^{-t_{1}H_{\gamma}}(V(0),Y(0))^{T}\right\rangle
\\
&\leq\frac{1}{1-\frac{\theta\gamma C_{\hat{\varepsilon}}^{2}}{\sqrt{m}(1-\hat{\varepsilon})}}C_{\hat{\varepsilon}}^{2}e^{-2\sqrt{m}(1-\hat{\varepsilon})t_{1}}r^{2}
\leq
2C_{\hat{\varepsilon}}^{2}e^{-2\sqrt{m}(1-\hat{\varepsilon})t_{1}}r^{2},
\end{align*}
by choosing $\theta=\frac{1}{2}\gamma^{-1}C_{\hat{\varepsilon}}^{-2}\sqrt{m}(1-\hat{\varepsilon})$,
Finally, since 
\begin{equation*}
t_{1}\geq\mathcal{T}_{\text{rec}}^{U}=\frac{2}{\sqrt{m}(1-\hat{\varepsilon})}\log\left(8rC_{\hat{\varepsilon}}/\varepsilon\right)
\geq
\frac{1}{\sqrt{m}(1-\hat{\varepsilon})}\log\left(8rC_{\hat{\varepsilon}}/\varepsilon\right),
\end{equation*}
so that $2C_{\hat{\varepsilon}}^{2}e^{-2\sqrt{m}(1-\hat{\varepsilon})t_{1}}r^{2} \le \frac{1}{32}\varepsilon^{2}$, and thus
\begin{equation*}
\left\langle\mu_{t_{1}},(I-\beta\theta_{t_{1}}\Sigma_{t_{1}})^{-1}\mu_{t_{1}}\right\rangle\leq\frac{1}{32}\varepsilon^{2}.
\end{equation*}
Then with the choice of $h=(\varepsilon+re^{-\sqrt{m}(1-\hat{\varepsilon}) t_{1}})c_{0}$
and $\theta=\frac{1}{2}\gamma^{-1}C_{\hat{\varepsilon}}^{-2}\sqrt{m}(1-\hat{\varepsilon})$ in Lemma \ref{lem:ineq}, 
and notice that $h=(\varepsilon+re^{-\sqrt{m}(1-\hat{\varepsilon}) t_{1}})c_{0}\geq\varepsilon c_{0}$
we get
\begin{align*}
&\mathbb{P}\left(\sup_{t_{0}\leq t\leq t_{1}\wedge\tau}\frac{\Vert Q_{t_{0}}(t_{1})Z_{t}^{0}\Vert}
{\varepsilon+re^{-\sqrt{m}(1-\hat{\varepsilon}) t}}\geq c_{0},\tau\geq t_{0}\right)
\\
&\leq
\mathbb{P}\left(\sup_{t_{0}\leq t\leq t_{1}}\Vert Q_{t_{0}}(t_{1})Z_{t}^{0}\Vert
\geq(\varepsilon+re^{-\sqrt{m}(1-\hat{\varepsilon}) t_{1}})c_{0}\right)
\\
&
\leq
2^{d}
\cdot
\exp\left(-\frac{\beta\gamma^{-1}C_{\hat{\varepsilon}}^{-2}\sqrt{m}(1-\hat{\varepsilon})}{4}[h^{2}-\langle\mu_{t_{1}},(I-\beta\theta\Sigma_{t_{1}})^{-1}\mu_{t_{1}}\rangle]\right)
\\
&\leq
2^{d}\cdot\exp\left(-\frac{\beta\gamma^{-1}C_{\hat{\varepsilon}}^{-2}\sqrt{m}(1-\hat{\varepsilon})\varepsilon^{2}}{4}
\left(c_{0}^{2}-\frac{1}{32}\right)\right)
\leq
2^{d}\cdot\exp\left(-\frac{\beta\gamma^{-1}C_{\hat{\varepsilon}}^{-2}\sqrt{m}(1-\hat{\varepsilon})\varepsilon^{2}}{128}\right).
\end{align*}
Thus for any $t_{0}\geq\mathcal{T}_{\text{rec}}^{U}$ and $t_{0}\leq t_{1}\leq t_{0}+\frac{1}{2\Vert H_{\gamma}\Vert}$,  
\begin{equation*}
\mathbb{P}(\tau\in[t_{0},t_{1}])
\leq
2^{d}\cdot\exp\left(-\frac{\beta\gamma^{-1}C_{\hat{\varepsilon}}^{-2}\sqrt{m}(1-\hat{\varepsilon})\varepsilon^{2}}{128}
\right).
\end{equation*}

Fix any $\mathcal{T}>0$ and recall the definition of the escape time $\mathcal{T}_{\text{esc}}^{U}=\mathcal{T}+\mathcal{T}_{\text{rec}}^{U}$.
Partition the interval $[\mathcal{T}_{\text{rec}}^{U},\mathcal{T}_{\text{esc}}^{U}]$ using the points
$\mathcal{T}_{\text{rec}}^{U}=t_{0}<t_{1}<\cdots<t_{\lceil 2\Vert H_{\gamma}\Vert \mathcal{T}\rceil}=\mathcal{T}_{\text{esc}}^{U}$
with $t_{j}=j/(2\Vert H_{\gamma}\Vert)$, then we have
\begin{align*}
&\mathbb{P}\left(\tau\in\left[\mathcal{T}_{\text{rec}}^{U},\mathcal{T}_{\text{esc}}^{U}\right]\right)
\\
&=\sum_{j=0}^{\lceil 2\Vert H_{\gamma}\Vert \mathcal{T}\rceil}\mathbb{P}(\tau\in[t_{j},t_{j+1}])
\leq(2\Vert H_{\gamma}\Vert \mathcal{T}+1)\cdot 2^{d}
\cdot\exp\left(-\frac{\beta\gamma^{-1}C_{\hat{\varepsilon}}^{-2}\sqrt{m}(1-\hat{\varepsilon})\varepsilon^{2}}{128}\right)\leq\delta,
\end{align*}
provided that
\begin{equation*}
\beta\geq\frac{128\gamma C_{\hat{\varepsilon}}^{2}}{\sqrt{m}(1-\hat{\varepsilon})\varepsilon^{2}}\left(d\log(2)+\log\left(\frac{2\Vert H_{\gamma}\Vert \mathcal{T}+1}{\delta}\right)\right).
\end{equation*}
The proof is complete.
\end{proof}

%%%%%%%%%%%%%%%%%%%%%%%%%%%%%%%%%%%%%%%%%%

\section{Generalization to population risk}\label{sec:population}

In this section, we apply Theorem \ref{MainThm1}, Theorem \ref{MainThm1:small} 
and Theorem \ref{MainThm1:Q} to study the population risk.
We recall that the population risk is denoted by $\overline{F}$,
and the empirical risk is denoted by $F$.
First, we need the assumption
that the population risk $\overline{F}$ is $(2\varepsilon_{0},2m)$-strongly Morse (see e.g. \cite{MBM}),
that is,
$\left\Vert\nabla \overline{F}(x)\right\Vert\leq 2\varepsilon_{0}$
implies
$\min_{j\in[d]}\left|\lambda_{j}(\nabla^{2}\overline{F}(x))\right|\geq 2m$.

%%%%%%%%%%%%%%%%%%%%%%%%%%%%%%%%
\subsection{Underdamped Langevin dynamics}

\begin{theorem}\label{MainThm2}
Suppose the assumptions in Theorem \ref{MainThm1} holds for $\gamma=2\sqrt{m}$ case
and the assumptions in Theorem \ref{MainThm1:small} holds for $\gamma<2\sqrt{m}$ case.
Assume that $\frac{n}{\log n}\geq\frac{c\sigma_{0}^{2}d}{(\varepsilon_{0}\wedge m)^{2}}$
and $\varepsilon\leq\frac{3m}{(1+\frac{\sqrt{m}}{8\sqrt{(C_{H}+2)m+(m+1)^{2}}})L}$ 
for the case $\gamma=2\sqrt{m}$
and $\varepsilon\leq\frac{3m}{(1+\frac{1}{8C_{\hat{\varepsilon}}})L}$ for the case $\gamma<2\sqrt{m}$.
With probability at least $1-\delta$, w.r.t. both the training data $z$ and the Gaussian noise, 
for any local minimum $x_{\ast}^{z}$ of $F$
\footnote{With the notation $x_{\ast}^{z}$ emphasizing the dependence on the training
data $z$}, either $\Vert X_{k}-x_{\ast}^{z}\Vert\geq\varepsilon/2$
for some $k\leq\lceil\eta^{-1}\mathcal{T}_{\text{esc}}^{U}\rceil$ or
\begin{equation*}
\overline{F}(x_{\ast}^{z})\leq\min_{\eta^{-1}\mathcal{T}_{\text{rec}}^{U}\leq k\leq\eta^{-1}\mathcal{T}_{\text{esc}}^{U}}F(X_{k})
+\sigma_{0}\sqrt{\frac{cd\log n}{n}},
\end{equation*}
where
\begin{align}
&c:=c_{0}(1\vee\log((M\vee L\vee(B+MR))R\sigma_{0}/\delta)),\label{eqn:c}
\\
&\sigma_{0}:=(A+(B+MR)R)\vee(B+MR)\vee(C+LR).\label{eqn:sigma}
\end{align}
\end{theorem}

\begin{proof}
Let us first assume that $\gamma=2\sqrt{m}$. 
Let $x_{\ast}^{z}$ be a local minimum of the empirical risk $F$.
By Lemma \ref{lem:eigen}, all eigenvalues of
of the Hessian $H=\nabla^{2}F(x_{\ast}^{z})$ are at least $m$
and therefore the norm $\Vert\cdot\Vert_{H}=\Vert H^{1/2}\cdot\Vert$
is well defined and
$\Vert\cdot\Vert_{H}\geq\sqrt{m}\Vert\cdot\Vert$.

We can decompose (letting $K_{0}=\lceil\eta^{-1}\mathcal{T}_{\text{rec}}^{U}\rceil$
and $K=\lceil\eta^{-1}\mathcal{T}_{\text{esc}}^{U}\rceil$)
\begin{equation*}
\overline{F}(x_{\ast}^{z})-\min_{K_{0}\leq k\leq K}F(X_{k})
=\left(\overline{F}(x_{\ast}^{z})-F(x_{\ast}^{z})\right)
+\left(F(x_{\ast}^{z})-\min_{K_{0}\leq k\leq K}F(X_{k})\right).
\end{equation*}
From Lemma~\ref{lem:unif-deviation}, we know with probability at least $1- \delta$,
\begin{equation*}
\overline{F}(x_{\ast}^{z})-F(x_{\ast}^{z})
\leq\sigma_{0}\sqrt{(cd/n)\log n}.
\end{equation*}
In addition, we can infer from \eqref{eq:hessian-L} that  for any $x$,
\begin{equation*}
\left|F(x)-F(x_{\ast}^{z})-\frac{1}{2}\Vert x-x_{\ast}^{z}\Vert_{H}^{2}\right|
\leq\frac{L}{6}\Vert x-x_{\ast}^{z}\Vert^{3}, 
\end{equation*}
Hence, we obtain
%\footnote{In \cite{pmlr-v75-tzen18a}, they used the range $K_{1}\leq k\leq K$, 
%and that is a typo and should be $\lceil\eta^{-1}\mathcal{T}_{\text{rec}}^{U}\rceil\leq k\leq K$.}
\begin{align*}
F(x_{\ast}^{z})-\min_{K_{0}\leq k\leq K}F(X_{k})
&=\max_{K_{0}\leq k\leq K}(F(x_{\ast}^{z})-F(X_{k}))
\\
&\leq\max_{K_{0}\leq k\leq K}
\left(\frac{L}{6}\Vert X_{k}-x_{\ast}^{z}\Vert^{3}-\frac{1}{2}\Vert X_{k}-x_{\ast}^{z}\Vert_{H}^{2}\right)
\\
&\leq\max_{K_{0}\leq k\leq K}
\left(\frac{L}{6}\Vert X_{k}-x_{\ast}^{z}\Vert^{3}-\frac{m}{2}\Vert X_{k}-x_{\ast}^{z}\Vert^{2}\right).
\end{align*}
By Theorem \ref{MainThm1}, with probability $1-\delta$,
either $\Vert X_{k}-x_{\ast}^{z}\Vert\geq\varepsilon/2$
for some $k\leq \lceil\eta^{-1}\mathcal{T}_{\text{rec}}^{U}\rceil$ 
or 
\begin{align}
\Vert X_{k}-x_{\ast}^{z}\Vert
\leq\varepsilon+re^{-\sqrt{m}\mathcal{T}_{\text{rec}}^{U}}
&=\varepsilon+\frac{\varepsilon^{2}}{\mathcal{T}_{\text{rec}}^{U}8r\sqrt{C_{H}+2+(m+1)^{2}}}
\nonumber
\\
&\leq\varepsilon+\frac{\varepsilon\sqrt{m}}{8\sqrt{(C_{H}+2)m+(m+1)^{2}}},\label{eqn:2e}
\end{align}
for all $K_{0}\leq k\leq K$, 
where we used the definition of $\mathcal{T}_{\text{rec}}^{U}$ for the case $\gamma=2\sqrt{m}$ in \eqref{rec:time:U}, and the assumption
$\varepsilon<\sqrt{\frac{C_{H}+2+(m+1)^{2}}{(C_{H}+2)m+(m+1)^{2}}}r$ and
the property $\mathcal{T}_{\text{rec}}^{U}>\frac{1}{\sqrt{m}}$. 
If the latter occurs, then by \eqref{eqn:2e}, we have
\begin{align*}
&F(x_{\ast}^{z})-\min_{K_{0}\leq k\leq K}F(X_{k})
\\
&\leq\max_{K_{0}\leq k\leq K}
\left(\frac{L}{6}\Vert X_{k}-x_{\ast}^{z}\Vert^{3}-\frac{m}{2}\Vert X_{k}-x_{\ast}^{z}\Vert^{2}\right)
\\
&\leq\max_{K_{0}\leq k\leq K}
\Vert X_{k}-x_{\ast}^{z}\Vert^{2}\left(\frac{L}{6}\cdot\left(1+\frac{\sqrt{m}}{8\sqrt{(C_{H}+2)m+(m+1)^{2}}}\right)\varepsilon-\frac{m}{2}\right)
\leq 0\,,
\end{align*}
for $\varepsilon\leq\frac{3m}{(1+\frac{\sqrt{m}}{8\sqrt{(C_{H}+2)m+(m+1)^{2}}})L}$.  
The proof for the case $\gamma=2\sqrt{m}$ is therefore complete. 

The proof for the case $\gamma<2\sqrt{m}$ is similar.
The only difference is that we replace \eqref{eqn:2e} by
the following estimate
\begin{equation*}
\Vert X_{k}-x_{\ast}^{z}\Vert
\leq\varepsilon+re^{-\sqrt{m}(1-\hat{\varepsilon})\mathcal{T}_{\text{rec}}^{U}}
\leq\varepsilon+re^{-\frac{1}{2}\sqrt{m}(1-\hat{\varepsilon})\mathcal{T}_{\text{rec}}^{U}}
=\varepsilon+\frac{\varepsilon}{8C_{\hat{\varepsilon}}},
\end{equation*}
for all $K_{0}\leq k\leq K$, 
where we used the definition of $\mathcal{T}_{\text{rec}}^{U}$ 
for the $\gamma<2\sqrt{m}$ case in \eqref{rec:time:U:small}. 
\end{proof}

%%%%%%%%%%%%%%%%%%%%%%%%%%%%%%%%%%%%%%%%%%%%
\subsection{Non-reversible Langevin dynamics}

\begin{theorem}\label{MainThm2:Q}
Suppose the assumptions in Theorem \ref{MainThm1:Q} holds.
Assume that $\frac{n}{\log n}\geq\frac{c\sigma_{0}^{2}d}{(\varepsilon_{0}\wedge m)^{2}}$
and $\varepsilon\leq\frac{3m}{(1+\frac{1}{8C_{J}(\tilde{\varepsilon})})L}$.
With probability at least $1-\delta$, w.r.t. both the training data $Z$ and the Gaussian noise, 
for any local minimum $x_{\ast}^{z}$ of $F$, either $\Vert X_{k}-x_{\ast}^{z}\Vert\geq\varepsilon/2$
for some $k\leq\eta^{-1}\mathcal{T}_{\text{rec}}^{J}$ or
\begin{equation*}
\overline{F}(x_{\ast}^{z})\leq\min_{\eta^{-1}\mathcal{T}_{\text{rec}}^{J}\leq k\leq\eta^{-1}\mathcal{T}_{\text{esc}}^{J}}F(X_{k})
+\sigma_{0}\sqrt{\frac{cd\log n}{n}},
\end{equation*}
where $c$ and $\sigma_{0}$ are defined in \eqref{eqn:c} and \eqref{eqn:sigma}.
\end{theorem}

\begin{proof}
The proof is similar to Theorem \ref{MainThm2} and Theorem 3 in \cite{pmlr-v75-tzen18a}.
The only difference is that we replace \eqref{eqn:2e} by
the following estimate
\begin{equation*}
\Vert X_{k}-x_{\ast}^{z}\Vert
\leq\varepsilon+re^{-m_{J}(\tilde{\varepsilon})\mathcal{T}_{\text{rec}}^{J}}
\leq\varepsilon+re^{-\frac{1}{2}m_{J}(\tilde{\varepsilon})\mathcal{T}_{\text{rec}}^{J}}
=\varepsilon+\frac{\varepsilon}{8C_{J}(\tilde{\varepsilon})},
\end{equation*}
for all $\lceil\eta^{-1}\mathcal{T}_{\text{rec}}^{J}\rceil\leq k
\leq\lceil\eta^{-1}\mathcal{T}_{\text{esc}}^{J}\rceil$, where we used
the definition of $\mathcal{T}_{\text{rec}}^{J}$. 
\end{proof}

%\end{comment}
%%%%%%%%%%%%%%%%%%%%%%%%%%%%%%%%%%%%%%%%%%
%%%%%%%%%%%%%%%%%%%%%%%%%%%%%%%%%%%%%%%%%%%
\section{Supporting technical lemmas}
\begin{lemma}\label{lem:H:gamma} 
Consider the square matrix $H_\gamma$ defined by \eqref{def-Hgamma0}. We have 
    $$\| H_\gamma \| \leq \sqrt{\gamma^{2} + M^{2} + 1}. $$
\end{lemma}

\begin{proof} 
It follows from \eqref{eq-perm-diagonalize} that 
   \beq \| H_\gamma\| = \| T_\gamma \|  = \max_i \|T_i(\gamma)\|.
    \label{eq-H-gamma-norm}
   \eeq
We also compute 
\begin{equation*}
     \|T_i(\gamma)\|^2 = \lambda_{\max} \left( T_i(\gamma) T_i(\gamma)^T \right) = \lambda_{\max}\left( 
     \begin{bmatrix} \gamma^2 +\lambda_i^2 & -\gamma \\
                     -\gamma               &    1
     \end{bmatrix}
     \right),
\end{equation*}
where $\lambda_{\max}$ denotes the largest real part of the eigenvalues. This leads to 
\begin{equation*}
     \|T_i(\gamma)\|^2 = \frac{\gamma^2 +\lambda_i^2 +1 + \sqrt{(\gamma^2 +\lambda_i^2 +1)^2 - 4\lambda_i^2 } }{2} \leq \gamma^2 +\lambda_i^2 +1.
\end{equation*}
Since $m\leq \lambda_i \leq M$ for every $i$, we obtain
\begin{equation*}
     \max_i \|T_i(\gamma)\|^2 \leq \max_i \left( \gamma^2 +\lambda_i^2 +1 \right) = \gamma^2 +M^2 +1.
\end{equation*}
We conclude from \eqref{eq-H-gamma-norm}.
\end{proof}

\begin{lemma}\label{lem:Brownian}
Let $B_{t}$ be a standard $d$-dimensional Brownian motion.
For any $u>0$ and any $t_{1}>t_{0}\geq 0$ with $t_{1}-t_{0}=\eta>0$, we have
\begin{equation*}
\mathbb{P}\left(\sup_{t\in[t_{0},t_{1}]}\Vert B_{t}-B_{t_{1}}\Vert\geq u\right)
\leq
2^{1/4}e^{1/4}e^{-\frac{u^{2}}{4d\eta}}.  
\end{equation*}
\end{lemma}

\begin{proof}
Also, by the time reversibility, stationarity of time increments of Brownian motion
and Doob's martingale inequality, for any $\theta>0$ so that $2\theta\eta<1$, we have
\begin{align*}
\mathbb{P}\left(\sup_{t\in[t_{0},t_{1}]}\Vert B_{t}-B_{t_{1}}\Vert\geq u\right)
&=\mathbb{P}\left(\sup_{t\in[0,\eta]}\Vert B_{t}-B_{0}\Vert\geq u\right)
\\
&\leq
e^{-\theta u^{2}}\mathbb{E}\left[e^{\theta\Vert B_{\eta}-B_{0}\Vert^{2}}\right]
\\
&=e^{-\theta u^{2}}(1-2\theta\eta)^{-d/2}.
\end{align*}
By choosing $\theta=1/(4d\eta)$, we get
\begin{equation*}
\mathbb{P}\left(\sup_{t\in[t_{0},t_{1}]}\Vert B_{t}-B_{t_{1}}\Vert\geq u\right)
\leq
\left(1-\frac{1}{2d}\right)^{-\frac{d}{2}}e^{-\frac{u^{2}}{4d\eta}}.
\end{equation*}
Note that for any $x>0$, $(1+\frac{1}{x})^{x}<e$. 
Let us define $x>0$ via
\begin{equation*}
1-\frac{1}{2d}=\frac{1}{1+x}.
\end{equation*}
Then, we get $d=\frac{1+x}{2x}$ and $x=\frac{1}{1-\frac{1}{2d}}-1\leq 1$, and
\begin{equation*}
\left(1-\frac{1}{2d}\right)^{-\frac{d}{2}}
=\left(\frac{1}{1+x}\right)^{-\frac{1+x}{4x}}
=(1+x)^{\frac{1}{4}}(1+x)^{\frac{1}{4x}}
\leq 2^{1/4}e^{1/4}.
\end{equation*}
Hence,
\begin{equation*}
\mathbb{P}\left(\sup_{t\in[t_{0},t_{1}]}\Vert B_{t}-B_{t_{1}}\Vert\geq u\right)
\leq
2^{1/4}e^{1/4}e^{-\frac{u^{2}}{4d\eta}}. 
\end{equation*}
\end{proof}

\begin{lemma}[See Lemma 2 in \cite{Raginsky}] \label{lem:gradient-bound}
If parts $(i)$ and $(ii)$ of Assumption \ref{assumptions} hold, 
then for all $x\in\mathbb{R}^{d}$ and $z\in\mathcal{Z}$,
\begin{equation*}
\Vert\nabla f(x,z)\Vert
\leq M\Vert x\Vert+B,
\end{equation*}
and
\begin{equation*}
\frac{m}{3}\Vert x\Vert^{2}
-\frac{b}{2}\log 3
\leq f(x,z)
\leq\frac{M}{2}\Vert x\Vert^{2}+B\Vert x\Vert+ A.
\end{equation*}
\end{lemma}

%%%%%%%%%%%%%%%%%%%%%%%%%%%%%%%%%%%%%%%%
\begin{lemma}[Lemma 6 in \cite{pmlr-v75-tzen18a}, Uniform Deviation Guarantees] \label{lem:unif-deviation}
Under $(i)$ and $(ii)$ of Assumption \ref{assumptions},
there exists an absolute constant $c_{0}$ such
that for:
\begin{align*}
&c:=c_{0}(1\vee\log((M\vee L\vee(B+MR))R\sigma_{0}/\delta)),
\\
&\sigma_{0}:=(A+(B+MR)R)\vee(B+MR)\vee(C+LR),
\end{align*}
we have, if $n\geq cd\log d$, then with probability at least $1-\delta$:
\begin{align*}
&\sup_{x\in\mathbb{R}^{d}}\left|F(x)-\overline{F}(x)\right|\leq\sigma_{0}\sqrt{\frac{cd\log n}{n}},
\\
&\sup_{x\in\mathbb{R}^{d}}\left\Vert\nabla F(x)-\nabla \overline{F}(x)\right\Vert\leq\sigma_{0}\sqrt{\frac{cd\log n}{n}},
\\
&\sup_{x\in\mathbb{R}^{d}}\left\Vert\nabla^{2}F(x)-\nabla^{2}\overline{F}(x)\right\Vert
\leq\sigma_{0}\sqrt{\frac{cd\log n}{n}}.
\end{align*}
\end{lemma}

\begin{lemma}[Proposition 7 in \cite{pmlr-v75-tzen18a}]\label{lem:eigen}
If the population risk $\overline{F}(x)$ is $(2\varepsilon_{0},2m)$-strongly Morse,
then provided that $n\geq cd\log d$ and $\frac{n}{d\log n}\geq\frac{c\sigma_{0}^{2}}{(\varepsilon_{0}\wedge m)^{2}}$,
the empirical risk $F(x)$ is $(\varepsilon_{0},m)$-strongly Morse with probability 
at least $1-\delta$.
\end{lemma}

\end{document}